\documentclass[11pt,a4paper,twoside]{article}
\usepackage{amssymb}
\usepackage{amsfonts}
\usepackage{geometry}
\usepackage{color}
\usepackage{mathtools}
\usepackage{amsmath,amsthm,amssymb,amsfonts}
\usepackage{CJK}
\usepackage{latexsym}
\usepackage{graphicx}
\usepackage{pgfpages}
\usepackage{mathrsfs}
\usepackage{ytableau}
\usepackage{enumerate}
\usepackage{comment}
\usepackage{cite}
\usepackage[title,titletoc]{appendix}
\usepackage[symbols,nogroupskip,sort=none]{glossaries-extra}
\usepackage[hyphens]{url}
\usepackage[bookmarks=false,hidelinks, colorlinks, citecolor=red,linkcolor=blue]{hyperref}
\usepackage[affil-it]{authblk}
\allowdisplaybreaks

\newtheorem{definition}{Definition}[section]
\newtheorem{theorem}[definition]{Theorem}
\newtheorem{lemma}[definition]{Lemma}
\newtheorem{proposition}[definition]{Proposition}
\newtheorem{corollary}[definition]{Corollary}
\newtheorem{remark}[definition]{Remark}
\newtheorem{condition}[definition]{Assumption}

\newtheorem{example}[definition]{Example}
\newtheorem{question}{Question}

\numberwithin{equation}
{section}
\def \a{\alpha}   \def \d{\delta}
\def \t{\theta} \def \T{\Theta}  \def \e{\epsilon}
\def \s{\sigma} \def \l{\lambda}  \def \o{\omega}
\def \O{\Omega}   
  \def \G{\Gamma}

\def \E{\mathbb{E}}

\def\deq{\stackrel{d}{=}}
\def\one{\mathbf 1}

\begin{document}
	\title{Ergodicity and Mixing of Sublinear Expectation System and Applications}
	\author [1] {Wen Huang}
    \author [2] {Chunlin Liu}
    \author [3] {Shige Peng}
	\author [4] {Baoyou Qu}
    \affil[1] {School of Mathematical Sciences, University of Science and Technology of China, Hefei, Anhui, 230026, P.R. China}
     \affil[2] {School of Mathematical Sciences, Dalian University of Technology, Dalian, 116024, P.R. China}
    \affil[3] {School of Mathematics, Zhongtai Securities Institute for Financial Studies, Shandong University, Jinan, 250100, P.R. China}
	\affil[4] {Department of Mathematical Sciences, Durham University, Durham, DH1 3LE, UK}
 \affil[ ]{wenh@mail.ustc.edu.cn, chunlinliu@mail.ustc.edu.cn, peng@sdu.edu.cn, baoyou.qu@durham.ac.uk}
	\date{}
	
	\maketitle
	
	\begin{abstract}
We utilize an ergodic theory framework to explore sublinear expectation theory.  Specifically, we investigate the pointwise Birkhoff's ergodic theorem for invariant sublinear expectation systems.  By further assuming that these sublinear expectation systems are ergodic, we derive stronger results.  Furthermore, we relax the conditions for the law of large numbers and the strong law of large numbers under sublinear expectations from independent and identical distribution to $\alpha$-mixing. These results can be applied to a class of stochastic differential equations driven by $G$-Brownian motion (i.e., $G$-SDEs), such as $G$-Ornstein-Uhlenbeck processes.

As  byproducts, we also obtain a series of applications for classical ergodic theory and capacity theory.

\noindent
{\bf MSC2020 subject classifications:} Primary 37A25, 60G65; secondary 28A12, 60F17

		
\end{abstract}
	
	{
\hypersetup{linkcolor=black}
\tableofcontents
}

	\section{Introduction}
The laws of large numbers (LLN) are used in various fields, including statistics, probability theory, economics, and insurance, establishing that empirical averages converge to expectations. Under suitable conditions, these laws ensure that as data accumulates, inferences about consistent phenomena become increasingly accurate. The origins of the laws of large numbers trace back to  Bernoulli \cite{Bernoulli1713}. The combined work of these mathematicians resulted in the creation of two main types of the laws of large numbers: the weak law, which establishes conditions for convergence in distribution, and the strong law, which establishes conditions for almost sure convergence of empirical averages. These developments have been foundational in the progress of statistical inference and probability theory (see Seneta \cite{Eugene2013} for an overview of the historical development).

However, while probability theory is a crucial tool for measuring uncertain quantities, it often falls short in real-world scenarios where exact probabilities are difficult to determine. Often, the uncertainty of the probability itself poses significant challenges. For example, in the field of economics, this higher level of uncertainty is referred to as Knight uncertainty \cite{Knight1921}.

One significant concept in addressing this type of uncertainty is sublinear expectation, which was inspired by the challenges of probability and statistics under uncertainty, as well as risk measures and super-hedging in finance (cf. El Karoui, Peng and Quenez \cite{ElKarouiPengQuenez1997}; Artzner, Delbaen, Eber and Heath \cite{ArtznerDelbaenEberHeath2002}; Chen and Epstein \cite{ChenEpstein2002}; Föllmer and Schied \cite{FllmerSchied2011}). Compared to the well-accepted classical notion of capacity (or called non-additive probability), the systematic stochastic analysis of sublinear expectation and $G$-Brownian motion, as extensively developed in the third author's substantial work, provides a framework for studying uncertainty \cite{Peng2005,Pengbook,Peng2007,Peng2004Fil,Peng2008LLN,Peng2008multi}. In the framework of $G$-expectation, Hu, Ji, Peng, and Song \cite{Hu-Ji-Peng-Song2014a,Hu-Ji-Peng-Song2014b} also studied the well-posedness of backward stochastic differential equations driven by $G$-Brownian motion (G-BSDEs) and their corresponding Feynman-Kac formulas for fully nonlinear partial differential equations (PDEs).

Unlike the extensive studies on linear cases, the investigation of the dynamical properties and long-term behavior of $G$-diffusion processes has been relatively limited (see Feng and Zhao \cite{FengZhao2021} from  a dynamical system viewpoint and Hu, Li, Wang and Zheng \cite{Hu-Li-Wang-Zheng2015} from an equation viewpoint).

 In this paper, we further explore sublinear expectation systems and introduce the concept of mixing for these systems. Specifically, we begin by establishing some essential notations.

Let $\Omega$ be a nonempty given set and $\mathcal{H}$ denote a linear space of real-valued functions defined on $\Omega$. We refer to $\mathcal{H}$ as a vector lattice if it satisfies two conditions: (1) for any constant $c \in \mathbb{R}$, the function $c$ is in $\mathcal{H}$, and (2) for any $X \in \mathcal{H}$, the function $|X|$ is also in $\mathcal{H}$. Throughout this paper, we will assume that $\mathcal{H}$ is always a vector lattice. It is straightforward to verify that for any $X, Y \in \mathcal{H}$, both $X \vee Y := \max\{X, Y\}$ and $X \wedge Y := \min\{X, Y\}$ are elements of $\mathcal{H}$. Additionally, we denote $X^+ = X \vee 0$ and $X^- = - (X \wedge 0)$.

We also assume that for any $X_1, X_2, \dots, X_n \in \mathcal{H}$, the function $\phi(X_1, X_2, \dots, X_n)$ belongs to $\mathcal{H}$ for any $\phi \in C_{lip}(\mathbb{R}^n)$, where $C_{lip}(\mathbb{R}^n)$ represents the space of Lipschitz continuous functions. This assumption holds throughout the paper, except in Theorem \ref{thm:main1} in the introduction, and in Sections \ref{sec:ergodic} and \ref{sec:applications}.

	\begin{definition}
		A sublinear expectation $\hat{\mathbb{E}}$ on  $\mathcal{H}$ is a functional $\hat{\mathbb{E}}: \mathcal{H}\rightarrow \mathbb{R}$ satisfying the following properties: for all $X,Y\in \mathcal{H}$, we have
		\begin{description}
			\item[(a)] Monotonicity: $\hat{\mathbb{E}}[X]\leq \hat{\mathbb{E}}[Y]$ if $X\leq Y$.
			\item[(b)]Constant preserving: $\hat{\mathbb{E}}[c]=c$ for $c\in\mathbb{R}$.
			\item[(c)]Sub-additivity: $\hat{\mathbb{E}}[X+Y]\leq \hat{\mathbb{E}}[X]+\hat{\mathbb{E}}[Y]$.
			\item[(d)]Positive homogeneity: $\hat{\mathbb{E}}[\lambda X]=\lambda\hat{\mathbb{E}}[X]$ for
			$\lambda\geq 0$.
		\end{description}
  If (c) and (d) is replaced by the the following:
  	\begin{description}
			\item[(e)] Linearity:  $\hat{\mathbb{E}}[\l_1 X+\l_2 Y]= \l_1\hat{\mathbb{E}}[X]+\l_2\hat{\mathbb{E}}[Y]$ for $\l_1,\l_2\in\mathbb{R}$,
		\end{description}
  then $\hat{\mathbb{E}}$ is called a linear expectation.
	\end{definition}
In the following, we call the triple $(\Omega, \mathcal{H}, \hat{\mathbb{E}})$  a sublinear expectation space and denote by $\sigma(\mathcal{H})$  the $\sigma$-algebra generated by $\mathcal{H}$.  It follows from \cite[Theorem 1.2.1]{Pengbook} that the   convex set
\begin{equation}\label{eq:set Theta}
    \T:=\{\mathbb{E}:\mathbb{E}\text{ is a linear expectation on $\mathcal{H}$ with }\mathbb{E}[X]\le \hat{\mathbb{E}}[X]\text{ for any }X\in\mathcal{H}\}
\end{equation}
is nonempty, and it satisfies 
the following property:
\begin{equation}\label{eq:max-in-Theta}
    \hat{\mathbb{E}}[X]=\max_{\mathbb{E}\in\T}\mathbb{E}[X]\text{ for any }X\in\mathcal{H}.
\end{equation}

In this context, the third author \cite{Peng2019b} established the following LLN within the sublinear expectation framework:

\emph{Let $\{X_k\}_{k \geq 1}$ be a sequence of independent and identically distributed (i.i.d.) $1$-dimensional random variables on $\mathcal{H}$ under a sublinear expectation $\hat{\mathbb{E}}$ with $\hat{\mathbb{E}}[|X_1|^2]<\infty$. Then for any $\phi \in C_{lip}(\mathbb{R})$, 
$$
\lim _{n\to\infty} \hat{\mathbb{E}}\left[\phi\left(\frac{1}{n}\sum_{k=1}^nX_k\right)\right]=\max_{y \in[\underline{\mu}, \bar{\mu}]} \phi(y),
$$
where
 $\underline{\mu}=-\hat{\mathbb{E}}\left[-X_1\right]$, and $\bar{\mu}=\hat{\mathbb{E}}\left[X_1\right]$.}

Subsequently, under the same conditions, Song provided error estimates for Peng’s LLN using Stein’s method \cite{Song2021}. We refer to \cite{Chen2016,FANGPENGSHAOSONG2019,Zhang2023} for more work on LLN under the sublinear expectation. However, many processes (see Section \ref{subsec:application to GSDE} for examples induced by 
$G$-SDEs) are not i.i.d., meaning the above LLN cannot be directly applied.

Motivated by the mixing condition in linear case (see e.g. \cite{Bradley2007i-iii,Doukhan1994}), we introduce the definition of $\alpha$-mixing for a sequence of random variables on sublinear expectation spaces, which is a kind of ``asymptotic independence".

\begin{definition}\label{Def:alpha-mixing}
    A sequence $\{X_k\}_{k\geq 1}$ of $d$-dimensional random vectors on a sublinear expectation space $(\Omega, \mathcal{H}, \hat{\mathbb{E}})$ is called $\alpha$-mixing if there exist $c_X>0,\alpha>0$ such that for any non-empty finite subset $\Lambda^1,\Lambda^2\subset \mathbb{N}$ with $\Lambda^1_{max}:=\max_{u_1\in \Lambda^1}u_1\leq \min_{u_2\in \Lambda^2}u_2=:\Lambda^2_{min}$ and $\phi\in C_{lip}(\mathbb{R}^{2d})$ with Lipschitz constant $l_{\phi}>0$,
		\begin{equation*}
			\begin{split}
				|\hat{\mathbb{E}}[\phi(\bar{X}_{\Lambda^1}, \bar{X}_{\Lambda^2})] - \hat{\mathbb{E}}[\hat{\mathbb{E}}[\phi(x, \bar{X}_{\Lambda^2})]|_{x=\bar{X}_{\Lambda^1}}]| \leq c_X l_{\phi} e^{-\alpha |\Lambda^2_{min}-\Lambda^1_{max}|},
			\end{split}
		\end{equation*}
		where $\bar{X}_{\Lambda^i}=\frac{1}{|\Lambda^i|}\sum_{u_i\in \Lambda^i}X_{u_i}$  and $|\Lambda^i|$ is the number of elements in $\Lambda^i$, $i=1,2$. 
\end{definition}
\begin{remark}
      It is obvious that an  i.i.d. sequence must be an $\a$-mixing stationary sequence.   
      
     However, the converse is not true. For instance, consider an i.i.d. sequence $\{\xi_n\}_{n \geq 1}$ defined on a sublinear expectation space $(\Omega, \mathcal{H}, \hat{\mathbb{E}})$. Define $X_n := \xi_n + \xi_{n+1}$ for all $n \geq 1$. It is straightforward to verify that the sequence $\{X_n\}_{n \geq 1}$ is $\alpha$-mixing for any $\alpha > 0$, yet it is not i.i.d.

More generally, there exists an $\a$-mixing sequence $\{X_t\}_{t\in\mathbb{R}}$ induced by a class of nonlinear SDEs in which $X_t$ is not independent from $X_s$ for any $t>s$ (see Theorem \ref{Thm:mixing of G-SDE} and Theorem \ref{Thm:not independent}). 
\end{remark}

 Under this weaker condition, we also can obtain the LLN as follows.
\begin{theorem}\label{thm:LLN}
     Suppose that $\{X_k\}_{k\geq 1}$ is a d-dimensional $\alpha$-mixing stationary sequence on a sublinear expectation space $(\Omega,\mathcal{H},\hat{\mathbb{E}})$ for some $\alpha>0$ with $\hat{\mathbb{E}}[|X_1|^{2}]<\infty$. Let
	\begin{equation}\label{eq:Gamma*}
	    \Gamma_n:=\Big\{\frac{1}{n}\sum_{k=1}^n\mathbb{E}[X_k]: \mathbb{E}\in \Theta\Big\}\subset \mathbb{R}^d, \ \text{ and } \ \Gamma_*:=\bigcap_{n\geq 1}\Gamma_n.
	\end{equation}
         Then for any $\d>0$, there exists $C>0$ such that for any $\phi\in C_{lip}(\mathbb{R}^d)$ with Lipschitz constant $l_{\phi}>0$ and for all $n\in \mathbb{N}$,
 \begin{equation}\label{eq:thm-LLN}
	\left|\hat{\mathbb{E}}\left[\phi\left(\frac{1}{n}\sum_{k=1}^nX_k\right)\right]-\max_{x\in \Gamma_*}\phi(x) \right|\le C  l_{\phi}n^{-1/4+\delta}.
\end{equation}
 \end{theorem}

\begin{remark}\label{rem:smaller}
    (i) Fix $X\in\mathcal{H}^d$ for some $d\geq 1$. If $\hat{\E}[\phi(X)]=\max_{y\in \Gamma} \phi(y)$ for all $\phi\in C_{lip}(\mathbb{R})$ and for some bounded closed convex subset $\Gamma$ of $\mathbb{R}^d$, then $X$ is said to have a $\Gamma$-maximal distribution. Thus, \eqref{eq:thm-LLN} shows that the average of an $\alpha$-mixing sequence converges to a $\Gamma_*$-maximal distribution.
    
    (ii) If $\{X_k\}_{k\geq 1}$ is an i.i.d. sequence on a sublinear expectation space $(\O,\mathcal{H},\hat{\E})$ (so it is $\alpha$-mixing for all $\alpha>0$), it is easy to see that $\Gamma_n=\Gamma_1$ for all $n\geq 1$, and hence, $\Gamma_*=\Gamma_1$. However, it is possible that $\Gamma_*\subsetneq \Gamma_1$ in general. Indeed, let $\{\xi_k\}_{k\geq 1}$ be a 1-dimensional i.i.d. $[-1,1]$-maximally distributed sequence and $X_k:=-\xi_k(\xi_{k+1}+2)$ for all $k\geq 1$. As $X_{k+2}$ is independent of $(X_1,\ldots,X_k)$ for each $k\in\mathbb{N}$,  $\{X_k\}_{k\geq 1}$ is $\alpha$-mixing for all $\alpha>0$. Note that
    \begin{equation*}
        \hat{\E}[X_1]=\hat{\E}[-\xi_1(\xi_{2}+2)]=\hat{\E}[3\xi_1^--\xi_1^+]=3,
    \end{equation*}
    and
    \begin{equation*}
        -\hat{\E}[-X_1]=-\hat{\E}[\xi_1(\xi_{2}+2)]=-\hat{\E}[3\xi_1^+-\xi_1^-]=-3,
    \end{equation*}
    where $x^+:=x\vee 0=\max\{x,0\}$ and $x^-:=(-x)\vee 0=\max\{-x,0\}$. Then $\Gamma_1=[-3,3]$. However,
    \begin{equation*}
        \hat{\E}\Big[\frac{X_1+X_2}{2}\Big]=\frac{1}{2}\hat{\E}[-\xi_1(\xi_{2}+2)-\xi_2(\xi_{3}+2)]=\frac{1}{2}\hat{\E}[-\xi_1(\xi_{2}+2)+3\xi_2^--\xi_2^+]=\frac{1}{2}\hat{\E}[3-\xi_1]=2,
    \end{equation*}
    and
    \begin{equation*}
        -\hat{\E}\Big[-\frac{X_1+X_2}{2}\Big]=-\frac{1}{2}\hat{\E}[\xi_1(\xi_{2}+2)+\xi_2(\xi_{3}+2)]=-\frac{1}{2}\hat{\E}[\xi_1(\xi_{2}+2)+3\xi_2^+-\xi_2^-]=-\frac{1}{2}\hat{\E}[3\xi_1+3]=-3.
    \end{equation*}
    Hence $\Gamma_2=[-3,2]\subsetneq [-3,3]=\Gamma_1$. Therefore, $\Gamma_*\subset \Gamma_2\subsetneq\Gamma_1$.
\end{remark}

     An element $X\in \mathcal{H}$ is said to have no mean-uncertainty if $\hat{\mathbb E}[-X]=-\hat{\mathbb E}[X]$. In this case, the sets $\Gamma^*$ in Theorem \ref{thm:LLN} are the same singleton $\{\mathbb E[X]\}$ for all $n\geq 1$. By choosing $\phi(x)=|x-\mathbb E[X]|$, we have the following corollary.

 \begin{corollary}\label{lem-L^1 convergence under alpha}
     Assume that $\{X_n\}_{n=1}^\infty$ is a $d$-dimensional $\a$-mixing stationary sequence on a sublinear expectation space $(\O,\mathcal{H},\hat{\mathbb E})$ for some $\a>0$. If $X_1\in \mathcal{H}$ has no mean-uncertainty, then 
     \[\lim_{n\to\infty}\hat{\mathbb E}\Big[\Big|\frac{1}{n}\sum_{i=1}^nX_i-\hat{\mathbb E}[X_1]\Big|\Big]=0.\]
 \end{corollary}

Theorem \ref{thm:LLN} indicates that 
$\alpha$-mixing sequences satisfy the LLN. However, in many cases, the focus may not be on the state of the process at every moment but rather on specific points in time. Consequently, we wish to consider the statistical behavior of the process along certain subsequences. 
To address this question, we introduce the concept of admissible sequences (see Definition \ref{Def: admissible}) and demonstrate that the LLN holds along such sequences (see Theorem \ref{Thm: LLN for subsequence}). Specifically, we prove that the LLN remains valid for any subsequence defined by polynomials of degree greater than one.

	\begin{theorem}\label{thm:main2}
		Assume that  $\{X_k\}_{k\geq 1}$ is a $d$-dimensional $\alpha$-mixing identical sequence on a sublinear expectation space $(\Omega,\mathcal{H},\hat{\mathbb{E}})$ for some $\alpha>0$ with $\hat{\mathbb{E}}[|X_1|^{2}]<\infty$. Let $t_k=\lfloor P(k)\rfloor\vee 1$ for each $k\in\mathbb N$, where $P(x)=a_mx^m+\cdots+a_1x+a_0$ is a real-valued polynomial with $a_m>0$ for some $m\geq 2$, and 	$\Gamma:=\{\mathbb{E}[X_1]: \mathbb{E}\in \Theta\}.$  
		Then there exists $C>0$ such that for any $\phi\in C_{lip}(\mathbb{R}^d)$ with Lipschitz constant $l_{\phi}>0$ and for all $n\in \mathbb{N}$,
		\begin{equation*}
			\left|\hat{\mathbb{E}}\left[\phi\left(\frac{1}{n}\sum_{k=1}^n X_{t_k}\right)\right]-\max_{x\in \Gamma}\phi(x) \right|\leq C  l_{\phi}n^{-1/2}.
		\end{equation*}
    \end{theorem}

  \begin{remark}
The convergence rate $n^{-\frac{1}{2}}$ is consistent with that of the i.i.d. scenario (see Song \cite{Song2021}). For the linear case, references such as \cite{LeonardMelvin1965} or standard texts like \cite{Durrettbook} provide further details on the convergence rate.
  
  Moreover, our LLN  is only under $\a$-mixing condition, which is applied to sublinear expectation systems generated by the solution flow of $G$-SDEs (see Theorem \ref{Thm: alpha-mixing GSDE}).
	\end{remark}

 The strong law of large numbers (SLLN) is also a fundamental concept in probability theory. Recently, the SLLN for capacities or sublinear expectations has emerged as a challenging topic that has garnered increasing interest. In the context of capacities, several notable works have been published, including those by Chen, Wu, and Li \cite{ChenWuLi2013}, Epstein and Schneider \cite{EpsteinSchneider2003}, Marinacci \cite{Marinacci1999}, and Maccheroni and Marinacci \cite{MaccheroniMarinacci2005}.
In the framework of sublinear expectations, the SLLN has been explored under independence conditions by Chen \cite{Chen2016}, Song \cite{Song2022}, and Zhang \cite{Zhang2022} in one-dimensional cases on regular sublinear expectation spaces (see Section \ref{sec:pre-sub} for the definition of regularity).

It is important to note that under the regularity condition (see Definition \ref{def:regular}), there exists a set $\mathcal{P}$ of probability measures on $(\Omega, \sigma(\mathcal{H}))$ such that

\begin{equation}\label{P of sub exp} \hat{\mathbb{E}}[X] = \sup_{P \in \mathcal{P}} \int X  dP \quad \text{for any } X \in \mathcal{H}. 
\end{equation}
Following ideas in \cite{Pengbook}, on a regular sublinear expectation space $(\O,\hat{\E},\mathcal{H})$, we define a capacity with respect to $\hat{\mathbb{E}}$ via
	\[C(A)=\sup_{P\in\mathcal{P}}P(A),\text{ for any }A\in\s(\mathcal{H}),\]
where $\mathcal{P}$ is from \eqref{P of sub exp}.

 A property holds “$\hat{\E}$-quasi surely”
($\hat{\E}$-q.s.) if it holds on a set $A\in\sigma(\mathcal{H})$ with $C(A^c)=0$, where $A^c:=\O\setminus A.$

In this paper, we will demonstrate that the SLLN also holds for all $\alpha$-mixing sequences.
 \begin{theorem}\label{Thm:SLLN}
     Suppose that $\{X_k\}_{k\geq 1}$ is a d-dimensional $\alpha$-mixing stationary sequence on a regular sublinear expectation space $(\Omega,\mathcal{H},\hat{\mathbb{E}})$ for some $\alpha>0$ with $\hat{\mathbb{E}}[|X_1|^{2}]<\infty$. Let
   $\Gamma_*$ be defined as in \eqref{eq:Gamma*}.
     Then 
     \begin{equation*}
         \lim_{n\to\infty}dist\left(\frac{1}{n}\sum_{k=1}^nX_k,\Gamma_*\right)=0, \ \text{$\hat{\mathbb{E}}$-q.s.,}
     \end{equation*}
     where $dist(x,\Gamma_*):=\inf_{y\in \Gamma_*}|x-y|$. In particular, this implies that all limit points of the sequence $\{\frac{1}{n}\sum_{k=1}^nX_k\}_{n=1}^\infty$ are contained within the set  $\Gamma_*$, $\hat\E$-q.s. For the case when 
 $d=1$,
     \begin{equation*}
       -\hat{\mathbb E}[-X_1] \le \inf_{x\in \Gamma_*}x\leq \liminf_{n\to\infty}\frac{1}{n}\sum_{k=1}^nX_k\leq \limsup_{n\to\infty}\frac{1}{n}\sum_{k=1}^nX_k\leq \sup_{x\in \Gamma_*}x\leq\hat{\mathbb E}[X_1], \ \text{$\hat\E$-q.s.} 
     \end{equation*}
 \end{theorem}

\begin{remark}
When the case
$\Gamma_*\subsetneq \Gamma_1$
  occurs (as the example in  Remark \ref{rem:smaller}), we can actually obtain more precise bound.   
\end{remark}

 Thus, as a corollary of Theorem \ref{Thm:SLLN}, we have the following result.
 \begin{corollary}
       Suppose that $\{X_k\}_{k\geq 1}$ is a d-dimensional $\alpha$-mixing stationary sequence on a regular sublinear expectation space $(\Omega,\mathcal{H},\hat{\mathbb{E}})$ for some $\alpha>0$ with $\hat{\mathbb{E}}[|X_1|^{2}]<\infty$. If $X_1$ has no mean-uncertainty, then 
        \begin{equation*}
         \lim_{n\to\infty}\frac{1}{n}\sum_{k=1}^nX_k=\hat{\mathbb E}[X_1], \ \hat{\mathbb{E}}\text{-q.s.}
     \end{equation*}
 \end{corollary}

However, in general, we cannot expect that for any $X\in\mathcal{H},$ the following equation holds (see \cite{Gilboa1987,MaccheroniMarinacci2005} for examples where such equations hold under very special assumptions)
 \begin{equation}\label{eq:13.38}
     \liminf_{n\to\infty}\frac{1}{n}\sum_{k=1}^nX_k=\limsup_{n\to\infty}\frac{1}{n}\sum_{k=1}^nX_k,\ \hat{\E}\text{-q.s.}
 \end{equation}
Therefore, there exists a natural question as follows.
\begin{question}\label{q1}
    What conditions for $X\in \mathcal{H}$ are necessary for the limit in  \eqref{eq:13.38} to exist?
\end{question}
To study this question, we employ the tools from ergodic theory, which is a mathematical branch that explores the long-term behavior of a dynamical system preserving a probability. Specifically, a measure-preserving system can be described by  $(\O,\mathcal{F},P,f)$, where  $(\O,\mathcal{F})$ is a measurable space,  $f:\O\to\O$ is a measurable transformation and $P$ is an invariant probability on $(\O,\mathcal{F})$, i.e., $P(f^{-1}A)=P(A)$ for any $A\in\mathcal{F}$.
  
A cornerstone of ergodic theory is Birkhoff's ergodic theorem \cite{Birkhoff1}, which has profound implications not only within ergodic theory itself, but also in related fields such as number theory (cf. Einsiedler and Ward \cite{Ward2011}; Furstenberg \cite{Furstenberg1981}), stationary processes (cf. Doob \cite{Doob1953}), and harmonic analysis (cf. Rosenblatt and Wierdl \cite{Rosenblatt1995}). 
Using Birkhoff's ergodic theorem, we derive a law of large numbers applicable to ergodic sequences (note that i.i.d. sequences are a special case of ergodic sequences) (see \cite[Theorem 2.1, Chapter X]{Doob1953}). This is why we use tools from ergodic theory to study Question \ref{q1}.

As research in ergodic theory has progressed, many generalizations and applications of Birkhoff's ergodic theorem have been developed. However, the majority of this research has been conducted within the framework of probability theory. 
This topic in nonlinear setting  has been studied in the ergodic framework from the capacity viewpoint by Cerreia-Vioglio, Maccheroni, and Marinacci \cite{CMM2016}; Feng, Wu, and Zhao \cite{FWZ2020}; and Huang, Feng, Liu, and Zhao \cite{FHLZ2023,FHLZ2024}. Additionally, it has been studied from the sublinear expectation viewpoint by Feng and Zhao \cite{FengZhao2021}. However, when described in the language of sublinear expectations, the above works can be viewed as results concerning regular sublinear expectations on the vector lattice of all bounded measurable functions. Consequently, these results cannot be directly applied to systems generated by general 
$G$-Brownian motion or the solutions of related nonlinear equations. Therefore, in this paper, we provide a condition to ensure the existence of limits within systems composed of arbitrary vector lattices and regular sublinear expectations defined on them.
Additionally, we  provide examples to demonstrate the necessity of this condition.

More precisely, for a given sublinear expectation space $(\Omega, \mathcal{H}, \hat{\mathbb{E}})$, if $(U_t)_{t\in \mathbb T}$, $\mathbb T=\mathbb R_+$ or $\mathbb N$ (resp. $\mathbb T=\mathbb R$ or $\mathbb Z$) is a semigroup (resp. group) of transformations on $\mathcal{H}$ preserving the sublinear expectation $\hat{\mathbb{E}}$, then
        we call the quadruple $(\Omega, \mathcal{H}, (U_t)_{t\in \mathbb T}, \hat{\mathbb{E}})$ a sublinear expectation system. If $\mathbb T=\mathbb Z$ or $\mathbb Z_+$, then $U_n=U_1^n$ for $n\in \mathbb T$, and we denote the sublinear expectation system by $(\Omega, \mathcal{H}, U_1, \hat{\mathbb{E}})$.
Furthermore, if $\mathbb T=\mathbb R$ or $\mathbb R_+$, and  
        \begin{equation}\label{eq:continuous}
            \lim\limits_{|t|\to 0,t\in\mathbb{T}}\hat{\mathbb{E}}[|U_tX-X|]=0, \ \text{ for all } \ X\in \mathcal{H},
        \end{equation}
the sublinear expectation system is said to be continuous.

Firstly, we present a sublinear version of Birkhoff's pointwise ergodic theorem. 
For convenience, in the introduction, we will state the results for $\mathbb T=\mathbb Z_+$ only; the corresponding results for other cases are detailed in Section \ref{sec:ergodic}. Let $(\O,\mathcal{H},U_f,\hat{\mathbb{E}})$ be a sublinear expectation system, where $U_f$ is the operator induced by a measurable transform $f:\O\to\O$. Denote 
\[
	\overline{X}:=\limsup_{n\to\infty}\frac{1}{n}\sum_{i=0}^{n-1}U_f^iX\text{ and }\underline{X}:=\liminf_{n\to\infty}\frac{1}{n}\sum_{i=0}^{n-1}U_f^iX\text{ for any }X\in\mathcal{H}. 
\]

	\begin{theorem}\label{thm:main1}
		Let $(\O,\mathcal{H},U_f,\hat{\mathbb{E}})$ be a sublinear expectation system, where $U_f$ is the operator induced by a measurable transform $f:\O\to\O$. Then 
		for any  $X\in \mathcal{H}$ if $\overline{X},\underline{X}\in\mathcal{H}$, then there exists $X^*\in\mathcal{H}$ with $U_fX^*=X^*$, $\hat{\mathbb{E}}$-q.s.  such that 
		\[\lim_{n\to\infty}\frac{1}{n}\sum_{i=0}^{n-1}U_f^iX=X^*\text{, }\hat{\mathbb{E}}\text{-q.s.}\]
		Moreover, the following two statements are equivalent:
		\begin{enumerate}[(i)]
			\item $\hat{\mathbb{E}}$ is ergodic (see Definition \ref{Definition of ergodicity});
			\item  for any  $X\in \mathcal{H}$ if
			$\overline{X},\underline{X}\in\mathcal{H},$ then there exists a constant $c_X\in\mathbb{R}$ such that 
			\[\lim_{n\to\infty}\frac{1}{n}\sum_{i=0}^{n-1}U_f^iX=c_X,\text{ }\hat{\mathbb{E}}\text{-q.s.}\]
		\end{enumerate}
If one of two equivalent statements above holds, then for any  invariant probability $P\in\mathcal{P}$ and  $X\in \mathcal{H}$ with
			$\overline{X},\underline{X}\in\mathcal{H}$
		\[\int XdP=c_X.\]
	\end{theorem}
\begin{remark}
  Let $(\O,\mathcal{F})$ be a measurable space.  When we consider the vector lattice $B_b(\O,\mathcal{F})$\footnote{If there is no ambiguity, we will omit $\mathcal{F}$ and denote it simply as $B_b(\O)$.} consisting of all bounded $\mathcal{F}$-measurable functions, it is straightforward to observe that
    \begin{equation}\label{0614-1}
        \overline{X},\underline{X}\in B_b(\O), \ \text{ for any } \ X\in B_b(\O).
    \end{equation}
    Thus, this result implies the correspnding results in \cite{FengZhao2021,FWZ2020,FHLZ2023}.
 In particular, if we consider the linear case, that is, $\hat{\mathbb E}$ is a linear expectation with respect to some invariant probability, then this result implies classical Birkhoff's ergodic theorem. 
    
However, when we consider the general vector lattice, the condition \eqref{0614-1} becomes necessary (see Example \ref{ex:not quasi}).
\end{remark}
\begin{remark}
 We will prove that the sublinear expectation systems generated by 
$G$-Brownian motion are ergodic (see Theorem \ref{ergodicity of G-Brownian motion}), which ensures that Theorem \ref{thm:main1} can be applied.
\end{remark}

Now we apply our results to investigate a class of nonlinear equations. Let $G: \mathbb{S}_d\to \mathbb{R}$ be a given monotone (i.e., $G(A)\le G(B)$ if $B-A$ is a semidefinite matrix) and sublinear function, where $\mathbb{S}_d$ is the collection of all $d\times d$ symmetric matrices and $B_t=\{B^i_t\}_{i=1}^d$ be the corresponding $d$-dimensional $G$-Brownian motion. We consider the following $G$-SDE: for any $x\in \mathbb{R}^n$,
	\begin{equation}\label{G-SDEintro}
		X_t^x=x+\int_0^tb(X_s^x)ds+\sum_{i,j=1}^d\int_0^th_{ij}(X_s^x)d\langle B^i,B^j\rangle_s+\int_0^t\sigma(X_s^x)dB_s,
	\end{equation}
	where $b,h_{ij}=h_{ji}: \mathbb{R}^n\to \mathbb{R}^n$ and $\sigma: \mathbb{R}^n\to \mathbb{R}^{n\times d}$ are deterministic continuous functions and  $\langle B^i,B^j\rangle_t$ is the mutual variation process of $B^i, B^j$ (see \eqref{eq:23.15} for details).

 \begin{condition}\label{Assumption G-SDE}
		Assume that $b,h_{ij},\sigma$ are locally Lipschitz continuous satisfying
  \begin{itemize}
      \item functions $\sigma$ is bounded and $b,h_{ij}$ are $\kappa$-polynomial growth for some $\kappa\geq 1$, i.e., there exists $C>0$ such that
      \begin{equation*}
          |b(x)|+\sum_{i,j=1}^d|h_{ij}(x)|\leq C(1+|x|^{\kappa}), \ \text{ for any } \ x\in \mathbb{R}^n;
      \end{equation*}
      \item there exists $\alpha>0$ such that 
		\begin{equation}\label{eq:dissipative assump}
			\begin{split}
				&\langle x_1-x_2, b(x_1)-b(x_2)\rangle\\
				&+G\bigl((\sigma(x_1)-\sigma(x_2))^{\top}(\sigma(x_1)-\sigma(x_2))+2[\langle x_1-x_2, h_{ij}(x_1)-h_{ij}(x_2)\rangle]_{i,j=1}^d\bigr)\\
				&\leq -\alpha|x_1-x_2|^2, \ \text{ for any } \ x_1,x_2\in \mathbb{R}^n.
			\end{split}
		\end{equation}
  \end{itemize}
\end{condition}
 Under Assumption \ref{Assumption G-SDE}, it is well-known that for any $x\in \mathbb{R}^n$, $G$-SDE \eqref{G-SDEintro} has a unique solution $\{X^x_t\}_{t\geq 0}$ (see e.g. \cite[Theorem 4.5]{Li-Lin-Lin2016}).
	We define $T_t$ for any $t\geq 0$ by
	\begin{equation}\label{def: T_t by G-SDEintroduction}
		T_tf(x):=\hat{\mathbb{E}}[f(X_t^x)], \ \text{ for all } \ x\in \mathbb{R}^n \ \text{ and  } \ f\in C_{lip}(\mathbb{R}^n).
	\end{equation}
	Then $T_t: C_{lip}(\mathbb{R}^n)\to C_{lip}(\mathbb{R}^n)$ and $\{T_t\}_{t\geq 0}$ is a sublinear Markovian system introduced by Peng  \cite{Peng2005} (see Definition \ref{def: markov}).

    \begin{remark}\label{Rem:example G-OU}
        Consider the following Ornstein–Uhlenbeck process $\{X_t\}_{t\geq 0}$ driven by $G$-Brownian mothion ($G$-OU process):
        \begin{equation}\label{G-OU process}
            dX_t=-AX_tdt+\Sigma dB_t,
        \end{equation}
        where $A\in \mathbb R^{n\times n}, \Sigma\in \mathbb{R}^{n\times d}$. If $A$ is positive definite, it is straightforward to check that $G$-SDE \eqref{G-OU process} satisfies Assumption \ref{Assumption G-SDE}.

        In the case that $n=d=1$, it can be also verified that the following $G$-SDE:
         \begin{equation*}
            dX_t=-(X_t+X_t^3)dt+dB_t,
        \end{equation*}
        satisfies Assumption \ref{Assumption G-SDE}.
    \end{remark}

 	\begin{theorem}\label{Thm: alpha-mixing GSDE}
		If Assumption \ref{Assumption G-SDE} holds, then the sublinear Markovian system $\{T_t\}_{t\geq 0}$ defined by \eqref{def: T_t by G-SDEintroduction} has a unique invariant sublinear expectation $\tilde{T}$, and for any $m\geq 1$, the $m$-dimensional sublinear expectation system $(\Omega, (Lip(\Omega))^m, (U_t)_{t\geq 0}, \hat{\mathbb{E}}^{\tilde{T}})$ is $\alpha$-mixing, continuous and regular.
  
        Moreover, let $L_G^1(\Omega)$ be
        the completion of $Lip(\Omega)$ under $\hat{\mathbb E}[|\cdot|]$. Then $(\Omega, L_G^1(\Omega), \hat{\mathbb{E}}^{\tilde{T}})$ is a regular mixing sublinear expectation space. 

        Denote by $\mathcal{P}$ the set of probabilities on $(\Omega,\sigma(L_G^1(\Omega)))$ generated by $(\Omega, L_G^1(\Omega), \hat{\mathbb{E}}^{\tilde{T}})$ as in \eqref{P of sub exp}. Then for any fixed $X\in (Lip(\Omega))^m$ and $\tau>0$,
\begin{enumerate}
    \item for any $\phi\in C_{lip}(\mathbb{R}^d)$,
		\begin{equation*}
			\lim_{n\to\infty}\left|\hat{\mathbb{E}}\left[\phi\left(\frac{1}{n}\sum_{k=0}^{n-1}U_\tau^kX\right)\right]-\max_{x\in \Gamma_{*}}\phi(x) \right|=0, \ \Gamma_*:=\bigcap_{n\geq 1}\left\{\frac{1}{n}\sum_{k=0}^{n-1}\mathbb{E}_P[U_{\tau}^kX]: P\in \mathcal{P}\right\};
		\end{equation*}
\item it holds
     \begin{equation*}
         \lim_{n\to\infty}dist\left(\frac{1}{n}\sum_{k=0}^{n-1}U_\tau^kX,\Gamma_*\right)=0, \ \text{$\hat{\E}$-q.s.};
         \end{equation*}
  \item  if in addition, assume that $\overline{X},\underline{X}\in (L_G^1(\Omega))^m$ or $X$ has no mean-uncertainty, there exists a constant $c_X\in\mathbb{R}$ such that 
			\[\lim_{n\to\infty}\frac{1}{n}\sum_{k=0}^{n-1}U_\tau^kX=c_X,\text{ }\hat{\mathbb{E}}\text{-q.s.}\]
\end{enumerate}
	\end{theorem}

 It is important to note that our results extend beyond applications to nonlinear equations; they are also applicable in classical ergodic theory and topological dynamical systems (see Section \ref{application:ergodic}) and in the context of capacity theory (see Section \ref{application:capacity}).

\medskip
	\medskip
	This manuscript is organized as follows. In Section \ref{sec:pre},
	we introduce main definitions and present  some necessary lemmas. Section \ref{sec:ergodic} investigates  the Birkhoff's ergodic theorem for sublinear expectation systems. 
In Section \ref{sec:mix}, we introduce and study the  mixing for sublinear expectation systems.   Section \ref{Sec: LLN} is dedicated to proving the LLN and the SLLN under $\alpha$-mixing conditions. In Section \ref{sec:GSDE}, we demonstrate that a class of processes induced by $G$-SDEs are $\a$-mixing. The final section discusses various applications in classical ergodic theory and capacity theory. In Appendix, we prove the classical Brownian motion is mixing.

\section{Preliminaries}\label{sec:pre}
	\subsection{Sublinear expectations}\label{sec:pre-sub}
	Let $X=(X_1,\cdots, X_n), X_i\in \mathcal{H}$, denoted by $X\in \mathcal{H}^n$, be a given n-dimensional random vector on a sublinear expectation space $(\Omega, \mathcal{H}, \hat{\mathbb{E}})$, and denote $|X|=\sum_{i=1}^n|X_i|$. Following Peng \cite{Pengbook}, given a sublinear expectation space $(\Omega, \mathcal{H}, \hat{\mathbb{E}})$, we assume that 
 \begin{equation}\label{eq:C_b_lip}
     \text{$\phi(X)\in \mathcal{H}$ for any $X\in \mathcal{H}^n$ and $\phi\in C_{b,lip}(\mathbb{R}^n)$},
 \end{equation}
 where $\phi\in C_{b,lip}(\mathbb{R}^n)$ is the collection of all bounded Lipschitz functions from $\mathbb{R}^n$ to $\mathbb{R}$.

 Let $\{X_i\}_{i=1}^\infty$ be a sequence in $\mathcal{H}$. Throughout this paper, we will use ``$X_i \downarrow \text{ (resp. $\uparrow$)} X$ as $i \to \infty$'' to represent ``$X_{i+1}(\omega) \le\text{ (resp. $\ge$)} X_i(\omega)$ for $i \in \mathbb{N}$, and $\lim_{i \to \infty} X_i(\omega) = X$ for each $\omega \in \Omega$.''
 \begin{definition}\label{def:regular}
A sublinear expectation space $(\Omega, \mathcal{H}, \hat{\mathbb{E}})$ is called regular if for each sequence $\{X_i\}_{i=1}^\infty$ in $\mathcal{H}$ with $X_i \downarrow 0$ as $i \to \infty$, one has $\lim_{i\to\infty}\hat{\mathbb{E}}[X_i]=0$.
 \end{definition}

 For a given regular sublinear expectation space $(\Omega, \mathcal{H}, \hat{\mathbb{E}})$, we consider the following set
 \begin{equation}\label{eq:extend-vector-lattice}
     \tilde{\mathcal{H}}:=\{\tilde{X}:\text{ there exist } X\in \mathcal{H}^n, \ \phi\in C_{lip}(\mathbb{R}^n)  \text{ for some } n\geq 1, \text{ such that } \tilde{X}=\phi(X)\}.
 \end{equation}
 It is easy to check that $\tilde{\mathcal{H}}\supset \mathcal{H}$ is a vector lattice. Moreover, for any $\tilde{X}\in \tilde{\mathcal{H}}$ and any $M,N>0$, we know that $\tilde{X}_M^N:=(\tilde{X}\wedge N)\vee (-M)=\psi(X)\in \mathcal{H}$ for some $X\in \mathcal{H}^n$ and $\psi=(\phi\wedge N)\vee (-M)\in C_{b,lip}(\mathbb{R}^n)$. Then we extend  the sublinear expectation $\hat{\mathbb{E}}$ to $\tilde{\mathcal{H}}$ by the following way:
 \begin{equation}\label{eq:extend-sub-expectation}
     \hat{\mathbb{E}}[\tilde{X}]:=\lim_{N\to\infty}\lim_{M\to\infty}\hat{\mathbb{E}}[\tilde{X}_M^N], \ \text{ for any } \ \tilde{X}\in \tilde{\mathcal{H}}.
 \end{equation}
 Then we have the following result.
 \begin{lemma}\label{lem:extend-sub-expectation}
     Assume that $(\Omega, \mathcal{H}, \hat{\mathbb{E}})$ is a regular sublinear expectation space. Let $\tilde{\mathcal{H}}$ be given by \eqref{eq:extend-vector-lattice} and the sublinear expectation $\hat{\mathbb{E}}$ on $\tilde{\mathcal{H}}$ be given as in \eqref{eq:extend-sub-expectation}. Then $(\Omega, \tilde{\mathcal{H}}, \hat{\mathbb{E}})$ is a regular sublinear expectation space.
 \end{lemma}
 \begin{proof}
     For any $\tilde{X}\in \tilde{\mathcal{H}}$, it follows from \eqref{eq:extend-vector-lattice} that there exist $C>0$ and $X\in \mathcal{H}^n$ for some $n\geq 1$ such that $|\tilde{X}_M^N|\leq |\tilde{X}|\leq C(1+|X|)\in \mathcal{H}$ for any $M,N>0$. Hence, the limit in \eqref{eq:extend-sub-expectation} exists and it is easy to verify that $\hat{\mathbb{E}}:\tilde{\mathcal{H}}\to \mathbb{R}$ is a sublinear expectation. 

     It remains to show that $(\Omega, \tilde{\mathcal{H}}, \hat{\mathbb{E}})$ is regular, i.e., for each sequence $\{\tilde{X}_i\}_{i=1}^{\infty}\subset \tilde{\mathcal{H}}$ with $\tilde{X}_i\downarrow 0$ as $i\to \infty$, one has $\lim_{i\to \infty}\hat{\mathbb{E}}[\tilde{X}_i]=0$. Note that $0\leq \tilde{X}_1\leq C(1+|X_1|):=Y\in \mathcal{H}$ for some $C>0$ and $X_1\in \mathcal{H}^n$. Since $(Y\vee N-N)\in \mathcal{H}$, $(Y\vee N-N)\downarrow 0$ as $N\to \infty$ and $(\Omega, \mathcal{H}, \hat{\mathbb{E}})$ is regular, we know that
     \begin{equation}\label{eq:23.23}
         \lim_{N\to \infty}\hat{\mathbb{E}}[\tilde{X}_1\vee N-N]\leq \lim_{N\to \infty}\hat{\mathbb{E}}[Y\vee N-N]=0.
     \end{equation}
     On the other hand, $\tilde{X}_i\wedge N\in \mathcal{H}$ for any $N>0$ and $\tilde{X}_i\wedge N\downarrow 0$ as $i\to \infty$, we conclude that for any $N>0$, $\lim_{i\to \infty}\hat{\mathbb{E}}[\tilde{X}_i\wedge N]=0$.
     Note that $\tilde{X}_i=(\tilde{X}_i\vee N-N)+\tilde{X}_i\wedge N\leq (\tilde{X}_1\vee N-N)+\tilde{X}_i\wedge N$ for any $i\geq 1$ and $N>0$. Then we have
     \begin{equation}\label{eq:23.24}
         \lim_{i\to \infty}\hat{\mathbb{E}}[\tilde{X}_i]\leq \hat{\mathbb{E}}[\tilde{X}_1\vee N-N]+\limsup_{i\to \infty}\hat{\mathbb{E}}[\tilde{X}_i\wedge N]=\hat{\mathbb{E}}[\tilde{X}_1\vee N-N], \ \text{ for all } \ N>0.
     \end{equation}
     Combining \eqref{eq:23.23} and \eqref{eq:23.24}, we obtain the desired result.
 \end{proof}
Thus, if $(\Omega, \mathcal{H}, \hat{\mathbb{E}})$ is regular, Lemma \ref{lem:extend-sub-expectation}, allows us to strengthen the assumption in \eqref{eq:C_b_lip} to the following form:
 \begin{equation}\label{eq:C_lip}
     \text{$\phi(X)\in \mathcal{H}$ for any $X\in \mathcal{H}^n$ and $\phi\in C_{lip}(\mathbb{R}^n)$}.
 \end{equation}
Consequently, this paper focuses on sublinear expectation spaces that meet the condition of \eqref{eq:C_lip}.

 Now for any $X\in \mathcal{H}^n$, we define a functional on $C_{lip}(\mathbb{R}^n)$ by
	$$\hat{\mathbb{F}}_X[\phi]:=\hat{\mathbb{E}}[\phi(X)], \text{ for all } \phi\in C_{lip}(\mathbb{R}^n).$$
	The triple $(\mathbb{R}^n, C_{lip}(\mathbb{R}^n), \hat{\mathbb{F}}_X)$ forms a sublinear expectation space and $\hat{\mathbb{F}}_X$ is called the distribution of X.
	
	\begin{definition}\label{Def: identically distributed}
		Let $X_1$ and $X_2$ be two n-dimensional random vectors defined respectively in sublinear expectation spaces $(\Omega_1, \mathcal{H}_1, \hat{\mathbb{E}}_1)$ and $(\Omega_2, \mathcal{H}_2, \hat{\mathbb{E}}_2)$. They are called identically distributed, denoted by $X_1\deq X_2$, if
        \begin{equation}\label{eq:def-distribution}
            \hat{\mathbb{E}}_1[\phi(X_1)]=\hat{\mathbb{E}}_2[\phi(X_2)], \text{ for all } \phi \in C_{lip}(\mathbb{R}^n).
        \end{equation}
		It is clear that $X_1\stackrel{d}{=}X_2$ if and only if their distributions coincide.
	\end{definition}
	
	The definition of independence under sublinear expectation framework is given by Peng (see  \cite{Pengbook}).
	\begin{definition}\label{def:independent}
		Let $(\Omega, \mathcal{H}, \hat{\mathbb{E}})$ be a sublinear expectation space. A random vector $Y\in \mathcal{H}^n$ is said to be independent from another random vector $X\in \mathcal{H}^m$ under $\hat{\mathbb{E}}$ if for each test function $\phi \in C_{lip}(\mathbb{R}^{n\times m})$ we have
        \begin{equation}\label{eq:def-independent}
            \hat{\mathbb{E}}[\phi(X,Y)]=\hat{\mathbb{E}}[\hat{\mathbb{E}}[\phi(x,Y)]_{x=X}].
        \end{equation}
  
        For two sets $\mathcal{A}_1\subset \mathcal{H}^m, \mathcal{A}_2\subset \mathcal{H}^n $, we call $\mathcal{A}_2$ is independent from $\mathcal{A}_1$ if  for any $X_i\in \mathcal{A}_i$, $i=1,2$,  $X_2$ is independent from $X_1$.
	\end{definition}
\begin{remark}
    In the context of a regular sublinear expectation space, the test function $\phi$ in Definition \ref{Def: identically distributed} can equivalently range over $C_{b,lip}(\mathbb{R}^n)$ instead of $C_{lip}(\mathbb{R}^n)$, and in Definition \ref{def:independent}, it can similarly range over $C_{b,lip}(\mathbb{R}^{n\times m})$ instead of $C_{lip}(\mathbb{R}^{n\times m})$.
\end{remark}
 
	\begin{remark}\label{rem:diff for inde}
	 Note that, different from the linear expectation case, the independence of $Y$ from $X$ does not generally imply the independence of $X$ from $Y$ (see \cite[Example 1.3.15]{Pengbook} for a counterexample).
	\end{remark}

    \begin{definition}\label{def:iid-sequence}
        Let $(\Omega, \mathcal{H}, \hat{\mathbb{E}})$ be a sublinear expectation space. A sequence $\{X_k\}_{k\geq 1}$ is called an i.i.d. sequence if $X_i\deq X_j$ for all $i,j\geq 1$ and $X_n$ is independent from $(X_{n-1},\cdots,X_1)$ for any $n\geq 2$.
    \end{definition}
    
	\begin{definition}\label{Def:q.s.via sub}
		Let $(\Omega, \mathcal{H}, \hat{\mathbb{E}})$ be a sublinear expectation space. For any $X, Y\in \mathcal{H}$, we say $X\geq Y$, $\hat{\mathbb{E}}$-quasi surely ($\hat{\mathbb{E}}$-q.s. for short) if
		$$\hat{\mathbb{E}}[(X-Y)^-]=0.$$
	We say $X=Y$, $\hat{\mathbb{E}}$-q.s. if $X\geq Y$ and $Y \geq X$, $\hat{\mathbb{E}}$-q.s.
	\end{definition}

	Let $\Theta$ be the family of all linear expectations dominated by $\hat{\mathbb{E}}$, i.e., $\mathbb{E}[X]\le \hat{\mathbb{E}}[X]$ for any $X\in\mathcal{H}$ and  $\mathbb{E}\in\Theta$.
	Then by Theorem  1.2.1 in \cite{Pengbook}, we have that 
	\begin{equation}\label{eq:represenation of sublinear}
		\hat{\mathbb{E}}[X]=\max_{\mathbb{E}\in\Theta}\mathbb{E}[X],\text{ for all }X\in\mathcal{H}.
	\end{equation}
	Moreover, if assume that $\hat{\mathbb E}$ is regular, then by Theorem 1.2.2 in \cite{Pengbook}, for any $\mathbb{E}\in\T$, there exists a probability $P$ on $(\O,\s(\mathcal{H}))$ such that 
	\begin{equation}\label{eq:representation of probability}
		{\mathbb{E}}[X]=\int XdP, \text{ for any }X\in\mathcal{H}.
	\end{equation}
	Denote by $\mathcal{P}$ the set of all probabilities $P$ obtained by \eqref{eq:representation of probability}  for all $\mathbb{E}\in \Theta$. Since for any $X\in\mathcal{H}$, $\hat{\E}[|X|]<\infty$, one has  $\mathcal{H}\subset L^1(\O,\s(\mathcal{H}),P)$ for all $P\in \mathcal{P}$.
	Define a capacity with respect to $\hat{\mathbb{E}}$ via
	\[C(A)=\sup_{P\in\mathcal{P}}P(A),\text{ for any }A\in\s(\mathcal{H}).\]
    \begin{remark}
        It is easy to check that $X\ge Y$, $\hat{\mathbb{E}}$-q.s. is equivalent to $C(\{\o\in\O:X(\o)\ge Y(\o)\}^c)=0$, when the sublinear expectation is regular.
    \end{remark}
	If in addition, we suppose that $\O$ is a complete separable metric space, and $\mathcal{H}=C_b(\O)$, the space of all bounded continuous functions, then we have more information about the corresponding $\mathcal{P}$.
	The following result can be found in \cite[Theorem 6.1.16]{Pengbook}.
	\begin{lemma}\label{lem:regualr<=>compact}
		Suppose that  $\O$ is a complete separable metric space, and $\mathcal{H}=C_b(\O)$.	The sublinear expectation $\hat{\mathbb{E}}$ is regular if and only if $\mathcal{P}$ is relatively compact  with respect to the weak convergence topology \footnote{A sequence of probabilities $\{P_n\}_{n=1}^\infty$ on $(\O,\mathcal{B}(\O))$  is said to be weakly converges to $P$ if for any bounded continuous function $X$ on $\O$ 
     $\lim_{n\to\infty}\int XdP_n=\int XdP.$ Then weak convergence topology is the topology induced by weak convergence.}.
	\end{lemma}
 
	\subsection{Invariant probabilities}
Let $(\Omega,\mathcal{F})$ be a measurable space, and $f:\Omega\to \Omega$ be a measurable transformation, i.e., $f^{-1}(\mathcal{F})\subset \mathcal{F}$.
     Denote
    \[
    \mathcal{I}:=\{A\in \mathcal{F}: f^{-1}A=A\}.
    \]
    A probability measure $P$ on $(\Omega,\mathcal{F})$ is said to be $f$-invariant if $P(f^{-1}A)=P(A)$ for all $A\in \mathcal{F}$. We say $P$ is $f$-ergodic if it is $f$-invariant and $P(A)=0$ or $1$ for all $A\in \mathcal{I}$.    Denote by $\mathcal{M}(f)$ the set of all $f$-invariant probabilities on $(\Omega,\mathcal{F})$, and by $ \mathcal{M}^e(f)$ the set of all $f$-ergodic probabilities on $(\Omega,\mathcal{F})$. 

	The following result should be classical (see \cite[Lemma 2.9]{FHLZ2023} for a proof).
	\begin{lemma}\label{lem:invariant meausre }
		Let $(\O,\mathcal{F})$ be a measurable space, and $f:\O\to \O$ be a measurable transformation. Given $P,Q\in\mathcal{M}(f)$, if $P(A)=Q(A)$ for any $A\in\mathcal{I}$, then $P=Q$. 
	\end{lemma}

    Correspondingly,   let $(\Omega,\mathcal{F})$ be a measurable space and $(\t_t)_{t\ge0}$ be a semigroup of measurable transformations  from $ \Omega$ to  $\Omega$ satisfying the map $\mathbb{R}_+\times \O\to\O,~(t,\o)\mapsto \t_t\o$ is measurable. A probability $P$ on $(\Omega,\mathcal{F})$ is called $\t$-invariant if  $P(\t_t^{-1}A)=P(A)$ for any $t\in\mathbb{R}_+$. We also call $(\Omega,\mathcal{F},(\t_t)_{t\ge0},P)$ a measure-preserving system. For convenience, we also denote 
    \[
    \mathcal{I}:=\{A\in \mathcal{F}: \t_t^{-1}A=A\text{ for any }t\in\mathbb{R}_+\}.
    \]
    We say $P$ is $\t$-ergodic if $P(A)=0$ or $1$ for all $A\in \mathcal{I}$. Denote by $\mathcal{M}(\t)$ the set of all $\t$-invariant probabilities on $(\Omega,\mathcal{F})$, and by $ \mathcal{M}^e(\t)$ the set of all $\t$-ergodic probabilities on $(\Omega,\mathcal{F})$.

We recall the classical Birkhoff's ergodic theorem for continuous time systems, which is proved by combining \cite[Page 7]{DaPratoZabczyk1996} and \cite[Lemma 1.2]{BekkaMayer2000}.
 \begin{lemma}\label{lem:classical ergodic theorem for flow}
     Let $(\O,\mathcal{F},(\t_t)_{t\geq 0},P)$ be a measure-preserving system. Then for any $X\in B_b(\O,\mathcal{F})$, there exists $X^*\in B_b(\O,\mathcal{I})$ such that
     \[\lim_{T\to\infty}\frac{1}{T}\int_0^TX\circ \t_tdt=X^*,~~P\text{-a.s.}\]
    Furthermore, $\int X^*dP=\int XdP$.
 \end{lemma}
 \begin{proof}
  By   \cite[Page 7]{DaPratoZabczyk1996}, there exists  $X^*_1\in B_b(\O,\mathcal{F})$ such that for any $t\in\mathbb{R}_+$, $X^*_1\circ \t_t=X^*_1$, $P$-a.s., and satisfies all requirements in this lemma. By \cite[Lemma 1.2]{BekkaMayer2000}, we obtain a function $X^*\in B_b(\O,\mathcal{I})$  such that $X^*=X^*_1$, $P$-a.s. The proof is completed.
 \end{proof}

  The following result can be proved by a similar argument of Lemma \ref{lem:invariant meausre }. For completeness, we provide a proof.
 \begin{lemma}\label{lem:invariant meausre for flow}
		Let $(\O,\mathcal{F},(\t_t)_{t\geq 0},P)$ be a measure-preserving system.  Given $P,Q\in\mathcal{M}(\t)$, if $P(A)=Q(A)$ for any $A\in\mathcal{I}$, then $P=Q$. 
	\end{lemma}
 \begin{proof}
Fix any $A\in\mathcal{F}$ and let $X=1_A$. By Lemma \ref{lem:classical ergodic theorem for flow}, there exist $X^*_{P},X^*_{Q}\in B_b(\O,\mathcal{I})$ such that 
	\[P(\{\o\in\O:\lim_{T\to\infty}\frac{1}{T}\int_0^TX( \t_t\o)dt=X^*_P(\o)\})=1,\]
	\[Q(\{\o\in\O:\lim_{T\to\infty}\frac{1}{T}\int_0^TX( \t_t\o)dt=X^*_Q(\o)\})=1,\]
and 
\begin{equation}\label{eq:22.41}
    \int XdP= \int X^*_PdP,~~ \int XdQ= \int X^*_QdQ.
\end{equation}
 Since $\{\o\in\O:\lim_{T\to\infty}\frac{1}{T}\int_0^TX( \t_t\o)dt=X^*_P(\o)\}\in \mathcal{I}$ and $P|_{\mathcal{I}}=Q|_{\mathcal{I}}$, it follows that 
	\[Q(\{\o\in\O:\lim_{T\to\infty}\frac{1}{T}\int_0^TX( \t_t\o)dt=X^*_P(\o)\})=1.\]
	Thus, $X^*_P(\o)=X^*_Q(\o)$, $Q$-a.s., and hence
 \begin{equation}\label{eq:22.43}
     \int X^*_PdQ=\int X^*_QdQ.
 \end{equation}
	Since $P|_{\mathcal{I}}=Q|_{\mathcal{I}}$ and $X^*_P(\o)$ is $\mathcal{I}$-measurable, it follows that 
	\begin{equation}\label{eq:22.45}
	    \int X^*_PdP=\int X^*_PdQ.
	\end{equation}
 Therefore, we have
 \begin{align*}
  P(A)=  \int XdP\overset{\eqref{eq:22.41}}{=}\int X^*_PdP\overset{\eqref{eq:22.45}}{=}\int X^*_PdQ\overset{\eqref{eq:22.43}}{=}\int X^*_QdQ\overset{\eqref{eq:22.41}}{=}\int XdQ=Q(A).  
 \end{align*}

	The proof is completed as $A\in\mathcal{F}$ is arbitrary.
\end{proof}
	\section{Ergodicity}\label{sec:ergodic} 
In this section, we first provide a framework to study ergodic theory under the sublinear expectation setting. Subsequently, we prove a sublinear version of Birkhoff's ergodic theorem. Similar to the classical case, if we further assume that this system is ergodic, then some observable functions' Birkhoff averages converge to the integral with respect to some ergodic measure. Finally, we prove that $G$-Brownian motion is ergodic, which plays a crucial role in sublinear stochastic analysis.

Let us begin with some basic notations and definitions.	For a given sublinear expectation space $(\Omega, \mathcal{H}, \hat{\mathbb{E}})$, we consider a semigroup (group) of    linear operators $(U_t)_{t\in \mathbb T}$ on $\mathcal{H}$ with $\mathbb T=\mathbb N, \mathbb R_+$ ($\mathbb T=\mathbb Z, \mathbb R$), which preserves the sublinear expectation $\hat{\mathbb{E}}$, i.e. 
	$$\hat{\mathbb{E}}[U_tX]=\hat{\mathbb{E}}[X], \text{ for all } X\in \mathcal{H}, \ \text{ for all } \ t\in \mathbb T.$$
	Given a sublinear expectation system $(\Omega, \mathcal{H},(U_t)_{t\in \mathbb T}, \hat{\mathbb{E}})$, set 
	$$\mathcal{I}:=\{X\in \mathcal{H} : U_tX=X, \ \hat{\mathbb{E}}\text{-q.s.}   \text{ for all } t\in \mathbb T\}.$$
	For convenience, in the following, we will use 
$\mathcal{I}$ to denote both invariant sets and invariant functions. This will not cause any ambiguity.

Now it is natural to give the following definition of ergodicity.
	\begin{definition}
		\label{Definition of ergodicity}
		We say the quadruple $(\Omega, \mathcal{H}, (U_t)_{t\in \mathbb T}, \hat{\mathbb{E}})$ is ergodic if for all $X\in \mathcal{I}$, $X$ is  constant, $\hat{\mathbb E}$-q.s.
	\end{definition}
	\begin{remark}
		\label{Ergodicity of U implies ergodicity of U_t}
		If $\mathbb T=\mathbb R_+, \mathbb R$, it is easy to check that for any fixed $\tau>0$, the ergodicity of $(\Omega, \mathcal{H}, U_{\tau}, \hat{\mathbb{E}})$   implies the ergodicity of $(\Omega, \mathcal{H}, (U_t)_{t\in \mathbb T}, \hat{\mathbb{E}})$.
	\end{remark}
	\subsection{Ergodic theorem}
	\subsubsection{Ergodic theorem for invariant sublinear expectation systems}
	In this subsection, we investigate the properties of invariant sublinear expectation. Furthermore, we provide a sublinear version of Birkhoff's ergodic theorem. 
	
Let $(\O,\mathcal{H},\hat{\E})$ be a sublinear expectation space.	According to $\eqref{eq:set Theta}$, there exists a family $\T$ of linear expectations  such that 
 \[\hat{\E}[X]=\sup_{\E\in\T}\E[X],\text{ for any }X\in\mathcal{H}.\]
Let $U:\mathcal{H}\to\mathcal{H}$ be a linear operator such that $\hat{\E}$ is invariant.  Now we prove that each $\E\in\T$ corresponds to an invariant linear expectation $\E_{\mathrm{inv}}$ such that $\E_{\mathrm{inv}}[X]=\E[X]$ for any $X\in\mathcal{I}$.  

	A positive linear functional $\phi:\ell^\infty\to \mathbb{R}$ is called a Banach-Mazur limit if 
	\begin{enumerate}
		\item $\phi(\textbf{e})=1$, where $\textbf{e}=(1,1,\cdots)$, and 
		\medskip
		\item $\phi(x_1,x_2,\cdots)=\phi(x_2,x_3,\cdots)$ for each $\mathbf{x}=(x_1,x_2,\cdots)\in \ell^\infty$,
	\end{enumerate}
	where $\ell^\infty$ is the space of all bounded sequences of $\mathbb{R}$ with the maximal norm. The existence of Banach-Mazur limit is proved in \cite[Theorem 16.47]{Aliprantis1999}.
	
	The following result can be found in \cite[Lemma 16.45]{Aliprantis1999}.
	\begin{lemma}\label{lem:limit and functional}
		If $\phi$ is a Banach-Mazur limit, then 
		\[\liminf_{n\to \infty}x_n\le \phi(\mathbf{x})\le\limsup_{n\to\infty}x_n\]
		for each $\mathbf{x}=(x_1,x_2,\cdots)\in \ell^\infty$. In particular, if the limit $\lim_{n\to\infty}x_n$ exists, then $\phi(\mathbf{x})= \lim_{n\to\infty}x_n$.
	\end{lemma}
	Following ideas of Theorem 16.48 in \cite{Aliprantis1999} (a similar technique also was used in \cite{CMM2016,FHLZ2023}), we have the following result.
	\begin{theorem}\label{thm:existence of inv}
		Let $(\O,\mathcal{H},U,\hat{\mathbb{E}})$ be a sublinear expectation system. Then for any $\mathbb{E}\in\Theta$, there exists $\mathbb{E}_{\operatorname{inv}}\in\T$ such that 
		\[\mathbb{E}_{\operatorname{inv}}[UX]=\mathbb{E}_{\operatorname{inv}}[X],\text{ for any }X\in\mathcal{H},\]
		and 
		\[\mathbb{E}[X]=\mathbb{E}_{\operatorname{inv}}[X]\text{, for any }X\in\mathcal{I}.\]
	\end{theorem}
	\begin{proof}
		Fix any $\mathbb{E}\in\T$.
		For any $X\in\mathcal{H}$, let $x_X=(\mathbb{E}_1[X],\mathbb{E}_2[X],\cdots)$, where
		\[\mathbb{E}_n[X]=\mathbb{E}\left[\frac{1}{n}\sum_{i=0}^{n-1}U^i X\right],\text{  for any }n\in\mathbb{N}.\]
		Since $\hat{\mathbb{E}}[|X|]<\infty$, it follows that $x_X\in\ell^\infty$. Define
		\[\mathbb{E}_{\operatorname{inv}}[X]=\phi(x_X)\text{ for any  }X\in\mathcal{H},\]
		where $\phi$ is a Banach-Mazur limit.
		Now we check that $\mathbb{E}_{\operatorname{inv}}\in\T$. Indeed, for any $\l_1,\l_2\in\mathbb{R}$ and $X_1,X_2\in\mathcal{H}$, as $\phi$ is a linear functional, it follows that
		\[\mathbb{E}_{\operatorname{inv}}[\l_1X_1+\l_2X_2]=\phi(x_{\l_1X_1+\l_2X_2})=\l_1\phi(x_{X_1})+\l_2\phi(x_{X_2})=\l_1\mathbb{E}_{\operatorname{inv}}[X_1]+\l_2\mathbb{E}_{\operatorname{inv}}[X_2].\]
		Meanwhile, by Lemma \ref{lem:limit and functional}, we have that 
		\[\mathbb{E}_{\operatorname{inv}}[c]=c,\text{ for any }c\in\mathbb{R}\text{ and }\mathbb{E}_{\operatorname{inv}}[X]\le \limsup_{n\to\infty}\mathbb{E}_n[X].\]
		Since for any $X\in \mathcal{H}$, $\limsup_{n\to\infty}\mathbb{E}_n[X]\le\hat{\mathbb E}[X]$, it follows that  $\mathbb{E}_{\operatorname{inv}}\in\T.$ 
		
		Now we prove it is $U$-invariant.  For any $n\in\mathbb{N}$, and $X\in\mathcal{H}$, 
		\[\mathbb{E}_n[UX]=\mathbb{E}\left[\frac{1}{n}\sum_{i=0}^{n-1}U^{i+1}X\right]=\frac{n+1}{n}\mathbb{E}_{n+1}[X]-\frac{1}{n}\mathbb{E}[X].\]
		Let $y=(\mathbb{E}_2[X],\mathbb{E}_3[X],\cdots)\in \ell^\infty$, and define $z=x_{UX}-y\in\ell^\infty$. Then
		\[|z_n|=|\mathbb{E}_n[UX]-\mathbb{E}_{n+1}[X]|\le\frac{1}{n}|\mathbb{E}_{n+1}[X]-\mathbb{E}[X]|\le \frac{2}{n}\hat{\mathbb{E}}[|X|],\]
		which, by Lemma \ref{lem:limit and functional}, shows that $\phi(z)=\lim_{n\to\infty}z_n=0$. Thus, by the second condition in the definition of Banach-Mazur limit, we have that
		\begin{align*}
			|\mathbb{E}_{\operatorname{inv}}[UX]-\mathbb{E}_{\operatorname{inv}}[X]|=|\phi(x_{UX})-\phi(x_X)|=|\phi(x_{UX})-\phi(y)|=|\phi(z)|=0.
		\end{align*}
        If $X\in \mathcal{I}$, then  $x_X=(\mathbb E[X], \mathbb E[X], \cdots)$, which together with Lemma \ref{lem:limit and functional}, yields $\mathbb{E}_{\operatorname{inv}}[X]=\phi(x_X)=\mathbb{E}[X]$.
		Now we finish the proof, as $X\in\mathcal{H}$ is arbitrary.
	\end{proof}
For sublinear expectation systems with continuous time, we have a similar result.
	\begin{theorem}\label{thm: existence of invariant skeleton for flow}
		Let $(\O,\mathcal{H},(U_t)_{t\geq 0},\hat{\mathbb{E}})$ be a continuous  sublinear expectation system  (see \eqref{eq:continuous} for the definition). Then for any $\mathbb{E}\in\Theta$, there exists $\mathbb{E}_{\operatorname{inv}}\in\T$ such that 
	\[\mathbb{E}_{\operatorname{inv}}[U_tX]=\mathbb{E}_{\operatorname{inv}}[X],\text{ for any }X\in\mathcal{H}\text{ and }t\ge0,\]
	and 
	\[\mathbb{E}[X]=\mathbb{E}_{\operatorname{inv}}[X]\text{, for any }X\in\mathcal{I}.\]
	\end{theorem}
	\begin{proof}
	Fix any $\mathbb{E}\in\T$. We firstly prove the map $s\mapsto \mathbb E[U_sX]$ is continuous for any fixed $X\in\mathcal{H}$. Indeed, for any $s<t$
 \begin{align*}
     |\mathbb E[U_sX]-\mathbb E[U_tX]|\le \hat{\mathbb E}[|U_sX-U_tX|]=\hat{\mathbb E}[|X-U_{t-s}X|],
 \end{align*}
 which together with the continuity of the system, implies the continuity of the map $s\mapsto \mathbb E[U_sX]$. 
 
	Given any $X\in\mathcal{H}$, let
	\[\mathbb{E}_n[X]=\frac{1}{n}\int_0^n\mathbb{E}\left[U_sX\right]ds,\text{  for any }n\in\mathbb{N},\]
where the integrability is well-defined, as  the map $s\mapsto \mathbb E[U_sX]$ is continuous.
 
	Let $x_X=(\mathbb{E}_1[X],\mathbb{E}_2[X],\cdots)$, and define
	\[\mathbb{E}_{\operatorname{inv}}[X]=\phi(x_X)\text{ for any  }X\in\mathcal{H},\]
	where $\phi$ is a Banach-Mazur limit.
	Similar to the discrete case, we can  check that $\mathbb{E}_{\operatorname{inv}}\in\T.$ 
	
	Now we prove that for any $t\ge0$, ${\mathbb{E}}_{\operatorname{inv}}[U_tX]={\mathbb{E}}_{\operatorname{inv}}[X]$. Indeed, we only need to prove this for $0<t\le 1$. Fix some $0<t\le 1$. For any $n\in\mathbb{N}$, and $X\in\mathcal{H}$, 
	\[\mathbb{E}_n[U_tX]=\frac{1}{n}\int_0^n\mathbb{E}\left[U_{s+t}X\right]ds=\frac{n+1}{n}\mathbb{E}_{n+1}[X]-\frac{1}{n}\int_0^t\mathbb{E}[U_sX]ds-\frac{1}{n}\int_{n+t}^{n+1}\mathbb{E}[U_sX]ds.\]
	Let $y=(\mathbb{E}_2[X],\mathbb{E}_3[X],\cdots)\in \ell^\infty$, and define $z=x_{U_tX}-y\in\ell^\infty$. Then
	\begin{align*}
	    |z_n|&=|\mathbb{E}_n[U_tX]-\mathbb{E}_{n+1}[X]|\le\frac{1}{n}(|\mathbb{E}_{n+1}[X]|+\int_0^t|\mathbb{E}[U_sX]|ds+\int_{n+t}^{n+1}|\mathbb{E}[U_sX]|ds)\\
     &\le\frac{1}{n}(\hat{\mathbb E}[|X|]+t\hat{\mathbb E}[|X|]+(1-t)\hat{\mathbb E}[|X|])\le \frac{2}{n}\hat{\mathbb E}[|X|],
	\end{align*}
	which together with Lemma \ref{lem:limit and functional}, shows that $\phi(z)=\lim_{n\to\infty}z_n=0$. By an argument similar to that of discrete cases, we finish the proof.
\end{proof}

Given a sublinear expectation system $(\Omega,\mathcal{H}, U, \hat{\E})$, from Theorem \ref{thm:existence of inv}, we know that for any $\mathbb{E}\in \T$, there exists $\mathbb{E}_{\mathrm{inv}}$ such that \[\mathbb{E}_{\operatorname{inv}}[UX]=\mathbb{E}_{\operatorname{inv}}[X],\text{ for any }X\in\mathcal{H},\]
 	and 
		\[\mathbb{E}[X]=\mathbb{E}_{\operatorname{inv}}[X]\text{, for any }X\in\mathcal{I}.\]
 If in addition, $\hat{\mathbb{E}}$ is regular, by \eqref{eq:representation of probability}, there exists  $P_{\operatorname{inv}}\in\mathcal{P}$ such that 
	$\mathbb{E}_{\mathrm{inv}}[X]=\int XdP_{\operatorname{inv}}$ for any $X\in\mathcal{H}$, 
	and hence
	\begin{equation}\label{eq:invariant}
		\int XdP_{\operatorname{inv}}=\int UXdP_{\operatorname{inv}}\text{ for any $X\in\mathcal{H}$}.
	\end{equation}
    A similar result is obtained for the continuous sublinear expectation system $(\Omega,\mathcal{H}, (U_t)_{t\geq 0}, \hat{\E})$.
 
In the following, both in the discrete and continuous time cases, we refer to 
 $\mathbb{E}_{\operatorname{inv}}$ or $P_{\operatorname{inv}}$ as the invariant skeleton of $\mathbb{E}\in \T$.
 
Now we prove $P_{\mathrm{inv}}$  corresponds to the invariant measure in the classical sense of ergodic theory. Given a measurable space $(\O,\mathcal{F})$, let $f:\O\to\O$  be a measurable transformation and $(\t_t)_{t\ge0}$ be a semigroup of measurable transformations such that the map $\mathbb{R}_+\times \O\to\O,~(t,\o)\mapsto \t_t\o$ is measurable.
		In the following, whenever we mention a semigroup $(\t_t)_{t\ge0}$, we will always assume it satisfies the aforementioned measurability condition. This assumption is standard in classical ergodic theory and is satisfied by the examples we will use later, such as $G$-Brownian motion.
  
 Now we begin showing that the invariant skeleton $P_{\mathrm{inv}}$ of each $\mathbb E\in \T$ is a classical invariant probability if $U_f$  is induced by a measurable map $f:\O\to\O$  or $(U_t)_{t\ge0}$ induced by a semigroup of measurable transformations $(\t_t)_{t\ge0}$.
	\begin{proposition}\label{re:existence of invariant measure}
        Let $(\Omega,\mathcal{H}, U_f, \hat{\E})$ be a regular sublinear expectation system, where $U_f$ is induced by a measurable map $f:\O\to\O$. Then for any $\mathbb{E}\in \T$, the invariant skeleton $P_{\mathrm{inv}}$ of $\mathbb{E}$ is an $f$-invariant probability on $(\O,\s(\mathcal{H}))$.  
	\end{proposition}
 \begin{proof}
To prove it, we firstly introduce some notations.  Let 
		\[\widetilde{\mathcal{H}}_+=\{X:\exists X_n\in\mathcal{H}_+\text{ such that }X_n\uparrow X\},\]
		where $\mathcal{H}_+=\{X\in\mathcal{H}:X\ge0\}$. Furthermore, let 
		\begin{equation}\label{eq:C}
		    \mathcal{C}=\{C\in\s(\mathcal{H}):\one_C\in\widetilde{\mathcal{H}}_+\}.
		\end{equation}
  Then it is easy to check that $\widetilde{\mathcal{H}}_+$ satisfies that for any $X,Y\in\widetilde{\mathcal{H}}_+$,  $X\vee Y$ and $X\wedge Y \in\widetilde{\mathcal{H}}_+$.
  Since for any $A,B\in\mathcal{C}$, $\one_{A\cap B}=\one_A\wedge \one_B$, it follows that $A\cap B\in\mathcal{C}$, that is,  $\mathcal{C}$ is a $\pi$-system (see \cite[Page 41]{Billingsley1995book} for definition). We can directly check that $\mathcal{A}:=\{A\in\s(\mathcal{H}):P_{\mathrm{inv}}(f^{-1}A)=P_{\mathrm{inv}}(A)\}$ is a $\l$-system (or called Dynkin system (see \cite[Page 41]{Billingsley1995book} for definition)).  As $1\wedge(n(X-\a)^+)\uparrow \one_{\{X>\a\}}$, when $n\to\infty$, it follows that for any $X\in\mathcal{H}$ and $\a\in\mathbb{R}$,  
		${\{X>\a\}}\in\mathcal{C}$. Thus, $\s(\mathcal{C})=\s(\mathcal{H})$. According to Dynkin's $\pi$-$\l$ theorem (see \cite[Theorem 3.2]{Billingsley1995book}), to prove $P_{\mathrm{inv}}$ is $f$-invariant, we only need to prove that $P_{\mathrm{inv}}(f^{-1}C)=P_{\mathrm{inv}}(C)$ for any $C\in\mathcal{C}$. Given any $C\in\mathcal{C}$, there exists a sequence $\{X_n\}_{n=1}^\infty\subset \mathcal{H}_+$ such that $X_n\uparrow \one_C$, as $n\to\infty$. Thus,  by monotone convergence theorem, one has that
		\[P_{\mathrm{inv}}(f^{-1}C)=\lim_{n\to\infty}\int X_n\circ fdP_{\mathrm{inv}}=\lim_{n\to\infty}\int X_ndP_{\mathrm{inv}}=P_{\mathrm{inv}}(C).\]
	The proof is completed.
 \end{proof}
\begin{proposition}\label{re:existence of invariant measure for flow}
Let $(\Omega,\mathcal{H}, (U_t)_{t\geq 0}, \hat{\E})$ be a regular sublinear expectation system, where $(U_t)_{t\ge0}$ is induced by a semigroup of measurable transformations $(\t_t)_{t\ge0}$. Then for any $\mathbb{E}\in \T$, the invariant skeleton $P_{\mathrm{inv}}$ of $\mathbb{E}$ is a $\t$-invariant probability on $(\O,\s(\mathcal{H}))$.
\end{proposition}
\begin{proof}
    Fix any $t>0$. Since $U_t$ can be viewed as a linear operator induced by a measurable map $\t_t$, it follows from Proposition \ref{re:existence of invariant measure}, we have that 
    \[P_{\mathrm{inv}}(\t_t^{-1}C)=P_{\mathrm{inv}}(C),\text{ for any }C\in\sigma(\mathcal{H}).\]
    By the arbitrariness of $t>0$, we finish the proof.
\end{proof}

 Applying Lemma \ref{lem:invariant meausre } (resp. Lemma \ref{lem:invariant meausre for flow}) and Proposition \ref{re:existence of invariant measure} (resp. Proposition \ref{re:existence of invariant measure for flow}) on $\mathcal{H}=B_b(\O)$, we have the following corollary, which implies the corresponding  result in \cite{FHLZ2023} for capacities.
\begin{corollary}\label{cor;8.28}
Let $(\Omega,B_b(\O), U_f, \hat{\E})$ (resp. $(\Omega,B_b(\O), (U_t)_{t\geq 0}, \hat{\E})$) be a regular (resp. continuous) sublinear expectation system, where $U_f$ is induced by a measurable transformation $f:\O\to\O$  (resp. $(U_t)_{t\ge0}$ is induced by a semigroup of measurable transformations $(\t_t)_{t\ge0}$). Then each $P\in\mathcal{P}$ has a unique invariant skeleton $P_{\mathrm{inv}}$.
\end{corollary}

 Now we want to prove a Birkhoff's ergodic theorem for discrete  and continuous time sublinear expectation systems. First, we consider the discrete case. We recall the Birkhoff's ergodic theorem for invariant probabilities \cite{Birkhoff1}.
 \begin{theorem}\label{thm:Birkhoff for measures}
     Let $(\O,\mathcal{F},f,P)$ be a measure-preserving system. Then for any $X\in L^1(\O,\mathcal{F},P)$, 
     \[\lim_{n\to\infty}\frac{1}{n}\sum_{i=0}^{n-1}X(f^i(\o))=\mathbb E_P[X|\mathcal{I}](\o),\text{ for }P\text{-a.s. } \o\in\O,\]
     where $\mathbb E_P[X|\mathcal{I}]$ is the conditional expectation of $X$ with respect to the $\sigma$-algebra $\mathcal{I}$. If in addition, we suppose $P$ is ergodic with respect to $f$, then  $\mathbb E_P[X|\mathcal{I}]=E_P[X]$, $P$-a.s.
 \end{theorem}
 Note that $L^1(\O,\mathcal{F},P)$ is a vector lattice, and by the above theorem, it satisfies that  for $X\in L^1(\O,\mathcal{F},P)$, the following limit functions exist:
	\[\overline{X}:=\limsup_{n\to\infty}\frac{1}{n}\sum_{i=0}^{n-1}U_f^iX=\mathbb E_P[X|\mathcal{I}]=\liminf_{n\to\infty}\frac{1}{n}\sum_{i=0}^{n-1}U_f^iX=:\underline{X},\text{ }P\text{-a.s}.\]
 Thus, $\overline{X},\underline{X}\in L^1(\O,\mathcal{F},P).$
However, when consider the general vector lattices $\mathcal{H}$, there may exist $X\in \mathcal{H}$ such that $\overline{X}$ and $\underline{X}$ 
 are not in $\mathcal{H}$ (see Example \ref{ex:not quasi}).
 
 Now we prove the following nonlinear version of Birkhoff's ergodic theorem for $X\in\mathcal{H}$ with $\overline{X},\underline{X}\in\mathcal{H}$. Indeed, after the proof of the following theorem, we will provide an example to show this condition is necessary.
	\begin{theorem}\label{thm:ergodic theorem for some function invariant}
		Let $(\O,\mathcal{H},U_f,\hat{\mathbb{E}})$ be a regular sublinear expectation system, where $U_f$ is the linear operator induced by a measurable transform $f:\O\to\O$. Then 
		for any  $X\in \mathcal{H}$ with $\overline{X},\underline{X}\in\mathcal{H}$,  there exists $X^*\in\mathcal{I}$  such that 
		\[\lim_{n\to\infty}\frac{1}{n}\sum_{i=0}^{n-1}U_f^iX=X^*,\text{ }\hat{\mathbb{E}}\text{-q.s.}\]
	\end{theorem}
	\begin{proof}
		Suppose to the contrary that $\hat{\mathbb{E}}[|\overline{X}-\underline{X}|]>0$. Thus, there exists $\mathbb{E}\in\T$ such that ${\mathbb{E}}[|\overline{X}-\underline{X}|]>0$.
		Since $|\overline{X}-\underline{X}|\in\mathcal{I}$, using Theorem \ref{thm:existence of inv}, we have that ${\mathbb{E}_{\mathrm{inv}}}[|\overline{X}-\underline{X}|]>0$, where $\mathbb{E}_{\mathrm{inv}}$ is the invariant skeleton of $\mathbb{E}$. Furthermore, Proposition \ref{re:existence of invariant measure} implies that there exists an invariant probability $P_{\mathrm{inv}}$ of $(\O,\sigma(\mathcal{H}))$ such that 
		\begin{equation}\label{eq:big}
			\int |\overline{X}-\underline{X}|dP_{\mathrm{inv}}>0.
		\end{equation}
		By Theorem \ref{thm:Birkhoff for measures}, one has that $\lim_{n\to\infty}\frac{1}{n}\sum_{i=0}^{n-1}U_f^iX$ exists, $P_{\mathrm{inv}}$-a.s.
		In particular, 
		\[\overline{X}=\lim_{n\to\infty}\frac{1}{n}\sum_{i=0}^{n-1}U_f^iX=\underline{X},\text{ }P_{\mathrm{inv}}\text{-a.s.},\]
		a contradiction with \eqref{eq:big}. Let $X^*=\overline{X}$, and we then finish the proof.
	\end{proof}

	\begin{remark}
		If we choose $\mathcal{H}=B_b(\O)$, the condition that $\overline{X},\underline{X}\in \mathcal{H}$ is naturally holds for any $X\in\mathcal{H}$. However, when $\O$ is a complete metric space  and $\mathcal{H}=L^1_{C_b(\O)}$ is the completion of the bounded continuous function space $C_b(\O)$  with the norm $\hat{\mathbb{E}}[|\cdot|]$ (see \cite{Pengbook} for more details), there exists an example that even if $\hat{\mathbb E}$ is ergodic and $\O$ is a compact metric space, there exists $X\in C_b(\O)\subset L^1_{C_b(\O)}$ such that $\overline{X},\underline{X}\notin L^1_{C_b(\O)}$ and the limit $\lim_{n\to\infty}\frac{1}{n}\sum_{i=0}^{n-1}U^iX$ does not exist, $\hat{\mathbb E}$-q.s., which shows our result is sharp. Here we consider the vector lattice $L^1_{C_b(\O)}$ is because it plays an important role in sublinear expectation theory (e.g. the construction of $G$-Brownian motion).
	\end{remark}
		\begin{example}\label{ex:not quasi}
			Let $\O=\{0,1\}^\mathbb{N}$ with a metric $d$, defined by 
   \[d((\o_i)_{i\in\mathbb{N}},(\o'_i)_{i\in\mathbb{N}})=2^{-\min\{i\in\mathbb{N}:\o_i\neq \o'_i\}}\text{ for any }\o=(\o_i)_{i\in\mathbb{N}},\o'=(\o'_i)_{i\in\mathbb{N}}\in\O.\] Under the topology induced by this metric, $\O$ is compact. Define the left-shift map by 
			\[\sigma((\o_i)_{i\in\mathbb{N}})=(\o_{i+1})_{i\in\mathbb{N}}.\]
			Consider $\mathcal{P}=\{\d_\o:\o\in\O\}$. Since $\O$ is compact, it follows that $\mathcal{P}$ is weakly compact, and hence the sublinear expectation  on $C_b(\O)$
			\begin{equation}\label{eq:18.56}
			    \hat{\mathbb{E}}[X]:=\sup_{P\in\mathcal{P}}\int XdP=\sup_{\omega\in\O}X(\o)\text{ for any }X\in C_b(\O)
			\end{equation}
			is regular by Lemma \ref{lem:regualr<=>compact}. Since $\s$ is surjective, we can check $\hat{\mathbb{E}}$ is $U_{\s}$-invariant. It is well known that $(\O,\s)$ is point transitive, i.e., there exists $\o^*\in\O$ such that $\overline{\{\s^n\o^*:n\in\mathbb{N}\}}=\O$ (see e.g. \cite{Walters1982}). By the construction of $\hat{\mathbb E}$, we know that $U_{\s}X=X$, $\hat{\mathbb{E}}$-q.s. is equivalent to that $U_{\s}X(\o)=X(\o)$ for any $\o\in\O$. In particular, 
			\[X(\s^n\o^*)=X(\o^*)\text{ for any }n\in\mathbb{N}.\]
			By the continuity of $X$, we have that $X$ is constant.
			Thus, $\hat{\mathbb{E}}$ is ergodic with respect to $U_{\s}$. By the construction of $\hat{\mathbb{E}}$ and the compactness of $\O$, we can see that $L^1_{C_b(\O)}=C(\O)$. Now we consider two Bernoulli probabilities 
			\[\mu=(P_{(1/2,1/2)})^\mathbb{N},\text{ and }\nu=(P_{(1/3,2/3)})^\mathbb{N},\]
			where $P_{(1/2,1/2)}(0)=P_{(1/2,1/2)}(1)=1/2$, and $P_{(1/3,2/3)}(0)=1/3,P_{(1/3,2/3)}(1)=2/3$. Then  $\mathrm{supp}\mu=\mathrm{supp}\nu=\O$, where $\mathrm{supp}\mu$ is the topological support of  probability $\mu$. 
			
			Now, we prove there exists a continuous function $X\in C_b(\O)$ such that $\overline{X},\underline{X}\notin  L^1_{C_b(\O)}$  and the limit $\lim_{n\to\infty}\frac{1}{n}\sum_{i=0}^{n-1}U_\s^iX$ does not exist, $\hat{\mathbb E}$-q.s. 
        Consider the function $X=\one_{\{\o\in\O:\o_0=0\}}$. Then $X$ is continuous, and $\int Xd\mu\neq \int Xd \nu$. It is well known that $\mu,\nu\in\mathcal{ M}^e(\s)$, as $\mu,\nu$ are Bernoulli probabilities. Thus, there exist Borel sets $A_\mu,A_\nu\in\mathcal{B}(\O)$ such that 
			\begin{enumerate}[(1)]
				\item $\mu(A_{\mu})=\nu(A_{\nu})=1$;
				\smallskip
				\item $A_\mu\cap A_\nu=\emptyset$;
				\smallskip
				\item $\overline{A}_\mu=\overline{A}_\nu=\O$;
				\smallskip
				\item $\overline{X}(\o)=\underline{X}(\o)=\int Xd\mu$ for any $\o\in A_{\mu}$ and $\overline{X}(\o)=\underline{X}(\o)=\int Xd\nu$ for any $\o\in A_{\nu}$.
			\end{enumerate} 
			For any $\o\in A_\nu$, by (3), there exists $\o^n\in A_\mu$ such that $\lim_{n\to\infty}\o^n=\o$, but by (4), $\lim_{n\to\infty}\overline{X}(\o^n)=\int  {X}d\mu\neq \int  {X}d\nu=\overline{X}(\o)$. Thus, $\overline{X}$ is not continuous. Similarly, $\underline{X}$ is not continuous. 

   Finally, we prove that the limit $\lim_{n\to\infty}\frac{1}{n}\sum_{i=0}^{n-1}U_\s^iX$ does not exist, $\hat{\mathbb E}$-q.s. By the construction of $\hat{\E}$ in \eqref{eq:18.56}, the limit $\lim_{n\to\infty}\frac{1}{n}\sum_{i=0}^{n-1}U_\sigma^iX$ exists, $\hat{\mathbb E}$-q.s. if and only if this limit exists at each point in $\O$. Now we prove this is impossible. Consider a point $\o=(\o_i)_{i\in\mathbb N}\in\O$ as follows:
   \begin{equation*}
       \omega_0=0, \omega_1=1, \ \omega_j=
       \begin{cases}
           0, & \text{if } \ 2^n\leq j\leq 2^{n}+2^{n-1}-1,\\
           1, & \text{if } \ 2^n+2^{n-1}\leq j\leq 2^{n+1}-1,
       \end{cases}
       \ \text{ for all } \ n\geq 1.
   \end{equation*}
   Then 
   \[\frac{1}{2^n}\sum_{i=0}^{2^n-1}X(\s^i\omega)=1/2, \ \text{ and } \ \frac{1}{2^n+2^{n-1}}\sum_{i=0}^{2^n+2^{n-1}-1}X(\s^i\omega)=\frac{2}{3}, \ \text{ for all } \ n\geq 2.
   \]
   Thus, $\underline X(\o)\le 1/2<2/3\le \overline X(\o)$, which implies that  the limit $\lim_{n\to\infty}\frac{1}{n}\sum_{i=0}^{n-1}U_\s^iX$ does not exist, $\hat{\mathbb E}$-q.s.
		\end{example}

	Furthermore, we also provide an example to show there exists sublinear expectation system $(\O,\mathcal{H},U,\hat{\E})$ which satisfies that for any $X\in C_b(\O)$, $\overline{X},\underline{X}\in C_b(\O)$.
	\begin{example}\label{ex:not continuous}
		Let $\mathbb{T}^1:=[0,1)$ be the 1-dimensional torus and $R_\a:\mathbb{T}^1\to \mathbb{T}^1$, $\o\mapsto \o+\a$ be the rotation on $\mathbb{T}^1$, where $\a$ is an irrational number. It is well known that there exists a unique  $R_\a$-ergodic probability, that is the Haar measure $m$ on $\mathbb{T}^1$. Denote by $\mathcal{P}$ the set of all Borel probabilities on $\mathbb{T}^1$. Since $\mathbb{T}^1$ is compact, it follows that $\mathcal{P}$ is weakly compact. 
		Consider the sublinear expectation 
		\[\hat{\mathbb{E}}[X]=\sup_{P\in\mathcal{P}}\int X dP,\text{ for any }X\in C(\mathbb{T}^1).\]	
		Then according to Lemma \ref{lem:regualr<=>compact}, $(\Omega,L^1_{C(\mathbb{T}^1)},\hat{\mathbb{E}})$ is a regular sublinear expectation space. Similar to the proof of Example \ref{ex:not quasi}, $L^1_{C(\mathbb{T}^1)}=C(\mathbb{T}^1)$ and this system is ergodic.
		
		Since the Haar measure $m$ is the unique ergodic probability, it follows from \cite[Theorem 6.20]{Walters1982} that for any $X\in C(\mathbb{T}^1)$, 
		\[\overline{X}(\o)=\underline{X}(\o)=\lim_{n\to\infty}\frac{1}{n}\sum_{i=0}^{n-1}X(R_{\alpha}^i(\o))=\int Xdm,\text{ for any }\o\in \mathbb{T}^1.\]
		Thus, $\overline{X}=\underline{X}\in C(\mathbb{T}^1)$.
		
		Note that the above proof can be applied on any  topological dynamical  system $(\O,f)$\footnote{Let $\O$ be a compact metric space, and $f:\O\to \O$ is a continuous map. Then $(\O,f)$ is said to be a topological dynamical system.}, which is uniquely ergodic.
	\end{example}

	In the last of this subsection, we also provide a nonlinear version of Birkhoff's ergodic theorem for continuous time case. 
For convenience, we still use the notations 
\[\overline{X}:=\limsup_{T\to\infty}\frac{1}{T}\int_0^TX\circ \t_tdt,\text{ and }\underline{X}:=\liminf_{T\to\infty}\frac{1}{T}\int_0^TX\circ \t_tdt.\]
\begin{theorem}\label{thm: Birkhoff's ergodic theorem for invaraint flow}
Let $(\O,\mathcal{H},(U_t)_{t\geq 0},\hat{\mathbb{E}})$ be a continuous regular sublinear expectation system, where $U_t$ is the linear operator induced by a semigroup of measurable transformations $\t_t:\O\to\O$ and $\hat{\mathbb{E}}$ is regular. Then 
		for any  $X\in \mathcal{H}$ with $\overline{X},\underline{X}\in\mathcal{H}$,  there exists $X^*\in\mathcal{I}$  such that the limit exists
		\[\lim_{T\to\infty}\frac{1}{T}\int_0^TX\circ \t_tdt=X^*,\text{ }\hat{\mathbb{E}}\text{-q.s.}\]
	\end{theorem}
	\begin{proof}
		Suppose to the contrary that $\hat{\mathbb{E}}[|\overline{X}-\underline{X}|]>0$. Thus, there exists $\mathbb{E}\in\T$ such that ${\mathbb{E}}[|\overline{X}-\underline{X}|]>0$.
		
		Since $|\overline{X}-\underline{X}|\in\mathcal{I}$, using Theorem \ref{thm: existence of invariant skeleton for flow}, we have that ${\mathbb{E}_{\mathrm{inv}}}[|\overline{X}-\underline{X}|]>0$, where $\mathbb{E}_{\mathrm{inv}}$ is the invariant skeleton of $\mathbb{E}$. Furthermore, Proposition \ref{re:existence of invariant measure} implies that there exists an invariant probability $P_{\mathrm{inv}}$ such that 
		\begin{equation}\label{eq:big1}
			\int |\overline{X}-\underline{X}|dP_{\mathrm{inv}}>0.
		\end{equation}
		By Lemma \ref{lem:classical ergodic theorem for flow}, one has that $\lim_{T\to\infty}\frac{1}{T}\int_0^TX\circ \t_tdt$ exists, $P_{\mathrm{inv}}$-a.s.
		In particular, 
		\[\overline{X}=\lim_{T\to\infty}\frac{1}{T}\int_0^TX\circ \t_tdt=\underline{X},\text{ }P_{\mathrm{inv}}\text{-a.s.},\]
	which is	a contradiction with \eqref{eq:big1}. The proof is completed by letting $X^*=\overline{X}$.
	\end{proof}	
	 
	\subsubsection{Ergodic theorem for ergodic sublinear expectation systems}
In this subsection, we further study the sublinear expectation systems with the additional assumption of ergodicity.
	We recall that a sublinear expectation system $(\O,\mathcal{H},U,\hat{\mathbb{E}})$ is said to be ergodic if for any $X\in\mathcal{I}$, $X$ is constant, $\hat{\mathbb{E}}$-q.s. 
	
	With the help of Theorem \ref{thm:ergodic theorem for some function invariant}, we can prove a Birkhoff's ergodic theorem for ergodic sublinear expectation systems.
	\begin{theorem}\label{thm:ergodic theorem for some function}
		Let $(\O,\mathcal{H},U_f,\hat{\mathbb{E}})$ be a regular sublinear expectation system, where $U_f$ is the linear operator induced by a measurable transform $f:\O\to\O$. Then the following two statements are equivalent:
		\begin{enumerate}[(i)]
			\item $\hat{\mathbb{E}}$ is ergodic;
			\item  for any  $X\in \mathcal{H}$ with
			$\overline{X},\underline{X}\in\mathcal{H},$ there exists a constant $c_X\in\mathbb{R}$ such that 
			\[\lim_{n\to\infty}\frac{1}{n}\sum_{i=0}^{n-1}U_f^iX=c_X,\text{ }\hat{\mathbb{E}}\text{-q.s.}\]
		\end{enumerate}
		Moreover, if one of two equivalent statements above holds, then for any invariant probability $P\in\mathcal{P}$ and for any  $X\in \mathcal{H}$ with
			$\overline{X},\underline{X}\in\mathcal{H},$
		\[\int XdP=c_X.\]
	\end{theorem}
	\begin{proof}
		(i) $\Rightarrow$ (ii).  Since  $\overline{X},\underline{X}\in\mathcal{I}$, it follows from the ergodicity of $U_f$  that there exist two constants $\overline{c},\underline{c}\in\mathbb{R}$ such that 
		\begin{equation*}\label{eq:1}
			\overline{X}=\overline{c}\text{ and }\underline{X}=\underline{c},\text{ }\hat{\mathbb{E}}\text{-q.s.}
		\end{equation*}
		By Theorem \ref{thm:ergodic theorem for some function invariant}, we have that 
		\[\overline{X}=\underline{X}=\lim_{n\to\infty}\frac{1}{n}\sum_{i=0}^{n-1}U_f^iX,\text{ }\hat{\mathbb{E}}\text{-q.s.}\]
		Let $c_X:=\overline{c}=\underline{c}$. Then 
  \[\lim_{n\to\infty}\frac{1}{n}\sum_{i=0}^{n-1}U_f^iX=c_X,\text{ }\hat{\mathbb{E}}\text{-q.s.}\]

		(ii) $\Rightarrow$ (i). For any $X\in\mathcal{I}$, one has that
		\[\underline{X}=\overline{X}=X\in\mathcal{H},\]
		which by (ii), implies that $X$ is constant, $\hat{\mathbb{E}}$-q.s. That is, $U_f$ is ergodic.
		
		In this case, by (ii), we have for any $P\in\mathcal{P}$,
		\[\lim_{n\to\infty}\frac{1}{n}\sum_{i=0}^{n-1}U_f^iX=c_X,\text{ }P\text{-a.s.}\]
  Since $\hat{\mathbb E}[|X|]<\infty$, it follows that $X\in L^1(\O,\s(\mathcal{H}),P)$. Thus, by Theorem \ref{thm:Birkhoff for measures}, we deduce that  for any $P\in\mathcal{P}\cap \mathcal{M}(f)$,
		\[\int\mathbb E_P[X| \mathcal{I}]dP= \int\lim_{n\to\infty}\frac{1}{n}\sum_{i=0}^{n-1}U_f^iXdP=c_X, \text{ }P\text{-a.s.}\]
  Thus, $\int XdP=\int \mathbb E_P[X| \mathcal{I}] dP=c_X.$
		The proof is completed.
	\end{proof}
	
	If we further suppose that a sublinear expectation system $(\O,\mathcal{H},U_f,\hat{\mathbb{E}})$ satisfies that $\overline{X}\in\mathcal{H}$ for any $X\in \mathcal{H}$,  then we obtain the following result.
	\begin{corollary}\label{cor:ergodic theorem for T}
		Let $(\O,\mathcal{H},U_f,\hat{\mathbb{E}})$ be a regular sublinear expectation system, where $U_f$ is induced by a measurable transform $f:\O\to\O$. If $U_f$ satisfies that $\overline{X}\in\mathcal{H}$ for any $X\in \mathcal{H}$, then the following two statements are equivalent:
		\begin{enumerate}[(i)]
			\item $\hat{\mathbb{E}}$ is ergodic;
			\item there exists a unique ergodic probability $P\in\mathcal{P}$ such that for any  $X\in \mathcal{H}$,  
			\[\lim_{n\to\infty}\frac{1}{n}\sum_{i=0}^{n-1}U_f^iX=\int XdP, \text{ }\hat{\mathbb{E}}\text{-q.s.}\]
		\end{enumerate}
	\end{corollary}
	\begin{proof}
(ii) $\Rightarrow$ (i). This is a  direct consequence of Theorem \ref{thm:ergodic theorem for some function}.
 
		(i) $\Rightarrow$ (ii).  For any $X\in\mathcal{H}$, since $-X\in\mathcal{H}$, it follows that 
  \[\underline{X}=-(\overline{-X})\in\mathcal{H},\]
  which allows us use Theorem \ref{thm:ergodic theorem for some function}
 to obtain that for any   $P_1,P_2\in\mathcal{M}(f)\cap \mathcal{P}$,
  \[\int XdP_1=\int XdP_2\text{ for any } X\in\mathcal{H}.\]
 Combining the uniqueness in Daniell-Stone theorem (see \cite[Section 4.5]{Dudley2002} or \cite[Theorem B.3.3]{Pengbook}), we have that $P_1=P_2$.
Thus, there exists a unique invariant probability in $\mathcal{P}$, denoted by $P$.

  Now we use the monotone class theorem to prove this invariant probability is ergodic. Denote by
 $B_b(\O,\s(\mathcal{H}))$ the space of all bounded $\s(\mathcal{H})$-measurable functions. Let 
 \[\widetilde{\mathcal{H}}=\Big\{X\in B_b(\O,\s(\mathcal{H})):\lim_{n\to\infty}\Big\|\frac{1}{n}\sum_{i=0}^{n-1}U_f^iX-\int XdP\Big\|_{1,P}=0\Big\},\]
 where $\|\cdot\|_{1,P}$ is the $L^1$-norm with respect to $P$. Then, by dominated convergence theorem and (ii) in Theorem \ref{thm:ergodic theorem for some function}, we have that 
\begin{equation}\label{eq22.31}
    \mathcal{H}\cap B_b(\O,\s(\mathcal{H}))\subset\widetilde{\mathcal{H}}.
\end{equation}
Here $\mathcal{H}\cap B_b(\O,\s(\mathcal{H}))\neq \emptyset$, as all constants are in it.
 
 Recall the $\pi$-system $\mathcal{C}$ defined by \eqref{eq:C}, satisfies $\s(\mathcal{C})=\s(\mathcal{H})$. By the construction of $\mathcal{C}$, for any $A\in\mathcal{C}$, there exists $X_m\in \mathcal{H}_+:=\{X\in\mathcal{H}:X\ge0\}$   such that $X_m\uparrow 
 \one_A$, as $m\to\infty$. By monotone convergence theorem, for any $\e>0$, there exists $M>0$ such that for any $m\ge M$, 
 \begin{equation}\label{eq:3.820.50}
\|X_m-\one_A\|_{1,P}<\e/4.     
 \end{equation}
 Meanwhile, as $X_M\in \mathcal{H}$, and $0\le X_M\le \one_A$, by \eqref{eq22.31}, we have that $X_M\in\widetilde{\mathcal{H}}$. So there exists $N_M>0$ such that for any $n\ge N_M$, 
 \begin{equation}\label{eq:3.820.53}
   \Big\|\frac{1}{n}\sum_{i=0}^{n-1}U_f^iX_M-\int X_MdP\Big\|_{1,P}<\e/4.
 \end{equation}
Combining \eqref{eq:3.820.50} and \eqref{eq:3.820.53}, we have for any $n\ge N_M$,
\begin{align*}
    \Big\|\frac{1}{n}\sum_{i=0}^{n-1}U_f^i\one_A-\int \one_AdP\Big\|_{1,P}\le& \Big\|\frac{1}{n}\sum_{i=0}^{n-1}U_f^i\one_A-\frac{1}{n}\sum_{i=0}^{n-1}U_f^iX_M\Big\|_{1,P}
    \\&+\Big\|\frac{1}{n}\sum_{i=0}^{n-1}U_f^iX_M-\int X_MdP\Big\|_{1,P}+\Big|\int \one_AdP-\int X_MdP\Big|\\
    \le& 2\|\one_A-X_M\|_{1,P}+\e/4<\e.
\end{align*}
Thus, $\one_A\in\widetilde{\mathcal{H}}$ for any $A\in\mathcal{C}$. By a similar argument, we can also prove that for any $X\in B_b(\O,\s(\mathcal{H}))$, if there exists $X_m\in\widetilde{\mathcal{H}}$ such that $X_m\uparrow X$, as $m\to\infty$, then $X\in\widetilde{\mathcal{H}}$. Furthermore, it is easy to check that for any $g,h\in\widetilde{\mathcal{H}}$ and $\a\in\mathbb R$, 
\[g+h\in\widetilde{\mathcal{H}}\text{ and }\a g\in\widetilde{\mathcal{H}}.\]
Now we have proven $\widetilde{\mathcal{H}}$ satisfies all conditions in monotone class theorem. Thus, $B_b(\O,\s(\mathcal{H}))\subset\widetilde{\mathcal{H}}$. In particular, for any $A\in\s(\mathcal{H})$ with $f^{-1}A=A$, i.e., $A\in \mathcal{I}$, we deduce that 
\[\|P(A)-\one_A\|_{1,P}=0,\]
which implies that $P(A)\in \{0,1\}$. As $A\in\mathcal{I}$ is arbitrary, we have $P$ is ergodic. The proof is completed.
	\end{proof}

	\begin{remark}\label{rem:20.29}
 The Birkhoff's ergodic theorem in the nonlinear setting has been discussed in many papers \cite{CMM2016,FHLZ2023,FWZ2020,FengZhao2021}.
		If we consider a measurable space $(\O,\mathcal{F})$ and $\mathcal{H}=B_b(\O)$, this corollary implies that  Theorem 3.3 in \cite{FHLZ2023}, and hence also implies that the Birkhoff's ergodic theorem obtained in  \cite{CMM2016,FWZ2020,FengZhao2021}. 
  
 However, the sublinear expectations generated by $G$-Brownian motion are not  regular on $B_b(\O)$, as $\mathcal{P}$ of $G$-Brownian motion has infinitely many singular probabilities (see \cite{Peng2007} for more details). Indeed, we will prove there exists uncountably many strongly mixing probabilities in the $\mathcal{P}$ of $G$-Brownian motion (see Theorem \ref{thm:uncountable many mixing measure}). Thus, their results in  \cite{CMM2016,FHLZ2023,FWZ2020,FengZhao2021} can not be used to investigate the ergodicity of $G$-Brownian motion or the solution of nonlinear SDE forced by $G$-Brownian motion. Note that our results hold for any regular sublinear expectation systems. In particular, if we consider $\O$ is a complete metric space and $\mathcal{H}=C_b(\O)$, the regularity is  equivalent to the weakly compactness of $\mathcal{P}$  (see Lemma \ref{lem:regualr<=>compact}). Thus, our results can be applied to $G$-Brownian motion or the solution of nonlinear SDE forced by $G$-Brownian motion. 
	
 Furthermore, we provide another classical example, in ergodic theory, which is regular on $C_b(\mathbb{T}^1)$, but it is not regular on $B_b(\mathbb{T}^1)$. We will show that $\hat{\mathbb E}[X]=\sup_{P\in\mathcal{P}}\int XdP$ in Example \ref{ex:not continuous} is not regular on $B_b(\mathbb{T}^1)$.  Indeed, we consider the sequence of subsets $\{(0,1/n)\}_{n=1}^\infty$. Then $\one_{(0,1/n)}\downarrow0$, as $n\to\infty$. For any $n\in\mathbb{N}$, we define a new probability by 
		\[m_n(A)=\frac{m(A\cap(0,1/n))}{m((0,1/n))}\text{ for any }A\in\mathcal{B}(\mathbb{T}^1).\]
		However, $\hat{\mathbb E} [\one_{(0,1/n)}]\ge m_n((0,1/n))=1$ for each $n\in\mathbb{N}$, and hence $\lim_{n\to\infty}\hat{\mathbb E} [\one_{(0,1/n)}]=1$. Thus, $\hat{\mathbb E}$ is not regular on $B_b(\mathbb{T}^1)$.
	\end{remark}
The following result is a direct consequence of Corollary \ref{cor:ergodic theorem for T}.
 \begin{corollary}\label{cor:P unique}
     	Let $(\O,\mathcal{H},U_f,\hat{\mathbb{E}})$ be an ergodic regular sublinear expectation system, where $U_f$ is induced by a measurable transform $f:\O\to\O$. If $U_f$ satisfies that $\overline{X}\in\mathcal{H}$ for any $X\in \mathcal{H}$, then there exists a probability   $P$  such that
     \[\mathcal{M}(f)\cap \mathcal{P}=\mathcal{M}^e(f)\cap \mathcal{P}=\{P\}.\]
     In particular, if $|\mathcal{M}(f)\cap \mathcal{P}|\neq 1$, then there exists $X\in \mathcal{H}$ such that $\overline{X}\notin\mathcal{H}$.
 \end{corollary}
 \begin{remark} Similar to the discussion in Remark \ref{rem:20.29}, Corollary \ref{cor:P unique} implies that if $\mathcal{H}=B_b(\O)$ and $(\O,\mathcal{H},U_f,\hat{\mathbb{E}})$ is ergodic, then 
  \[\mathcal{M}(f)\cap \mathcal{P}=\mathcal{M}^e(f)\cap \mathcal{P}=\{P\}.\]
However, for a general regular ergodic sublinear expectation system  $(\O,\mathcal{H},U_f,\hat{\mathbb{E}})$, the cardinality of the set $\mathcal{M}^e(f)\cap \mathcal{P}$ can be $0$, any $n\in\mathbb{N}$, countably infinite, and uncountably infinite. See the following examples.
\begin{example}
    Let $(\O,f)$ be a topological dynamical system  with an invariant probability $P$ which is not ergodic. Suppose that $(\O,f)$ is minimal, i.e., for any $\o\in \O$, the set $\{f^n(\o):n\in\mathbb{N}\}$ is dense in $\O$ (e.g. Auslander systems \cite{HaddadJohnson}). Consider the lattice vector
$\mathcal{H}=C_b(\O)$ and $\hat{\mathbb{E}}[X]=\int X dP$. Then it is easy to check that $(\O,\mathcal{H},\hat{\mathbb{E}},U_f)$ is a sublinear expectation system. Now we prove it is ergodic. Indeed, for any $X\in\mathcal{H}$ with $U_fX=X$, $\hat{\mathbb{E}}$-q.s. In particular, there exists $\o_0\in \O$ such that  $X(f(\o_0))=X(\o_0)$. Thus,  for any $\o\in \O$, if $\o=f^n(\o_0)$ for some $n\in\mathbb{N}$, then  $X(\o)=X(f^n(\o_0))=X(\o_0)$, and if $\o\notin \{f^n(\o_0):n\in\mathbb{N}\}$, there exists an increasing sequence $\{n_i\}_{i=1}^\infty$ such that $\lim_{i\to\infty}f^{n_i}(\o_0)=\o$, which together with the fact that $X$ is continuous, implies that $X(\o)=X(\o_0)$. Thus, $X$ is constant, and hence $\hat{\E}$ is ergodic. But there is no ergodic probability in its $\mathcal{P}$, as $\mathcal{P}=\{P\}$.

We may assume that $(\O,f)$ be a minimal topological dynamical systems with uncountably many ergodic probabilities (e.g. Auslander systems \cite[Proof of Theorem 4.1]{HaddadJohnson}). Let $\mathcal{P}$ be a subset of $\mathcal{M}^e(f)$ with the cardinality $n\in\mathbb{N}$, countable infinity or uncountable infinity.  Consider the lattice vector
$\mathcal{H}=C_b(\O)$ and $\hat{\mathbb{E}}[X]=\sup_{P\in\mathcal{P}}\int X dP$. Similarly, we can prove $(\O,\mathcal{H},\hat{\mathbb{E}},U_f)$ is ergodic.
\end{example}
\begin{example}
    It is noteworthy that we will later demonstrate that sublinear expectation systems generated by $G$-Brownian motion and a class of  $G$-SDEs are ergodic (Section \ref{sec: e.g. G-BM} and Section \ref{subsec:application to GSDE}), and we will investigate the structures of their $\mathcal{P}$. In particular, we will prove that there are uncountably many ergodic measures in their $\mathcal{P}$, respectively (see Theorem \ref{thm:uncountable many mixing measure} and Theorem \ref{thm: infinitely many ergodic probablities in P}).
\end{example}
  \end{remark}

 By an argument similar to that of the  discrete time case, we have the corresponding result for continuous sublinear expectation system.
 \begin{corollary}\label{cor:ergodic theorem for T, continuous}
		Let $(\O,\mathcal{H},(U_t)_{t\geq 0},\hat{\mathbb{E}})$ be a regular continuous sublinear expectation system, where $U_t$ is the linear operator induced by a semigroup of measurable transformations $\t_t:\O\to\O$. If $(U_t)_{t\geq 0}$ satisfies that $\overline{X}\in\mathcal{H}$, for any $X\in \mathcal{H}$. Then the following two statements are equivalent:
		\begin{enumerate}[(i)]
			\item $\hat{\mathbb{E}}$ is ergodic;
			\item there exists an ergodic probability $P$ on $\s(\mathcal{H})$ such that for any  $X\in \mathcal{H}$,  
			\[\lim_{T\to\infty}\frac{1}{T}\int_0^TX\circ \t_tdt=\int XdP, \text{ }\hat{\mathbb{E}}\text{-q.s.}\]
		\end{enumerate}
		In this case,  $P$ is the unique  ergodic probability with respect to $\theta_t, \ t\geq 0$ in $\mathcal{P}$.
	\end{corollary}
	
	\subsection{Ergodicity of $G$-Brownian motion}\label{sec: e.g. G-BM}
In this subsection, we will use the definition of ergodicity to check that $G$-Brownian motion is ergodic. In fact, in the subsequent section, we will prove a  stronger result, namely that $G$-Brownian motion is mixing (and hence ergodic). We use two methods to prove this because in the proof by definition of ergodicity, we obtain an inequality, which will also be of interest in the study of nonlinear expectations themselves (see Lemma \ref{lemma of inequality of the variant}).

Following  ideas in \cite[Definition 3.12]{Pengbook} for one-sided $G$-Brownian motion, we introduce the definition and construction of a two-sided $G$-Brownian motion.
\begin{example}[$G$-Brownian motion \cite{Pengbook}]\label{Example:G-BM}
Denote by $\Omega=C_0(\mathbb{R};\mathbb{R}^d)$ the space of all $\mathbb{R}^d$-valued continuous paths $(\omega_t)_{t\in \mathbb{R}}$, with $\omega_0=0$, equipped with the distance 
	$$\rho(\omega^1, \omega^2):=\sum_{i=1}^{\infty}2^{-i}[(\max_{t\in[-i,i]}|\omega^1_t-\omega^2_t|)\wedge 1].$$
 We consider the canonical process $B_t(\omega)=\omega_t, t\in \mathbb R,$ for all $\omega\in \Omega$.
	For any interval $I\subset \mathbb R$, set 
	$$Lip(\Omega_I):=\{\phi(B_{t_1}-B_{t'_1},\cdots, B_{t_n}-B_{t'_{n}}): n\geq 1, \{t_i,t'_i\}_{i=1}^n\subset I, \phi\in C_{l,lip}(\mathbb{R}^{d\times n}) \},$$
 where $C_{l,lip}(\mathbb R^{d\times n})$ is the collection of all local Lipschitz functions with polynomial growth, i.e., $\phi: \mathbb R^{d\times n}\to \mathbb R$ satisfying there exist $L>0, k\geq 1$ such that
    \[
    |\phi(x)-\phi(y)|\leq L(1+|x|^k+|y|^k)|x-y|.
    \]
    Note that there exist $r_0<r_1<\cdots<r_m$ such that
    $\{t_1,t_1',t_2,t_2'\cdots,t_n,t_n'\}=\{r_0,r_1,\cdots,r_m\}$, and hence for any $\phi\in C_{l,lip}(\mathbb{R}^{d\times n})$ there exists $\varphi\in C_{l,lip}(\mathbb{R}^{d\times m})$ such that 
    \[
    \phi(B_{t_1}-B_{t'_1},\cdots, B_{t_n}-B_{t'_{n}})=\varphi(B_{r_1}-B_{r_0},\cdots, B_{r_m}-B_{r_{m-1}}).
    \]
    Hence, $Lip(\Omega_I)$ also has the following characterization:
    \begin{equation}\label{eq:charact-of-Lip(Omega)}
    \begin{split}
           Lip(\Omega_I)=\big\{\phi(B_{t_1}-B_{t_0},\cdots, B_{t_n}-&B_{t_{n-1}}): \\
           &n\geq 1, t_0<\cdots<t_n, \{t_i\}_{i=0}^n\subset I,  \phi\in C_{l,lip}(\mathbb{R}^{d\times n})\big\}.
    \end{split}
    \end{equation}
	It is clear that $Lip(\Omega_{I_1})\subset Lip(\Omega_{I_2})$ if $I_1\subset I_2$. We also denote $Lip\big(\Omega_{(-\infty,t]}\big)$ and $Lip(\Omega_{\mathbb R})$ by $Lip(\Omega_t)$ and $Lip(\Omega)$ respectively.

 Let $(\xi_i)_{i=1}^{\infty}$ be a sequence of $d$-dimensional random vectors  on some sublinear expectation space $(\tilde{\Omega},\tilde{\mathcal{H}}, \tilde{\mathbb{E}})$ 
 such that $\xi_i$ is $G$-normal distributed \footnote{ A $d$-dimensional random vector $X=(X_1,\cdots,X_d)$ on a sublinear expectation space $(\tilde{\Omega},\tilde{\mathcal{H}}, \tilde{\mathbb{E}})$ is called $G$-normally distributed if 
    \[aX+bY\overset{d}{=}\sqrt{a^2+b^2}X,\text{ for }a,b\ge0,\]
    where $Y$ is an independent copy of $X$.} and $\xi_{i+1}$ is independent from $(\xi_1, \cdots, \xi_i)$ for each $i\in\mathbb N$ (the existence of this sequence can be guaranteed in \cite[Section 1.3 and Section 2.2]{Pengbook}). Then we define a sublinear expectation $\hat{\mathbb{E}}$ on $Lip(\Omega)$ via the following procedure: for each $X\in Lip(\Omega)$ with 
	$$X=\phi(B_{t_1}-B_{t_0}, B_{t_2}-B_{t_1},\cdots, B_{t_n}-B_{t_{n-1}})$$
	for some $\phi\in C_{l,lip}(\mathbb{R}^{d\times n})$ and $t_0<t_1<\cdots<t_n$, set
	\begin{equation}
		\label{definition of G-expectation}
		\begin{split}
		     \hat{\mathbb{E}}[\phi(B_{t_1}-B_{t_0}, B_{t_2}-B_{t_1},\cdots&, B_{t_n}-B_{t_{n-1}})]\\
      &:= \tilde{\mathbb{E}}[\phi(\sqrt{t_1-t_0}\xi_1, \sqrt{t_2-t_1}\xi_2,\cdots, \sqrt{t_n-t_{n-1}}\xi_n)].
		\end{split}
	\end{equation}

		The sublinear expectation $\hat{\mathbb{E}}:Lip(\Omega)\rightarrow \mathbb{R}$ defined through above is called a $G$-expectation. The corresponding canonical process $(B_t)_{t\in \mathbb R}$ on the sublinear expectation space $(\Omega, Lip(\Omega), \hat{\mathbb{E}})$, is called a two-sided $G$-Brownian motion, satisfying the following conditions
  \begin{enumerate}[(i)]
      \item $B_0(\o)=0$;
      \item for any $t\in \mathbb R,s\ge 0$,  $B_{t+s}-B_t$ and $B_s$ are identically distributed, and  $B_{t+s}-B_t$ is independent from $(B_{t_1}-B_{t'_1},\cdots,B_{t_n}-B_{t'_n})$ for each $n\in\mathbb N$ and $\{t_i,t'_i\}_{i=1}^n\subset (-\infty,t]$;
      \item $\lim_{|t|\downarrow0}\hat{\mathbb E}[|B_t|^3]|t|^{-1}=0$;
      \item $\hat{\mathbb E}[B_t]=\hat{\mathbb E}[-B_t]=0$.
  \end{enumerate}

\end{example}

	We denote by $L_G^p(\Omega_I), p\geq 1$, the completion of $Lip(\Omega_I)$ under the norm $\|X\|_p:=(\hat{\mathbb{E}}[|X|^p])^{1/p}$ for any non-empty interval $I\subset \mathbb R$. It is clear that $L_G^p(\Omega_{I_1})\subset L_G^p(\Omega_{I_2})$ if any non-empty intervals $I_1\subset I_2\subset \mathbb{R}$ and $L_G^p(\Omega_{I_2})$ is independent from $L_G^p(\Omega_{I_1})$ if $\sup I_1\leq \inf I_2$.

 Combining Theorem 6.1.35, Theorem 6.2.5 and Proposition 6.3.2 in \cite{Pengbook}, we have the following regularity result.
\begin{proposition}\label{prop:Gregular}
   The sublinear expectation space  $(\Omega, L_G^p(\Omega), \hat{\mathbb{E}})$ is regular for each $p\ge 1$.
\end{proposition}
	
Now we define the transformation $\theta: \mathbb{R}\times \Omega\rightarrow \Omega$ ($G$-Brownian shift) by 
	$$\theta_t\omega(\cdot):=\omega(\cdot+t)-\omega(t).$$
 It is easy to see that $\{\theta_t\}_{t\in \mathbb R}$ is a group with $\theta_0=Id$ being the identical map. Hence, $\theta_t$ is invertible with $\theta_t^{-1}=\theta_{-t}$ for all $t\in \mathbb R$. For $t\in \mathbb R$ and $p\geq 1$, consider the map $U_t: L_G^p(\Omega)\rightarrow L_G^p(\Omega)$  induced by $\theta_t$,
	$$U_tX:=X\circ\theta_t, \text{ for all }  X\in L_G^p(\Omega).$$
 More precisely, $U_t: L_G^p(\Omega_I)\rightarrow L_G^p(\Omega_{t+I})$ for all interval $I$, where $t+I=\{t+r: r\in I\}$.
	By the definition of $\theta_t$ and equation (\ref{definition of G-expectation}), we can easily deduce that for all $p\geq 1$ and $t\in \mathbb R$, $U_t$ preserves the $G$-expectation $\hat{\mathbb{E}}$, i.e., 
	$$\hat{\mathbb{E}}[U_tX]=\hat{\mathbb{E}}[X], \text{ for all } X\in L_G^p(\Omega).$$

 \begin{lemma}
  For any $p\ge1$,   $(\Omega, L_G^p(\Omega), (U_{t})_{t\in \mathbb R}, \hat{\mathbb{E}})$ is a continuous sublinear expectation system.
 \end{lemma}
 \begin{proof}
     According to the discussion above, we only need to show that
     \begin{equation}\label{0614-2}
         \lim_{|t|\to 0}\|U_tX-X\|_1=0, \ \text{ for all } \ X\in L_G^p(\Omega).
     \end{equation}
     Indeed, fix $X\in L_G^p(\Omega)$. Then  by \eqref{eq:charact-of-Lip(Omega)}, for any $\epsilon>0$, there exists $X_{\epsilon}\in Lip(\Omega)$ with
     \[
     X_{\epsilon}=\phi(B_{t_1}-B_{t_0}, B_{t_2}-B_{t_1},\cdots, B_{t_n}-B_{t_{n-1}}) 
     \]
   for some $\phi\in C_{l,lip}(\mathbb{R}^{d\times n})$ and $t_0<t_1<\cdots<t_n$ such that $\hat{\mathbb{E}}[|X-X_{\epsilon}|^p]<\epsilon$. Note that 
     there exist $L>0,k>0$ such that
     \[
     |\phi(x)-\phi(y)|\leq L(1+|x|^k+|y|^k)|x-y|, \ \text{ for all } \ x,y\in \mathbb{R}^n,
     \]
     and
     \[
     U_tX_{\epsilon}=\phi(B_{t+t_1}-B_{t+t_0}, B_{t+t_2}-B_{t+t_1},\cdots, B_{t+t_n}-B_{t+t_{n-1}})\text{ for any }t\in\mathbb{R}.
     \]
     Thus, by H$\ddot{\mathrm{o}}$lder inequality and \eqref{definition of G-expectation}, for any $t\in \mathbb{R}$,
     \begin{equation}\label{eq:remove-0827-1}
         \begin{split}
             \hat{\mathbb{E}}[|U_tX_{\epsilon}-X_{\epsilon}|^p]&\leq C_{n,p,L,k}\Big(1+\sum_{i=0}^{n-1}\big(\hat{\mathbb{E}}[|B_{t_{i+1}}-B_{t_i}|^{2kp}]+\hat{\mathbb{E}}[|B_{t+t_{i+1}}-B_{t+t_i}|^{2kp}]\big)^{\frac{1}{2}}\Big)\\
             &\qquad\qquad\qquad\qquad\qquad\qquad\qquad\qquad\qquad \times\sum_{i=0}^{n}\big(\hat{\mathbb{E}}[|B_{t+t_i}-B_{t_i}|^{2p}]\big)^{\frac{1}{2}}\\
             &\leq C_{n,p,L,k}\Big(1+\sum_{i=0}^{n-1}|t_{i+1}-t_i|^{\frac{kp}{2}}\big(\tilde{\mathbb{E}}[|\xi_1|^{2kp}]\big)^{\frac{1}{2}}\Big)\big(\tilde{\mathbb{E}}[|\xi_1|^{2p}]\big)^{\frac{1}{2}}|t|^{\frac{p}{2}},
         \end{split}
     \end{equation}
     where $\xi_1$ and $\tilde{E}[\cdot]$ is from  \eqref{definition of G-expectation} and $C_{\alpha}$  denotes a constant depending only on parameters set $\alpha$ (here, $\alpha=\{n,p,L,k\}$), which may differ from line to line.  
    On the other hand, by \cite[Proposition 2.2.15]{Pengbook} and the fact that $\xi_1$ is $G$-normal distributed, we know that 
     \begin{equation}\label{eq:14.05}
         \tilde{\mathbb{E}}[|\xi_1|^{p}]\leq C_p, \ \text{ for all } \ p\geq 1.
     \end{equation}
     Then it follows from \eqref{eq:remove-0827-1} and \eqref{eq:14.05} that
     \begin{equation*}
         \lim_{|t|\to 0}\hat{\mathbb{E}}[|U_tX_{\epsilon}-X_{\epsilon}|^p]=0
     \end{equation*}
     and hence
     \begin{equation*}
         \begin{split}
             \limsup_{|t|\to 0}\hat{\mathbb{E}}[|U_tX-X|^p]&\leq \limsup_{|t|\to 0}C_{p}\big(\hat{\mathbb{E}}[|U_tX-U_tX_{\epsilon}|^p]+\hat{\mathbb{E}}[|U_tX_{\epsilon}-X_{\epsilon}|^p]+\hat{\mathbb{E}}[|X_{\epsilon}-X|^p]\big)\\
             &\leq C_p\epsilon.
         \end{split}
     \end{equation*}
     Since $\epsilon>0$ is arbitrary, we conclude that $\lim_{|t|\to 0}\|U_tX-X\|_p=0$. The proof of \eqref{0614-2} is finished, as $\|X\|_1\leq \|X\|_p$ for any $p\geq 1$ and $X\in L_G^p(\Omega)$.
 \end{proof}
 \begin{remark}
     Actually, we have shown that $U_tX\to X$ as $t\to 0$ in $L_G^p(\Omega)$. Note that $U_t: L_G^p(\Omega)\to L_G^p(\Omega)$ is a contraction operator for all $t\in \mathbb R$, i.e., $U_t$ is a bounded linear operator and $\|U_t\|:=\sup_{\{X\in L_G^p(\Omega): \|X\|_p=1\}}\|U_tX\|\leq 1$. Then $\{U_t\}_{t\in \mathbb R}$ is a $C_0$-group on Banach space $L_G^p(\Omega)$, i.e.,
     \[
     U_0=Id, \ U_{t+s}=U_t\circ U_s, \text{ for all } t,s\in \mathbb R, \ \text{ and } \ U: \mathbb R\times L_G^p(\Omega)\to L_G^p(\Omega) \text{ is continuous.}
     \]
 \end{remark}

 Then we have the following theorem.
	
	\begin{theorem}
		\label{ergodicity of G-Brownian motion}
		For all $\tau\neq0$, the $G$-expectation system $(\Omega, L_G^2(\Omega), U_{\tau}, \hat{\mathbb{E}})$ is ergodic, and hence $(\Omega, L_G^2(\Omega), (U_{t})_{t\geq 0}, \hat{\mathbb{E}})$ and $(\Omega, L_G^2(\Omega), (U_{t})_{t\in \mathbb R}, \hat{\mathbb{E}})$ are ergodic.
	\end{theorem}
	
	Before the proof of this theorem, we first give the following lemma. 
	
	\begin{lemma}\label{lemma of inequality of the variant}
		Let $(\Omega, \mathcal{H}, \hat{\mathbb{E}})$ be a sublinear expectation space. 
        Then  for any $X,Y\in \mathcal{H}^n$ for some $n\geq 1$ such that $X,Y$ are bounded and $Y$ is a copy of $X$, i.e., $Y$ is independent from $X$ and $X, Y$ are identically distributed, we have
		\begin{equation}
			\label{inequality of the variant}
			\hat{\mathbb{E}}[|X-\hat{\mathbb{E}}[X]|^2]\leq \hat{\mathbb{E}}[|X-Y|^2].
		\end{equation}
        If moreover $(\Omega, \mathcal{H}, \hat{\mathbb{E}})$ is regular, then \eqref{inequality of the variant} holds for all $X,Y\in \mathcal{H}$ such that $|X|^2,|Y|^2, |X-Y|^2\in \mathcal{H}$ and $Y$ is a copy of $X$.
	\end{lemma}
	
	\begin{proof}
		We only prove the case $n=1$, i.e., $X,Y\in \mathcal{H}$ for convenience, and the general cases of dimension $n>1$ are similar.
        
        We first claim that for a non-negative bounded $X$, equation (\ref{inequality of the variant}) holds. Since $X,Y$ are bounded, i.e., $|X|\vee|Y|\leq M$ for some $M>0$, we know that $|X|^2=\phi(X),|Y|^2=\phi(Y), |X-Y|^2=\phi(X-Y)$ are in $\mathcal{H}$ with $\phi(x):=(|x|\wedge 2M)^2\in C_{b,lip}(\mathbb{R}^n)$. When $X\geq 0$, $\hat{\mathbb{E}}$-q.s., one has
		\begin{equation}
			\begin{split}
				\hat{\mathbb{E}}[|X-Y|^2]&=\hat{\mathbb{E}}[|X|^2-2XY+|Y|^2]\\
				&=\hat{\mathbb{E}}[\hat{\mathbb{E}}[|x|^2+|Y|^2-2xY]|_{x=X}]\\
				&\geq \hat{\mathbb{E}}[(|x|^2+\hat{\mathbb{E}}[|Y|^2]-\hat{\mathbb{E}}[2xY])|_{x=X}]\\
				&=\hat{\mathbb{E}}[|X|^2+\hat{\mathbb{E}}[|Y|^2]-(2X^+\hat{\mathbb{E}}[Y]+2X^-\hat{\mathbb{E}}[-Y])]\\
				&=\hat{\mathbb{E}}[|X|^2+\hat{\mathbb{E}}[|X|^2]-2X\hat{\mathbb{E}}[X]]\\
				&\geq \hat{\mathbb{E}}[|X|^2+(\hat{\mathbb{E}}[X])^2-2X\hat{\mathbb{E}}[X]]\\
				&=\hat{\mathbb{E}}[|X-\hat{\mathbb{E}}[X]|^2].
			\end{split}
		\end{equation}
  
        Then we will show that \eqref{inequality of the variant} holds for bounded $X,Y\in \mathcal{H}$. Note that there exists $M>0$ such that $|X|\leq M$, $\hat{\mathbb{E}}$-q.s. Now let
        $\tilde{X}=X+M$ and $\tilde{Y}=Y+M$. Then $\tilde{Y}$ is also a copy of $\tilde{X}$ and $\tilde{X}$ is non-negative. Thus by the above discussion,
		$$\hat{\mathbb{E}}[|\tilde{X}-\hat{\mathbb{E}}[\tilde{X}]|^2]\leq \hat{\mathbb{E}}{|\tilde{X}-\tilde{Y}|^2}.$$
        Since $\tilde{X}-\hat{\mathbb{E}}[\tilde{X}]=X-\hat{\mathbb{E}}[X]$ and $\tilde{X}-\tilde{Y}=X-Y$, then we derive \eqref{inequality of the variant}.
        
        Finally, if $(\Omega, \mathcal{H}, \hat{\mathbb{E}})$ is regular, we will show that \eqref{inequality of the variant} holds for general $X,Y\in \mathcal{H}$ with $|X|^2,|Y|^2,|X-Y|^2\in \mathcal{H}$. For any $M>0$, let
        \[
        X_M:=X\wedge M\vee (-M), \ \text{ and } \ Y_M:=Y\wedge M\vee (-M).
        \]
        Then $X_M,Y_M\in \mathcal{H}$ and hence $|X_M|^2, |Y_M|^2\in \mathcal{H}$. Notice that
        \[
        |X-c|^2=|X|^2-2cX+|c|^2, \ \ |X-X_M|^2=|X|^2-2X_MX+|X_M|^2,
        \]
        and $X_MX=\phi(X)\in \mathcal{H}$ with $\phi(x)=(x\wedge M\vee (-M))x\in C_{lip}(\mathbb{R})$.
        Then we conclude that $|X-c|^2, |X-X_M|^2\in \mathcal{H}$ for all $c\in \mathbb R, M>0$. Since $Y$ is a copy of $X$, it is easy to check that $Y_M$ is a copy of $X_M$ for all $M>0$. Hence, by the bounded case of \eqref{inequality of the variant},
        \begin{equation}\label{0603-1}
            \hat{\mathbb{E}}[|X_M-\hat{\mathbb{E}}[X_M]|^2]\leq \hat{\mathbb{E}}[|X_M-Y_M|^2], \ \text{ for all } \ M>0.
        \end{equation}
        On the other hand,
        \begin{equation}\label{0603-2}
            \begin{split}
                &\ \ \ \ \big|\hat{\mathbb{E}}[|X-\hat{\mathbb{E}}[X]|^2]-\hat{\mathbb{E}}[|X_M-\hat{\mathbb{E}}[X_M]|^2]\big|\\
                &\leq \hat{\mathbb{E}}\big[\big||X-\hat{\mathbb{E}}[X]|^2-|X_M-\hat{\mathbb{E}}[X_M]|^2\big|\big]\\
                &\leq \hat{\mathbb{E}}\big[\big(|X-X_M|+\hat{\mathbb{E}}[|X-X_M|]\big)\big(|X|+|X_M|+\hat{\mathbb{E}}[|X|]+\hat{\mathbb{E}}[|X_M|]\big)\big]\\
                &\leq 8\big(\hat{\mathbb{E}}[|X|^2]\big)^{\frac{1}{2}}\big(\hat{\mathbb{E}}[|X-X_M|^2]\big)^{\frac{1}{2}},
            \end{split}
        \end{equation}
        and similarly
        \begin{equation}\label{0603-3}
            \big|\hat{\mathbb{E}}[|X-Y|^2]-\hat{\mathbb{E}}[|X_M-Y_M|^2]\big|\leq 8\big(\hat{\mathbb{E}}[|X|^2]\big)^{\frac{1}{2}}\big(\hat{\mathbb{E}}[|X-X_M|^2]\big)^{\frac{1}{2}}.
        \end{equation}
        Note that $|X-X_M|^2\downarrow 0$ as $M\to\infty$, the regularity of $\hat{\mathbb E}$ on $\mathcal{H}$ gives
        \begin{equation}\label{0603-4}
            \lim_{M\to \infty}\hat{\mathbb{E}}[|X-X_M|^2]=0.
        \end{equation}
        Then \eqref{inequality of the variant} follows from \eqref{0603-1}-\eqref{0603-4}.
	\end{proof}

	Now we give the proof of Theorem \ref{ergodicity of G-Brownian motion}.
	
	\begin{proof}[Proof of Theorem \ref{ergodicity of G-Brownian motion}]
		By Remark \ref{Ergodicity of U implies ergodicity of U_t}, we only prove that $(\Omega, L_G^2(\Omega), U_{\tau}, \hat{\mathbb{E}})$ is ergodic for $\tau>0$ (since $U_{-\tau}=U_{\tau}^{-1}$). It is sufficient to show that for any $X\in \mathcal{I}$, i.e., $U_{\tau}X=X$, $\hat{\mathbb{E}}$-q.s., then $X=\hat{\mathbb{E}}[X]$, $\hat{\mathbb{E}}$-q.s.
		
		Now fix $X\in \mathcal{I}\subset L_G^2(\Omega)$. Since $L_G^2(\Omega)$ is the completion of $Lip(\Omega)$ under norm $\|\cdot\|_2$, then  by \eqref{eq:charact-of-Lip(Omega)} for any $\epsilon>0$, there 
		exists 
     \[
     X_{\epsilon}=\phi(B_{t_1}-B_{t_0}, B_{t_2}-B_{t_1},\cdots, B_{t_n}-B_{t_{n-1}})
     \]
     for some $\phi\in C_{l,lip}(\mathbb{R}^{d\times n})$ and $t_0<t_1<\cdots<t_n$ such that $\hat{\mathbb{E}}[|X-X_{\epsilon}|^2]<\epsilon$. Note that 
     \[
     U_{\tau}^mX_{\epsilon}=U_{m\tau}X_{\epsilon}=\phi(B_{m\tau+t_1}-B_{m\tau+t_0}, B_{m\tau+t_2}-B_{m\tau+t_1},\cdots, B_{m\tau+t_n}-B_{m\tau+t_{n-1}}).
     \]
     Hence, $U_{\tau}^mX_{\epsilon}$ is independent from $X_{\epsilon}$ for all $m\in\mathbb{N}$ with $m\geq \frac{t_n-t_0}{\tau}$. Choose $M\in\mathbb{N}$ with $M\geq \frac{t_n-t_0}{\tau}$. Since $U_{\tau}$ preserves $G$-expectation $\hat{\mathbb{E}}$, so $U_{\tau}^MX_{\epsilon}$ and $X_{\epsilon}$ are identically distributed. Thus $U_{\tau}^MX_{\epsilon}$ is a copy of $X_{\epsilon}$.  Note that $|X_{\epsilon}|^2, |U_{\tau}^MX_{\epsilon}|^2, |U_{\tau}^MX_{\epsilon}-X_{\epsilon}|^2\in Lip(\Omega)$, By Proposition \ref{prop:Gregular} and Lemma \ref{lemma of inequality of the variant}, we know that 
		$$\hat{\mathbb{E}}[|X_{\epsilon}-\hat{\mathbb{E}}[X_{\epsilon}]|^2]\leq \hat{\mathbb{E}}[|X_{\epsilon}-U_{\tau}^MX_{\epsilon}|^2].$$
		Since $X=U_{\tau}X$, $\hat{\mathbb{E}}$-q.s., then $X=U_{\tau}^MX$, $\hat{\mathbb{E}}$-q.s.,  we have 
		\begin{equation}
			\begin{split}
				\hat{\mathbb{E}}[|X_{\epsilon}-U_{\tau}^MX_{\epsilon}|^2]&=\hat{\mathbb{E}}[|X_{\epsilon}-X+U_{\tau}^MX-U_{\tau}^MX_{\epsilon}|^2]\\
				&\leq 2(\hat{\mathbb{E}}[|X_{\epsilon}-X|^2]+\hat{\mathbb{E}}[U_{\tau}^M|X-X_{\epsilon}|^2])\\
				&\leq 4\hat{\mathbb{E}}[|X-X_{\epsilon}|^2]\\
				&< 4\epsilon.
			\end{split}
		\end{equation}
		Then 
		\begin{equation}
			\begin{split}
				\hat{\mathbb{E}}[|X-\hat{\mathbb{E}}[X]|^2]&=\hat{\mathbb{E}}[|X-X_{\epsilon}+X_{\epsilon}-\hat{\mathbb{E}}[X_{\epsilon}]+\hat{\mathbb{E}}[X_{\epsilon}]-\hat{\mathbb{E}}[X]|^2]\\
				&\leq  3(\hat{\mathbb{E}}[|X-X_{\epsilon}|^2]+\hat{\mathbb{E}}[|X_{\epsilon}-\hat{\mathbb{E}}[X_{\epsilon}]|^2]+|\hat{\mathbb{E}}[X_{\epsilon}]-\hat{\mathbb{E}}[X]|^2)\\
				&\leq  3(\epsilon+4\epsilon+\epsilon)\\
				&= 18\epsilon,
			\end{split}
		\end{equation}
		where 
		$$|\hat{\mathbb{E}}[X_{\epsilon}]-\hat{\mathbb{E}}[X]|^2\leq (\hat{\mathbb{E}}[|X_{\epsilon}-X|])^2\leq \hat{\mathbb{E}}[|X-X_{\epsilon}|^2] \leq \epsilon.$$
		Since $\epsilon>0$ is arbitrary, then $\hat{\mathbb{E}}[|X-\hat{\mathbb{E}}[X]|^2]=0$, which means $X=\hat{\mathbb{E}}[X]$, $\hat{\mathbb{E}}$-q.s.
	\end{proof}
	
	\section{Mixing}\label{sec:mix}
In this section, we first introduce the definition of mixing for sublinear expectation systems through  ``asymptotic independence". Similar to classical ergodic theory, we will demonstrate that mixing in the sublinear setting also implies ergodicity. Finally, we provide two concrete examples that exhibit mixing behavior.
	\subsection{Mixing for sublinear expectation systems}
	In the classical ergodic theory, mixing property plays an important role in dynamical systems. Roughly speaking, mixing means a kind of asymptotic independence, i.e. for a given dynamical system $(\Omega, \mathcal{F}, (\theta_t)_{t\in \mathbb T}, P)$, $P(A\cap\theta_t^{-1}B)\xrightarrow{t\rightarrow \infty} P(A)P(B)$  for any $A,B\in \mathcal{F}$. Now we give the notion of \textit{mixing} for sublinear expectation systems.
	
	\begin{definition}
		\label{Definition of mixing}
		A sublinear expectation system $(\Omega, \mathcal{H}, (U_t)_{t\in \mathbb T}, \hat{\mathbb{E}})$ ($\mathbb T=\mathbb N, \mathbb Z, \mathbb R_+, \mathbb R$) is called mixing if for any $X\in \mathcal{H}^{d_1},Y \in \mathcal{H}^{d_2}$ and $\phi \in C_{lip}(\mathbb{R}^{d_1+d_2})$, we have
		\begin{equation}
			\label{Equation in the definition of mixing}
			\lim_{t\rightarrow \infty}\hat{\mathbb{E}}[\phi(X,U_tY)]=\hat{\mathbb{E}}[\hat{\mathbb{E}}[\phi(x,Y)]|_{x=X}].
		\end{equation} 
	\end{definition}
	\begin{remark}
		\label{Mixing of U_t implies mixing of U}
		By  Definition \ref{Definition of mixing},  we can easily check that the system $(\Omega, \mathcal{H}, (U_t)_{t\in \mathbb R}, \hat{\mathbb{E}})$ being mixing will imply the systems $(\Omega, \mathcal{H}, (U_t)_{t\geq 0}, \hat{\mathbb{E}})$ and $(\Omega, \mathcal{H}, U_{\tau}, \hat{\mathbb{E}})$ being mixing for any $\tau>0$.
	\end{remark}

 \begin{remark}
Recall that in classical ergodic theory, if $(\theta_t)_{t\in \mathbb T}$, where $\mathbb{T}=\mathbb{R}$ or $\mathbb{Z}$, then $(\Omega, \mathcal{F}, (\theta_t)_{t\in \mathbb T}, P)$ is mixing if and only if $(\Omega, \mathcal{F}, (\theta_t^{-1})_{t\in \mathbb T}, P)$ is mixing. 
However, in the sublinear expectation framework, it follows from Remark \ref{rem:diff for inde} that in general, this is not true.
 \end{remark}
 \begin{remark}
    For capacity cases, Feng, Huang, Liu and Zhao \cite{FHLZ2023} introduced the definition of weak mixing, which also displays some type of asymptotic independence. However, our definition is from Peng’s  sublinear expectation framework.
 \end{remark}
	In the classical ergodic theory, a mixing  system must be   ergodic. Now we prove this is also true in  sublinear expectation framework.
	Let us begin with a lemma.
 	\begin{lemma}\label{lem:11.57}
		Let $(\Omega, \mathcal{H}, \hat{\mathbb{E}})$ be a sublinear expectation space. If $X\in \mathcal{H}$ is independent from itself, then $X$ is  constant, $\hat{\mathbb E}$-q.s.
	\end{lemma}
 \begin{proof}
     For any $c\in\mathbb{R}$, let $\phi_n(x,y)=((x-c)^+\wedge n)\cdot((y-c)^-\wedge n)$ for all $n\in \mathbb{N}$. Then $\phi_n \in C_{b,lip}(\mathbb{R}^2)$ for any $n\in \mathbb{N}$. By the independence, we have
		$$0=\hat{\mathbb{E}}[((X-c)^+\wedge n)\cdot((X-c)^-\wedge n)]=\hat{\mathbb{E}}[(X-c)^+\wedge n] \cdot \hat{\mathbb{E}}[(X-c)^-\wedge n].$$
		Then $\hat{\mathbb{E}}[(X-c)^+\wedge n]=0$ or $\hat{\mathbb{E}}[(X-c)^-\wedge n]=0$. Without loss of generality, we assume there exists a sequence $\{n_k\}_{k=1}^{\infty}$ with $n_k<n_{k+1}$ such that $\hat{\mathbb{E}}[(X-c)^-\wedge n_k]=0$. Since $X\in \mathcal{H}$ and $(X-c)^-\wedge n_k \uparrow (X-c)^-$, then we have 
		$$\hat{\mathbb{E}}[(X-c)^-]=\lim_{k\rightarrow \infty}\hat{\mathbb{E}}[(X-c)^-\wedge n_k]=0.$$
  Thus, 
  \begin{equation}\label{eq:1115-1}
      \text{$\hat{\mathbb{E}}[(X-c)^+]=0$ or $\hat{\mathbb{E}}[(X-c)^-]=0$ for any $c\in\mathbb R$}
  \end{equation}
		Let $c_0:=\sup\{c\in\mathbb R: \hat{\mathbb{E}}[(X-c)^-]=0\}$ with $\sup\emptyset:=-\infty$. It follows from \eqref{eq:1115-1} that $c_0=\inf\{c\in\mathbb R: \hat{\mathbb{E}}[(X-c)^+]=0\}$ with $\inf\emptyset:=\infty$. We first show that $c_0\neq \pm \infty$. Notice that $\hat{\mathbb{E}}[|X|]<\infty$. If $c_0=\infty$, there exist $c_n\uparrow \infty$ such that $\hat{\mathbb{E}}[(X-c_n)^-]=0$, i.e., $X\geq c_n$, $\hat{\mathbb{E}}$-q.s. for all $n\geq 1$, which contradicts $\hat{\mathbb{E}}[|X|]<\infty$. If $c_0=-\infty$, we know that $\hat{\mathbb{E}}[(-X+c)^-]=\hat{\mathbb{E}}[(X-c)^+]=0$ for all $c\in \mathbb R$, i.e., $X\leq c$, $\hat{\mathbb{E}}$-q.s. for all $c\in \mathbb R$, which also gives a contradiction. 
  
    Hence,
		$$\hat{\mathbb{E}}[(X-c_0)^-]=\lim_{c\uparrow c_0}\hat{\mathbb{E}}[(X-c)^-]=0, \ \text{ and } \ \hat{\mathbb{E}}[(X-c_0)^+]=\lim_{c\downarrow c_0}\hat{\mathbb{E}}[(X-c)^+]=0,$$
		which means $X=c_0,$ $\hat{\mathbb E}$-q.s.
 \end{proof}

 \begin{remark}\label{rem:X is constant}
     According to the proof of Lemma \ref{lem:11.57}, we know that if $X\in \mathcal{H}$ satisfies
     \[
     \hat{\mathbb{E}}[\phi_1(X)\phi_2(X)]=\hat{\mathbb{E}}[\phi_1(X)]\hat{\mathbb{E}}[\phi_2(X)], \ \text{ for all non-negative } \ \phi_1,\phi_2\in C_{b,lip}(\mathbb R),
     \]
     then $X=\hat{\mathbb{E}}[X]$, $\hat{\mathbb{E}}$-q.s.
 \end{remark}
	\begin{theorem}
		\label{mixing implies ergodicity}
		If the sublinear expectation system $(\Omega, \mathcal{H}, (U_t)_{t\in \mathbb T}, \hat{\mathbb{E}})$ is mixing, then it is ergodic.
	\end{theorem}
	
	\begin{proof}
		We only need to prove that if $X\in \mathcal{H}$ such that $U_tX=X$, $\hat{\mathbb E}$-q.s. for all $t\in \mathbb T$, then $X$ is a constant, $\hat{\mathbb E}$-q.s.  Note that  for any $\phi\in C_{lip}(\mathbb{R}^2)$ and any $t>0$, 
  \[|\hat{\mathbb{E}}[\phi(X,X)]-\hat{\mathbb{E}}[\phi(X,U_tX)]|\le l_\phi\hat{\mathbb{E}}[|X-U_tX|]=0,\]
 which together with the assumption that $(\Omega, \mathcal{H}, (U_t)_{t\in \mathbb T}, \hat{\mathbb{E}})$ is mixing, implies that
		\begin{equation*}
			\hat{\mathbb{E}}[\phi(X,X)]=\lim_{t\rightarrow \infty}\hat{\mathbb{E}}[\phi(X,U_tX)]= \hat{\mathbb{E}}[\hat{\mathbb{E}}[\phi(x, X)]|_{x=X}].
		\end{equation*}
		So $X$ is independent from $X$. Then Lemma \ref{lem:11.57} gives $X$ is constant, $\hat{\mathbb E}$-q.s.
	\end{proof}
   Let $L_{\mathcal{H}}^1$ be the completion of $\mathcal{H}$ under $\hat{\mathbb E}[|\cdot|]$. Then we have the following result.
   \begin{theorem}
       If the sublinear expectation system $(\Omega, \mathcal{H}, (U_t)_{t\in \mathbb T}, \hat{\mathbb{E}})$ is mixing, then the sublinear expectation system $(\Omega, L_{\mathcal{H}}^1, (U_t)_{t\in \mathbb T}, \hat{\mathbb{E}})$ is mixing.
   \end{theorem}
   \begin{proof}
       Fix $X\in (L_{\mathcal{H}}^1)^{d_1}, Y\in (L_{\mathcal{H}}^1)^{d_2}$ and $\phi\in C_{lip}(\mathbb R^{d_1+d_2})$. For any $\epsilon>0$, there exist $X_{\epsilon}\in \mathcal{H}^{d_1}$ and $ Y_{\epsilon}\in \mathcal{H}^{d_2}$ such that
       \[
       \hat{\mathbb E}[|X-X_\epsilon|]+\hat{\mathbb E}[|Y-Y_\epsilon|]\leq \epsilon.
       \]
       Note that
       \begin{equation}\label{eq:mixing-tw0-part-1}
       |\hat{\mathbb E}[\phi(X,U_tY)]-\hat{\mathbb E}[\phi(X_{\epsilon},U_tY_{\epsilon})]|\leq l_{\phi}\big(\hat{\mathbb E}[|X-X_\epsilon|]+\hat{\mathbb E}[|Y-Y_\epsilon|]\big)\leq l_{\phi}\epsilon,
       \end{equation}
       and
       \begin{equation}\label{eq:mixing-tw0-part-2}
            \begin{split}
                |\hat{\mathbb{E}}[\hat{\mathbb{E}}[\phi(x,Y_{\epsilon})]|_{x=X_{\epsilon}}]-\hat{\mathbb{E}}[\hat{\mathbb{E}}[\phi(x,Y)]|_{x=X}]|
				&\leq |\hat{\mathbb{E}}[\hat{\mathbb{E}}[\phi(x,Y_{\epsilon})]|_{x=X_{\epsilon}}]-\hat{\mathbb{E}}[\hat{\mathbb{E}}[\phi(x, Y_{\epsilon})]|_{x=X}]|\\
				&\ \ \ \ +|\hat{\mathbb{E}}[\hat{\mathbb{E}}[\phi(x, Y_{\epsilon})]|_{x=X}]-\hat{\mathbb{E}}[\hat{\mathbb{E}}[\phi(x,Y)]|_{x=X}]|\\
                &\leq \hat{\mathbb{E}}[|\varphi_{\epsilon}(X_\epsilon)-\varphi_{\epsilon}(X)|]\\
                &\ \ \ \ +\hat{\mathbb{E}}\big[\hat{\mathbb{E}}[|\phi(x, Y_{\epsilon})-\phi(x, Y)|]\big|_{x=X}\big]\\
				&\leq l_{\phi}\hat{\mathbb{E}}[|X_{\epsilon}-X|]+l_{\phi}\hat{\mathbb{E}}[|Y_{\epsilon}-Y|]\\
				&\leq l_{\phi}\epsilon, 
            \end{split}
        \end{equation}
        where $l_{\phi}>0$ is the Lipschitz constant of $\phi$ and  $\varphi_{\epsilon}(x):=\hat{\mathbb{E}}[\phi(x,Y_{\epsilon})]$ is a Lipschitz function with the same Lipschitz constant $l_{\phi}$. Hence
        \[
        \limsup_{t\to\infty}|\hat{\mathbb E}[\phi(X,U_tY)]-\hat{\mathbb{E}}[\hat{\mathbb{E}}[\phi(x,Y)]|_{x=X}]|\leq 2l_{\phi}\epsilon.
        \]
        We obtain the desired result since $\epsilon>0$ is arbitrary.
   \end{proof}
	
	\subsection{Examples of mixing}\subsubsection{Mixing of i.i.d. sequence}
	We say $\{\xi_n\}_{n\geq 1}$ is an i.i.d. sequence in a sublinear expectation space $(\Omega, \mathcal{H}, \hat{\mathbb{E}})$ if $\xi_i$ and $\xi_j$ are identically distributed for all $i \neq j$ and $\xi_n$ is independent from $(\xi_1,\xi_2,\cdots,\xi_{n-1})$ for all $n\geq 2$. For any given i.i.d. sequence $\{\xi_n\}_{n\geq 1}$, set 
	$$Lip(\xi):=\{\phi(\xi_1, \xi_2, \cdots, \xi_n): \text{for all } n\geq 1 \text{ and } \phi\in C_{lip}(\mathbb{R}^n)\}.$$
	Then $Lip(\xi)\subset \mathcal{H}$. We denote by $L^1(\xi)$ the completion of $Lip(\xi)$ under the norm $\|\cdot\|_1=\hat{\mathbb{E}}[|\cdot|]$. Define $U: Lip(\xi)\rightarrow Lip(\xi)$  by 
$$U\phi(\xi_1,\xi_2,\cdots,\xi_n)=\phi(\xi_2,\xi_3,\cdots,\xi_{n+1}).$$
	Then $U$ is a linear transformation and preserves $\hat{\mathbb{E}}$, which can be continuously extended to $L^1(\xi)$. Now we consider the sublinear expectation system $(\Omega, L^1(\xi), U, \hat{\mathbb{E}})$ and we have the following theorem.
	
	\begin{theorem}
		The sublinear expectation system $(\Omega, L^1(\xi), U, \hat{\mathbb{E}})$ is mixing.
	\end{theorem}
	\begin{proof}
		 Fix $X\in (L^1(\xi))^{d_1}$, $Y\in (L^1(\xi))^{d_2}$. For any $\epsilon>0$, there exist 
  \[
  X_{\epsilon}=\phi_1(\xi_1,\xi_2,\cdots,\xi_{n}), \ \ Y_{\epsilon}=\phi_2(\xi_1,\xi_2,\cdots,\xi_{n})
  \]
   for some $n\in \mathbb{N}$ and $\phi_1 \in (C_{lip}(\mathbb{R}^{n}))^{d_1}, \phi_2\in (C_{lip}(\mathbb{R}^{n}))^{d_2}$, such that
		$\hat{\mathbb{E}}[|X_{\epsilon}-X|]+\hat{\mathbb{E}}[|Y_{\epsilon}-Y|]<\epsilon$.
Note that $U^mY_{\epsilon}=\phi_2(\xi_{m+1},\xi_{m+2},\cdots,\xi_{m+n})$ for any $m\in \mathbb{N}$. Thus, $U^mY_{\epsilon}$ is independent from $X_{\epsilon}$ for all $m\geq n$. Since $U$ preserves $\hat{\mathbb{E}}$, then for any $\phi\in C_{lip}(\mathbb{R}^{d_1+d_2})$ and $m\geq n$, we have
		\begin{equation}\label{Mixing inequation seperated into two parts *}
			\begin{split}
				|\hat{\mathbb{E}}[\phi(X,U^mY)]-\hat{\mathbb{E}}[\hat{\mathbb{E}}[\phi(x,Y)]|_{x=X}]|
				&\leq |\hat{\mathbb{E}}[\phi(X,U^mY)]-\hat{\mathbb{E}}[\phi(X_{\epsilon},U^mY_{\epsilon})]|\\
				&\ \ \ \ +|\hat{\mathbb{E}}[\hat{\mathbb{E}}[\phi(x,Y_{\epsilon})]|_{x=X_{\epsilon}}]-\hat{\mathbb{E}}[\hat{\mathbb{E}}[\phi(x,Y)]|_{x=X}]|.
			\end{split}
		\end{equation}
		 Arguing as in \eqref{eq:mixing-tw0-part-1}-\eqref{eq:mixing-tw0-part-2}, we conclude
		\begin{equation}
			\label{Mixing inequation seperated into two parts 1}
			\begin{split}
				|\hat{\mathbb{E}}[\phi(X,U^mY)]-\hat{\mathbb{E}}[\phi(X_{\epsilon},U^mY_{\epsilon})]|
				&\leq l_{\phi}\hat{\mathbb{E}}[(|X-X_{\epsilon}|+U^m|Y-Y_{\epsilon}|)]\\
				&\leq l_{\phi}(\hat{\mathbb{E}}[|X-X_{\epsilon}|]+\hat{\mathbb{E}}[|Y-Y_{\epsilon}|])\\
				&\leq l_{\phi}\epsilon,
			\end{split}
		\end{equation}
		and
        \begin{equation}\label{Mixing inequation seperated into two parts 2}
            \begin{split}
                |\hat{\mathbb{E}}[\hat{\mathbb{E}}[\phi(x,Y_{\epsilon})]|_{x=X_{\epsilon}}]-\hat{\mathbb{E}}[\hat{\mathbb{E}}[\phi(x,Y)]|_{x=X}]|
				&\leq |\hat{\mathbb{E}}[\hat{\mathbb{E}}[\phi(x,Y_{\epsilon})]|_{x=X_{\epsilon}}]-\hat{\mathbb{E}}[\hat{\mathbb{E}}[\phi(x, Y_{\epsilon})]|_{x=X}]|\\
				&\ \ \ \ +|\hat{\mathbb{E}}[\hat{\mathbb{E}}[\phi(x, Y_{\epsilon})]|_{x=X}]-\hat{\mathbb{E}}[\hat{\mathbb{E}}[\phi(x,Y)]|_{x=X}]|\\
				&\leq l_{\phi}\hat{\mathbb{E}}[|X_{\epsilon}-X|]+l_{\phi}\hat{\mathbb{E}}[|Y_{\epsilon}-Y|]\\
				&\leq l_{\phi}\epsilon, 
            \end{split}
        \end{equation}
		where $l_{\phi}>0$ is the Lipschitz constant of $\phi$.  Combining \eqref{Mixing inequation seperated into two parts 1} and \eqref{Mixing inequation seperated into two parts 2}, we know that for all $m\ge n$,
		$$|\hat{\mathbb{E}}[\phi(X,U^mY)]-\hat{\mathbb{E}}[\hat{\mathbb{E}}[\phi(x,Y)]|_{x=X}]| \leq 2l_{\phi}\epsilon,$$
		which means
		$$\lim_{n\rightarrow \infty}\hat{\mathbb{E}}[\phi(X,U^nY)]=\hat{\mathbb{E}}[\hat{\mathbb{E}}[\phi(x,Y)]|_{x=X}].$$
  The proof is completed.
	\end{proof}
\subsubsection{Mixing of $G$-Brownian motion}
By checking definition, we have proved that the $G$-expectation systems $(\Omega, L_G^2(\Omega), U_{\tau}, \hat{\mathbb{E}})$, for any $\tau>0$, and $(\Omega, L_G^2(\Omega), (U_t)_{t\geq 0}, \hat{\mathbb{E}})$ are ergodic. Now we prove a stronger property that they are also mixing. In particular, according to the fact mixing implies ergodicity, this provides a new method to prove $G$-expectation systems are ergodic. Furthermore, we study the structure of $\mathcal{P}$ induced by $G$-expectation systems.
	
	\begin{theorem}
		\label{mixing of G-Brownian motion}
		Both the discrete time $G$-expectation system $(\Omega, L_G^p(\Omega), U_{\tau}, \hat{\mathbb{E}})$, for any $\tau>0$, and the continuous time  $G$-expectation system $(\Omega, L_G^p(\Omega), (U_t)_{t\geq 0}, \hat{\mathbb{E}})$ are mixing.
	\end{theorem}
	\begin{proof}
		By Remark \ref{Mixing of U_t implies mixing of U}, we just need to prove that $(\Omega, L_G^p(\Omega), (U_t)_{t\geq 0}, \hat{\mathbb{E}})$ is mixing. Note that for any $p\ge 1$, $L_G^p(\Omega)\subset L_G^1(\Omega)$, and then we only need to prove the case for $p=1$.

  Now fix $X\in (L_G^1(\Omega))^{d_1}, Y\in (L_G^1(\Omega))^{d_2}$. Since $L_G^1(\Omega)$ is the completion of $Lip(\Omega)$ under norm $\|\cdot\|_1$, then for any $\epsilon>0$, there exist
  \[
  X_{\epsilon}=\phi_1(B_{t_1}-B_{t_0},B_{t_2}-B_{t_1},\cdots,B_{t_n}-B_{t_{n-1}}), \ \ Y_{\epsilon}=\phi_2(B_{t_1}-B_{t_0},B_{t_2}-B_{t_1},\cdots,B_{t_n}-B_{t_{n-1}})
  \]
  for some $n\in \mathbb{N}$, $\phi_1\in (C_{l,lip}(\mathbb{R}^{n\times d}))^{d_1}, \phi_2 \in (C_{l,lip}(\mathbb{R}^{n\times d}))^{d_2}$ and $t_0<t_1<\cdots<t_n$, such that
		$\hat{\mathbb{E}}[|X_{\epsilon}-X|]+\hat{\mathbb{E}}[|Y_{\epsilon}-Y|]<\epsilon$.
  
Note that $U_tY_{\epsilon}$ is independent from $X_{\epsilon}$ for all $t\geq t_n-t_0$, similar to \eqref{Mixing inequation seperated into two parts *}-\eqref{Mixing inequation seperated into two parts 2}, for any $\phi\in C_{lip}(\mathbb{R}^{d_1+d_2})$, we know that for all $t\geq t_n-t_0$,
		$$|\hat{\mathbb{E}}[\phi(X,U_tY)]-\hat{\mathbb{E}}[\hat{\mathbb{E}}[\phi(x,Y)]|_{x=X}]| \leq 2l_{\phi}\epsilon,$$
		where $l_{\phi}>0$ is the Lipschitz constant of $\phi$. Thus 
		$$\lim_{t\rightarrow \infty}\hat{\mathbb{E}}[\phi(X,U_tY)] =\hat{\mathbb{E}}[\hat{\mathbb{E}}[\phi(x,Y)]|_{x=X}],$$
  which finishes the proof.
	\end{proof}

Now we consider the structure of $\mathcal{P}$ generated by $G$-Brownian motion.
 For a given $d$-dimensional Brownian motion $(B_t)_{t\in \mathbb{T}}$ ($\mathbb{T}=\mathbb{R}$ or $\mathbb{R}_+$) on a probability space $(\Omega,\mathcal{F},P)$ with covariance matrix $\Sigma$ ($\Sigma$ can be degenerate or even 0), where
\[
\Omega=C_0(\mathbb{T};\mathbb{R}^d), \ B_t(\omega)=\omega_t, \text{ for all } t\in \mathbb{T}, \omega\in \Omega, \ \text{ and } \ \mathcal{F}=\sigma(B_t: t\in \mathbb{T}).
\]
Let $\theta_{t}: \Omega\to \Omega$ be the Brownian shift, i.e., $(\theta_t\omega)_s=\omega_{t+s}-\omega_t$ for all $t\geq 0, s\in \mathbb{T}$ and $\omega\in \Omega$.
It is well known that the continuous time dynamical system $(\Omega,\mathcal{F}, (\theta_t)_{t\geq 0}, P)$ and the discrete  time dynamical system $(\Omega,\mathcal{F}, \theta_{\tau}, P)$ for any $\tau>0$ are ergodic (see e.g. \cite{Arnold1998} and \cite[Theorem 2.8]{Feng-Qu-Zhao2020}). We obtain a stronger result--the continuous time dynamical system $(\Omega,\mathcal{F}, (\theta_t)_{t\geq 0}, P)$ is strongly mixing (see Appendix for a detail proof).
\begin{theorem}\label{Thm: strong-mixing-of-BM}
    The measure-preserving dynamical system $(\Omega,\mathcal{F}, (\theta_t)_{t\geq 0}, P)$ generated by a Brownian motion is strongly mixing.
\end{theorem}
 
	For a given $d$-dimensional $G$-Brownian motion $(B_t)_{t\geq 0}$ on a $G$-expectation space $(\Omega, L_G^1(\Omega), \hat{\mathbb{E}})$, where $\Omega=C_0(\mathbb{R}_+;\mathbb{R}^d)$, $G: \mathbb{S}_d\to \mathbb{R}$ is defined by
	\begin{equation}
		G(A):=\frac{1}{2}\sup_{q\in Q}Tr[Aq],
	\end{equation}
	for some bounded, convex and closed subset $Q\subset \mathbb{S}_d$. Let $(W_t)_{t\geq 0}$ be a standard $d$-dimensional Brownian motion on a probability space $(\tilde{\Omega},\tilde{\mathcal{F}},\tilde{P})$ and $\tilde{\mathbb{F}}:=(\tilde{\mathcal{F}}_t)_{t\geq 0}$ be the filtration generated by the Brownian motion $(W_t)_{t\geq 0}$. Denote by $\mathcal{A}_{0,\infty}^{Q}$ the collection of all $Q$-valued $\tilde{\mathbb{F}}$-adapted processes on $[0,\infty)$. For each $\eta\in \mathcal{A}_{0,\infty}^{Q}$, define
	\begin{equation}\label{eq: represent G_BM by BM}
		B^{\eta}_t=\int_0^t\eta_sdW_s, \ \text{ and } \ P_{\eta}:=\mathcal{L}(B^{\eta}_{\cdot}) \text{ the law of the process } \ (B^{\eta}_t)_{t\geq 0}.
	\end{equation}
	Then $P_{\eta}$ is a probability on $(\Omega,\mathcal{F})$ with $\mathcal{F}=\sigma\{B_t:t\geq 0\}$. In particular, we denote by $B^{q}$ and $P_{q}$ if 
	$\eta\equiv q$ in \eqref{eq: represent G_BM by BM} for any $q\in Q$. Set
	\begin{equation*}
		\mathcal{P}:=\overline{\{P_{\eta}: \eta\in \mathcal{A}_{0,\infty}^{Q}\}},
	\end{equation*}
	where $\overline{A}$ for a set of probabilities $A$ represents the closure of $A$ under the weak convergence topology.
	
	According to Proposition 49 and Proposition 50 in \cite{DenisPeng2011} (see also \cite[Proposition 6.2.14 and Proposition 6.2.15]{Pengbook}), we know that
	\begin{equation*}
		\hat{\mathbb{E}}[X]=\sup_{P\in \mathcal{P}}\mathbb{E}_P[X], \ \text{ for all } \ X\in L_G^1(\Omega).
	\end{equation*}
	Let $\mathcal{F}_t:=\sigma\{B_s:0\leq s\leq t\}$ be the natural filtration generated by the $G$-Brownian motion $\{B_{t}\}_{t\geq 0}$. Then we have the following theorem.
	\begin{theorem}\label{thm:uncountable many mixing measure}
		The $G$-Brownian motion $B_t(\omega)=\omega_t$ is a martingale on $(\Omega,\mathcal{F},(\mathcal{F}_t)_{t\geq 0}, P_{\eta})$ for each $\eta\in \mathcal{A}_{0,\infty}^{Q}$ and
		\begin{equation}\label{eq: quadratic variation}
			\mathbb{E}_{P_{\eta}}\Big[\phi\big(\langle B,B\rangle_t^{P_{\eta}}\big)\Big]={\mathbb{E}_{\tilde P}}\bigg[\phi\bigg(\int_0^t\eta_sds\bigg)\bigg], \ \text{ for all bounded measurable } \ \phi:\mathbb R^d\to \mathbb R,
		\end{equation}
        where $\langle B,B\rangle^{P_{\eta}}$ is the quadratic variance process of $B$ under $P_{\eta}$.
		Moreover, for any $q\in Q$, $\{B_t\}_{t\geq 0}$ is a Brownian motion with covariance $(qt)_{t\geq 0}$ under $P_{q}$ and $(\Omega,\mathcal{F}, (\theta_t)_{t\geq 0}, P_{q})$ is strongly mixing.
	\end{theorem}

	\begin{proof}
		Note that $P_{\eta}:=\mathcal{L}(B^{\eta}_{\cdot})$ for any $\eta\in \mathcal{A}_{0,\infty}^{Q}$. Then the canonical process $(B_t)_{t\geq 0}$ with $B_t(\omega)=\omega_t$ on $(\Omega,\mathcal{F}, P_{\eta})$ and $(B^{\eta}_t)_{t\geq 0}$ on $(\tilde{\Omega},\tilde{\mathcal{F}},\tilde{P})$ have the same distribution. Since $(B^{\eta}_t)_{t\geq 0}$ is a martingale on $(\tilde{\Omega},\tilde{\mathcal{F}},(\tilde{\mathcal{F}}_t)_{t\geq 0},\tilde{P})$ with
		\begin{equation*}
			\langle B^{\eta},B^{\eta}\rangle_t=\int_0^t\eta_sds, \ \ \tilde{P}\text{-a.s.}
		\end{equation*}
		Hence, the $G$-Brownian motion $(B_t)_{t\geq 0}$ is a martingale on $(\Omega,\mathcal{F},(\mathcal{F}_t)_{t\geq 0}, P_{\eta})$ such that \eqref{eq: quadratic variation} holds.

		Now, we already know that for any $q\in Q$, $(B_t)_{t\geq 0}$ is a Brownian motion with covariance $(q t)_{t\geq 0}$ under $P_{q}$. By Theorem \ref{Thm: strong-mixing-of-BM}, we know that $(\Omega,\mathcal{F}, (\theta_t)_{t\geq 0}, P_{q})$ is strongly mixing.
	\end{proof}
\begin{remark}
   As a corollary of Corollary \ref{cor:P unique} and the above theorem, we immediately obtain that the Birkhoff's ergodic theorem does not hold for all elements in $G$-expectation systems generated by $G$-Brownian motion. 
\end{remark}

	\section{Law of large numbers under $\alpha$-mixing condition}\label{Sec: LLN}
 In this section, we first introduce the definition of $\a$-mixing for sublinear expectation systems, which is a strictly weaker condition than independence. After that, we establish the law of large numbers and the strong law of large numbers under the $\a$-mixing condition.
 \subsection{The $\a$-mixing sequence and admissible subsequence}
	Suppose that $(\Omega, \mathcal{H}, U, \hat{\mathbb{E}})$ is a given sublinear expectation system. 
	For a fixed random vector $X=(X_1,\cdots,X_d)\in \mathcal{H}^d$, set
	\begin{equation}\label{eq:S_n-0827}
		S_n:=\frac{1}{n}\sum_{i=1}^nU^iX.
	\end{equation}
	If $U^nX$ is independent from $(U^{n-1}X,\cdots,UX,X)$ for all $n\geq 1$, Peng \cite{Pengbook} gave the following law of large numbers by the viscosity solution of $G$-equation (nonlinear PDE):  for the fixed $X=(X_1,\cdots,X_d)\in \mathcal{H}^d$ in \eqref{eq:S_n-0827} and for all $\phi\in C_{lip}(\mathbb{R}^d)$
	\begin{equation*}
		\lim_{n\to \infty}\hat{\mathbb{E}}[\phi(S_n)]=\max_{x\in \Gamma} \phi(x),
	\end{equation*}
    where 
    \begin{equation}\label{Gamma defined in Peng}
        \begin{split}
            \Gamma\subset \mathbb R^d &\text{ is the unique bounded, closed and convex subset satisfying }\\
            & \qquad \qquad \max_{x\in \Gamma}\langle p,x\rangle= \hat{\mathbb E}[\langle p,X\rangle], \ \text{ for all } \  p\in \mathbb R^d.
        \end{split}
    \end{equation}  
Here $\langle p,X\rangle=\sum_{i=1}^dp_iX_i$ for any $p=(p_1,\cdots,p_d)\in \mathbb R^d$.

	By Stein's method under sublinear expectations, Song \cite{Song2021} obtained the convergence rate $n^{-\frac{1}{2}}$ for some $\delta\in (0,1]$ with $\hat{\mathbb{E}}[|X|^{2}]<\infty$ in $1$-dimensional case. In this section, we will give the law of large numbers under $\alpha$-mixing condition.
	
	We first state some notions. For a non-empty subset $\Lambda$ of $\mathbb{R}_+$ with finite elements, let $|\Lambda|$ be the number of elements in $\Lambda$ and $\Lambda_{max}$ (resp. $\Lambda_{min}$) be the maximum (resp. minimum) element in $\Lambda$. For any two finite subset $\Lambda^1,\Lambda^2\subset \mathbb{R}_+$, we say $\Lambda^1\leq \Lambda^2$ if $\Lambda^1_{max}\leq \Lambda^2_{min}$.  Now we give the following mixing condition.
	
	\begin{definition}\label{def: alpha-mixing}
 A process $\{X_t\}_{t\geq 0}$ of $d$-dimensional random vectors on a sublinear expectation space $(\Omega, \mathcal{H}, \hat{\mathbb{E}})$ is said to be $\alpha$-mixing if there exist $c_X>0,\alpha>0$ such that for any non-empty finite subset $\Lambda^1,\Lambda^2\subset \mathbb{R}_+$ with $\Lambda^1\leq \Lambda^2$ and $\phi\in C_{lip}(\mathbb{R}^{2d})$,
		\begin{equation}\label{equ: alpha-mixing}
			\begin{split}
				|\hat{\mathbb{E}}[\phi(\bar{X}_{\Lambda^1}, \bar{X}_{\Lambda^2})] - \hat{\mathbb{E}}[\hat{\mathbb{E}}[\phi(x, \bar{X}_{\Lambda^2})]|_{x=\bar{X}_{\Lambda^1}}]| \leq c_X l_{\phi} e^{-\alpha |\Lambda^2_{min}-\Lambda^1_{max}|},
			\end{split}
		\end{equation}
		where $\bar{X}_{\Lambda^i}=\frac{1}{|\Lambda^i|}\sum_{s\in \Lambda^i}X_s$, $i=1,2$, $l_{\phi}$ is the Lipschitz constant of $\phi$.
  
   Similarly, a sequence $\{X_k\}_{k\geq 0}$ of $d$-dimensional random vectors on a sublinear expectation space $(\Omega, \mathcal{H}, \hat{\mathbb{E}})$ is called $\alpha$-mixing if \eqref{equ: alpha-mixing} holds for non-empty finite subsets $\Lambda^1,\Lambda^2\subset \mathbb{N}$ with $\Lambda^1\leq \Lambda^2$.
	\end{definition}

	\begin{definition}
		A sublinear expectation system $(\Omega,\mathcal{H},(U_t)_{t\in \mathbb T},\hat{\mathbb{E}})$ ($\mathbb T=\mathbb N, \mathbb R_+$) is called $\alpha$-mixing if for any $X\in \mathcal{H}$, the process (sequence) $\{X_t\}_{t\in \mathbb T}$ with $X_t=U_tX$ is $\alpha$-mixing.
	\end{definition}
	
	\begin{remark}
		If a continuous sublinear expectation system $(\Omega,\mathcal{H},(U_t)_{t\geq 0},\hat{\mathbb{E}})$ is $\alpha$-mixing, then for any $\tau>0$, the discrete sublinear expectation system $(\Omega,\mathcal{H},U_{\tau},\hat{\mathbb{E}})$ is $\alpha$-mixing.
	\end{remark}
	
	\begin{definition}\label{def: identical sequence}
		A process (sequence) $\{X_t\}_{t\in \mathbb{T}}$ ($\mathbb{T}\subset \mathbb{R}_+$) of d-dimensional random vectors on a sublinear expectation space $(\Omega, \mathcal{H}, \hat{\mathbb{E}})$ is called
        \begin{enumerate}[(i)]
            \item  an identical process (sequence) if $X_t$ and $X_s$ are identically distributed for all $t,s\in \mathbb{T}$;
            \item a stationary process (sequence) if for any $n\geq 1$, $t\in\mathbb T$ and $\{t_i\}_{i=1}^n\subset \mathbb{T}$ with $\{t+t_i\}_{i=1}^n\subset \mathbb{T}$
		\begin{equation*}
			(X_{t_1},\cdots,X_{t_n})\deq (X_{t+t_1},\cdots,X_{t+t_n}).
		\end{equation*}
        \end{enumerate}
	\end{definition}

	\begin{remark}
 \begin{enumerate}[(1)]
     \item If a sequence $\{X_t\}_{t\in \mathbb{N}}$ is i.i.d., then it is a stationary $\a$-mixing sequence;
     \item There exist a $\a$-mixing sublinear expectation system $(\Omega,\mathcal{H},(U_t)_{t\in \mathbb T},\hat{\mathbb{E}})$ satisfying: there exists $X\in \mathcal{H}$ such that $U_tX$ is not independent from $U_sX$ for any $t>s$ (see Theorem \ref{Thm:mixing of G-SDE} and Theorem \ref{Thm:not independent}).
 \end{enumerate}			
	\end{remark}

\begin{definition}\label{Def: admissible}
		A subsequence $\{t_k\}_{k\geq 1}$ of $\mathbb{R}_+$ is called $(\gamma,\delta)$-admissible for some $\delta\geq 0$ and $0\leq \gamma\leq 1$ if there exist $c_1,c_2>0$ and $N_0>0$ such that for any $n\geq N_0$, there exists $\Pi_n\subset \{1,\cdots,n\}$ such that
		\begin{equation}\label{eq:admissible}
			|\Pi_n|\leq c_1n^{\gamma}, \ \text{ and } \ \inf_{1\leq i\neq j\leq n, i,j\notin \Pi_n}|t_i-t_j|\geq c_2(\ln n)^{1+\delta}.
		\end{equation}
	\end{definition}

 \begin{lemma}
     Let $g: \mathbb N\to \mathbb R_+$ be a given function. If there exists $\delta>0$ such that
     \begin{equation}\label{Suf cond for admissible 1}
         \liminf_{k\to \infty}\frac{g(k+1)-g(k)}{(\ln k)^{1+\delta}}>0,
     \end{equation}
     then $\{g(k)\}_{k\geq 1}$ is $(\gamma,\delta)$-admissible for any $0<\gamma<1$. In particular, if there exists $\beta>0$ such that
     \begin{equation}\label{Suf cond for admissible 2}
         \liminf_{k\to \infty}\frac{g(k+1)-g(k)}{k^{\beta}}>0,
     \end{equation}
     then $\{g(k)\}_{k\geq 1}$ is $(\gamma,\delta)$-admissible for any $0<\gamma<1$ and $\delta>0$. 
 \end{lemma}

 \begin{proof}
     For any $0<\gamma<1$, according to \eqref{Suf cond for admissible 1}, there exist $c_0>0$ and $N_0>0$ such that for all $k\geq N_0^{\gamma}$,
     \begin{equation}\label{eq:remove-0827-2}
         g(k+1)-g(k)\geq c_0(\ln k)^{1+\delta}.
     \end{equation}
   Given $n\in\mathbb{N}$, let $\Pi_n=\{1,2,\cdots,n_{\gamma}\}$, where $n_{\gamma}$ is the largest integer less than or equal to $n^{\gamma}$. Then for any $n\geq N_0$, we have $|\Pi_n|=n_{\gamma}\leq n^{\gamma}$ and
     \[
     \inf_{1\leq i\neq j\leq n, i,j\notin \Pi_n}|g(i)-g(j)|\geq c_0 (\ln (n_\gamma+1))^{1+\delta}\geq c_0\gamma^{1+\delta}(\ln n)^{1+\delta}.
     \]
     Hence, $\{g(k)\}_{k\geq 1}$ is $(\gamma,\delta)$-admissible.

     In particular, as the condition \eqref{Suf cond for admissible 2} implies condition \eqref{Suf cond for admissible 1} for all $\delta>0$, the second statement holds. Thus, the proof is completed.
 \end{proof}

 \begin{remark}
     Let $P(x)=a_mx^m+\cdots+a_1x+a_0$ be a polynomial with $a_m>0$ for some $m\geq 2$. It is easy to check that $g: \mathbb N\to \mathbb R_+$ with $g(k)=\lfloor P(k)\rfloor\vee 1$ for all $k\in \mathbb N$ satisfies \eqref{Suf cond for admissible 2}. Hence, $\{g(k)\}_{k\geq 1}$ is $(\gamma,\delta)$-admissible for all  $0<\gamma<1$ and $\delta>0$.
 \end{remark}

 Now for a given sublinear expectation space $(\Omega, \mathcal{H}, \hat{\mathbb{E}})$, let $X=(X_1,\cdots,X_d)\in \mathcal{H}^d$ and $\Gamma$ be the unique bounded, closed and convex subset of $\mathbb R^d$ defined by \eqref{Gamma defined in Peng}, which is denoted by $X\sim \Gamma$. 

 Recall $\Theta$ be the family of all linear expectations dominated by $\hat{\mathbb{E}}$, i.e.,
 \begin{equation}\label{Def of Theta}
     \Theta=\{\mathbb E: \mathbb E \text{ is a linear expectation on } \mathcal{H} \text{ such that } \mathbb E[X]\leq \hat{\mathbb{E}}[X], \text{ for all } X\in \mathcal{H}\}.
 \end{equation}
 By Theorem  1.2.1 in \cite{Pengbook}, we know that
	\begin{equation}\label{0615-1}
		\hat{\mathbb{E}}[X]=\max_{\mathbb{E}\in\Theta}\mathbb{E}[X],\text{ for all }X\in\mathcal{H}.
	\end{equation}
 Moreover, for any $X\in \mathcal{H}$, there exists $\mathbb{E}_X\in \Theta$ such that $\hat{\mathbb{E}}[X]=\mathbb{E}_X[X]$.

 \begin{lemma}\label{lem:Charact of Gamma}
   For any $X=(X_1,\cdots,X_d)\in \mathcal{H}^d$  such that $X\sim \Gamma$, we have
     \begin{equation}\label{eq:Charac for Gamma}
         \begin{split}
             \Gamma&=\{x\in \mathbb R^d: \langle p,x\rangle\leq\hat{\mathbb E}[\langle p,X\rangle], \text{ for all } p\in \mathbb R^d\}\\
             &=\{\mathbb{E}[X]=(\mathbb{E}[X_1],\cdots, \mathbb{E}[X_d]): \mathbb{E}\in \Theta\}.
         \end{split}
     \end{equation}
 \end{lemma}
 \begin{proof}
     Given $X=(X_1,\cdots,X_d)\in \mathcal{H}^d$, let 
     \[
     A=\{x\in \mathbb R^d: \langle p,x\rangle\leq \hat{\mathbb E}[\langle p,X\rangle], \text{ for all }  p\in \mathbb R^d\}
     \]
     and 
     \[
     B=\{\mathbb{E}[X]=(\mathbb{E}[X_1],\cdots, \mathbb{E}[X_d]): \mathbb{E}\in \Theta\}.
     \]
     It is easy to check that $A$ is a bounded, closed and convex subset of $\mathbb{R}^d$. By the convexity of $\T$, $B$ is a convex subset of $\mathbb{R}^d$.  Note also that $|\mathbb{E}[X_i]|\leq \mathbb{E}[|X_i|]\leq \hat{\mathbb{E}}[|X_i|]\leq \hat{\mathbb{E}}[|X|]<\infty$ for all $1\leq i\leq d$ and $\mathbb{E}\in \Theta$, hence, $B$ is also bounded.
     
     We first show that $A=\Gamma$. By \eqref{Gamma defined in Peng}, it is obvious that $\Gamma\subset A$. Note that 
     \[
     \hat{\mathbb E}[\langle p,X\rangle]=\max_{x\in \Gamma}\langle p,x\rangle\leq \max_{x\in A}\langle p,x\rangle\leq \hat{\mathbb E}[\langle p,X\rangle], \ \text{ for any } \ p\in \mathbb R^d.
     \]
     Since $A$ is bounded, closed and convex, then  the uniqueness in \eqref{Gamma defined in Peng} gives $A=\Gamma$.

     Now we show that $B=\Gamma$. According to \eqref{0615-1}, we know that
     \[
     \max_{x\in B}\langle p,x\rangle=\max_{\mathbb{E}\in\Theta}\langle p,\mathbb{E}[X]\rangle=\max_{\mathbb{E}\in\Theta}\mathbb{E}[\langle p,X\rangle]=\hat{\mathbb E}[\langle p,X\rangle].
     \]
     Since $B$ is bounded and convex, by the uniqueness in \eqref{Gamma defined in Peng}, it suffices to prove that $B$ is closed. We only need to show that for any  $x_n=\mathbb{E}_n[X]\in B$  with $\mathbb{E}_n\in \Theta$ and $x\in \mathbb R^d$ such that $x_n\to x$ as $n\to \infty$, we have $x\in B$. Define $\tilde{\mathbb{E}}: \mathcal{H}\to \mathbb R$ by
     \begin{equation*}
         \tilde{\mathbb{E}}[Y]=\limsup_{n\to \infty}\mathbb{E}_n[Y], \ \text{ for any } \ Y\in \mathcal{H}.
     \end{equation*}
     It is easy to prove that $\tilde{\mathbb{E}}$ is a sublinear expectation on $\mathcal{H}$ which is dominated by $\hat{\mathbb E}$, i.e., $\tilde{\mathbb{E}}[Y]\leq \hat{\mathbb E}[Y]$ for all $Y\in \mathcal{H}$. Let $\tilde{\Theta}$ be the family of all linear expectations dominated by $\tilde{\mathbb{E}}$ defined as in \eqref{Def of Theta}. Since $\tilde{\mathbb{E}}$ is dominated by $\hat{\mathbb E}$, we have $\tilde{\Theta}\subset \Theta$. On the other hand, according to Theorem  1.2.1 in \cite{Pengbook}, there exists $\tilde{\mathbb{E}}_X\in \tilde{\Theta}\subset \Theta$ such that 
     \begin{equation*}
         \tilde{\mathbb{E}}_X[X]=\tilde{\mathbb{E}}[X]=\limsup_{n\to \infty}\mathbb{E}_n[X]=\limsup_{n\to \infty}x_n=x.
     \end{equation*}
     Hence, $x\in B$.
 \end{proof}
\begin{remark}\label{rem:d=>Gamma}
  Given $X_1,X_2\in\mathcal{H}^d$, if $X_1\overset{d}{=}X_2$ and $X_i\sim \G
_i$, $i=1,2$, then $\G_1=\G_2$. Indeed, since $X_1\overset{d}{=}X_2$, we have $\hat{\mathbb{E}}[\langle p,X_1\rangle]=\hat{\mathbb{E}}[\langle p,X_2\rangle]$ for all $p\in \mathbb{R}^d$. Hence, \eqref{eq:Charac for Gamma} gives $\G_1=\G_2$.
\end{remark}
If in addition, we assume $(\Omega, \mathcal{H}, \hat{\mathbb{E}})$ is regular, \cite[Theorem 1.2.2]{Pengbook} shows that for any $\mathbb E\in \Theta$, there exists a unique probability $P$ defined on $(\Omega,\sigma(\mathcal{H}))$ such that
\[
\mathbb E[X]=\int XdP=:\mathbb E_P[X], \ \text{ for all } \ X\in \mathcal{H}.
\]
Denote
\begin{equation}\label{Prob set for X}
    \mathcal{P}:=\{\text{all probabilities } P \text{ on } (\Omega,\sigma(\mathcal{H})) \text{ such that } \mathbb E_P\in \Theta\}.
\end{equation}
Then we have the following corollary.
 \begin{corollary}\label{coro:Charact of Gamma}
     Suppose that $(\Omega, \mathcal{H}, \hat{\mathbb{E}})$ is a regular sublinear expectation space. Then for any $X\in \mathcal{H}^d$ with $X\sim \Gamma$, we have
     \begin{equation*}
         \Gamma=\{\mathbb E_P[X]: P\in \mathcal{P}\},
     \end{equation*}
     where $\mathcal{P}$ is given by \eqref{Prob set for X}.
 \end{corollary}
 
	\subsection{Law of large numbers}
	In this subsection, we  prove the  LLN holds along $\mathbb N$ (that is, Theorem \ref{thm:LLN}) and all admissible sequences as follows.
	\begin{theorem}\label{Thm: LLN for subsequence}
		Assume that $\{X_t\}_{t\geq 0}$ (resp. $\{X_k\}_{k\geq 0}$) is a d-dimensional $\alpha$-mixing identical process (resp. sequence) on a sublinear expectation space $(\Omega,\mathcal{H},\hat{\mathbb{E}})$ for some $\alpha>0$ with $\hat{\mathbb{E}}[|X_0|^{2}]<\infty$. Let $\{t_k\}_{k\geq 1}$ (resp. $\{m_k\}_{k\geq 1}$) be a $(\gamma,\delta)$-admissible subsequence for some $0<\gamma<1, \delta>0$. Let
		\begin{equation}\label{0201-1}
			S_n:=\frac{1}{n}\sum_{i=1}^n X_{t_i} \ (resp. \ S_n:=\frac{1}{n}\sum_{i=1}^n X_{m_i}), \ \text{ and } \ X_0\sim \Gamma.
		\end{equation}
		Then there exists $C>0$ depending only on $(\alpha,\gamma,\delta)$, $c_X$ in Definition \ref{def: alpha-mixing} and $(c_1,c_2,N_0)$ in Definition \ref{Def: admissible}  such that for any $\phi\in C_{lip}(\mathbb{R}^d)$ with Lipschitz constant $l_{\phi}$,
		\begin{equation}\label{Convergence rate of LLN subsequence}
			\Big|\hat{\mathbb{E}}[\phi(S_n)]-\max_{x\in \Gamma}\phi(x) \Big|\leq C\bigl(1+\hat{\mathbb{E}}[|X_0|^{2}]\bigr)l_{\phi}\bigl(n^{-1/2}+n^{-(1-\gamma)}\bigr)\text{ for any }n\in\mathbb N.
		\end{equation}
		In particular, if $\gamma\leq 1/2$, we have the convergence rate $n^{-1/2}$.
	\end{theorem}

	Before the proof of Theorem \ref{thm:LLN} and Theorem \ref{Thm: LLN for subsequence}, we first give some notations and results here. Let $C_{b}^{l,k}(\mathbb R_+\times \mathbb R^d)$ be the collection of all $\mathbb R$-valued functions $v(t,x)$ with all bounded derivatives $\partial_t^iv, \nabla^j_xv$ for all $1\leq i\leq l$ and $1\leq j\leq k$ (note that $v$ itself does not need to be bounded), i.e., for all $1\leq i\leq l, 1\leq j\leq k$,
    \[
    |\partial_t^iv|_{\infty}:=\sup_{(t,x)\in \mathbb R_+\times \mathbb R^d}|\partial_t^iv(t,x)|<\infty, \ \ |\nabla_x^jv|_{\infty}:=\sup_{(t,x)\in \mathbb R_+\times \mathbb R^d}|\nabla_x^jv(t,x)|<\infty.
    \]
    Let $C_{b}^{l+\alpha,k+\beta}(\mathbb R_+\times \mathbb R^d)$ for some $\alpha,\beta\in (0,1]$ be the collection of all $v\in C_{b}^{l,k}(\mathbb R_+\times \mathbb R^d)$ such that $\partial_t^lv, \nabla^k_xv$ are $\alpha$-H{$\Ddot{\rm o}$}lder and $\beta$-H$\Ddot{\rm o}$lder continuous respectively, i.e.,
    \[
    \begin{split}
        |\partial_t^lv(\cdot,x)|_{\alpha}&:=\sup_{t,s\in \mathbb R_+, \ t\neq s}\frac{|\partial_t^lv(t,x)-\partial_t^lv(s,x)|}{|t-s|^{\alpha}}<\infty,\\
        |\nabla_x^kv(t,\cdot)|_{\beta}&:=\sup_{x,y\in \mathbb R^d, \ x\neq y}\frac{|\nabla_x^kv(t,x)-\nabla_x^kv(t,y)|}{|x-y|^{\beta}}<\infty.
    \end{split}
    \]
	The following result can be obtained by the same proof of 1-dimensional case in  \cite[Proposition 3.3]{Song2021}. So we omit the detail.
	\begin{lemma}\label{Lem: Song3}
		For $\phi\in C_{lip}(\mathbb{R}^d)$, let $v\in C_b^{1,1+\beta}(\mathbb{R}_+\times \mathbb{R}^d)$ with some $\beta\in (0,1]$ be the solution to the following PDE:
		\begin{equation}\label{PDE}
			\begin{cases}
				\partial_tv(t,x)-p(\nabla_xv(t,x))=f(t,x), \ (t,x)\in \mathbb{R}_+\times \mathbb{R}^d,\\
				v(0,x)=\phi(x),
			\end{cases}
		\end{equation}
		where $p(a)=\hat{\mathbb{E}}[\langle a, X\rangle]$ for some $X\in \mathcal{H}^d$ with $\hat{\mathbb{E}}[|X|^{1+\beta}]<\infty$. Then for any $0\leq t\leq \bar{t}\leq 1$, we have
		\begin{equation*}
			\begin{split}
				v(\bar{t},0)-\hat{\mathbb{E}}[v(t, (\bar{t}-t)X)]\geq -4(\bar{t}-t)^{\beta}\int_t^{\bar{t}}|\nabla_x v(s,\cdot)|_{\beta}ds\times \hat{\mathbb{E}}[|X|^{1+\beta}]-\int_t^{\bar{t}}\hat{\mathbb{E}}\big[-f(s,(\bar{t}-s)X)\big]ds,
			\end{split}
		\end{equation*}
		and
		\begin{equation*}
			\begin{split}
				v(\bar{t},0)-\hat{\mathbb{E}}[v(t,(\bar{t}-t)X)]\leq 4(\bar{t}-t)^{\beta}\int_t^{\bar{t}}|\nabla_x v(s,\cdot)|_{\beta}ds\times \hat{\mathbb{E}}[|X|^{1+\beta}]+\int_t^{\bar{t}}\hat{\mathbb{E}}\big[f(s,(\bar{t}-s)X)\big]ds.
			\end{split}
		\end{equation*}
	\end{lemma}
	
	Motivated by the proof of Theorem 3.4 in \cite{Song2021}, we have the following theorem.
	
	\begin{theorem}\label{Thm: 0127}
		Let $\{X_k\}_{k=1}^n$ ($n>1$ is finite) be a d-dimensional identical sequence on a sublinear expectation space $(\Omega,\mathcal{H},\hat{\mathbb{E}})$ such that $\hat{\mathbb{E}}[|X_1|^{1+\beta}]<\infty$ for some $\beta\in (0,1]$. Let $v\in C_b^{1,1+\beta}(\mathbb{R}_+\times \mathbb{R}^d)$ with some $\beta\in (0,1]$ be the solution to equation \eqref{PDE} with $p(a)=\hat{\mathbb{E}}[\langle a, X_1\rangle]$ and $S_n=\frac{1}{n}\sum_{i=1}^nX_i$. Then
		\begin{align}\label{0127-1}
				v(1,0)-\hat{\mathbb{E}}[\phi(S_n)]&\geq -\sum_{i=2}^n\bigg|\hat{\mathbb{E}}\Bigl[v\Bigl(1-\frac{i}{n}, \frac{1}{n}\sum_{j=1}^iX_j\Bigr)\Bigr]-\hat{\mathbb{E}}\Big[\hat{\mathbb{E}}\Bigl[v\bigl(1-\frac{i}{n}, x+\frac{X_i}{n}\bigr)\Bigr]\Big|_{x=\frac{1}{n}\sum_{j=1}^{i-1}X_j}\Big]\bigg|\notag\\
				&\ \ \ \ -4n^{-\beta}\int_0^1|\nabla_xv(s,\cdot)|_{\beta}ds\times \hat{\mathbb{E}}[|X_1|^{1+\beta}]+\underline{f},
		\end{align}
		where 
		$
			\underline{f}:=\inf_{\{0\leq t\leq 1,x\in \mathbb{R}^d\}}f(t,x)
	$, in which $f$ is from \eqref{PDE}.
	\end{theorem}
	
	\begin{proof}
		For any fixed $n\geq 1$, set
		\begin{equation*}
			W_{n,0}=0, \ \text{ and } \ W_{n,i}=\sum_{j=1}^i\frac{X_j}{n}, \ \text{ for all } \ 1\leq i\leq n,
		\end{equation*}
		and
		\begin{equation*}
			A_{n,i}=\hat{\mathbb{E}}\Bigl[v\bigl(1-\frac{i}{n}, W_{n,i}\bigr)\Bigr], \ \text{ for all } \ 0\leq i\leq n.
		\end{equation*}
		Note that $A_{n,0}=v(1,0)$ and $A_{n,n}=\hat{\mathbb{E}}[\phi(S_n)]$. Then 
		\begin{align}\label{0127-2}
				\hat{\mathbb{E}}[\phi(S_n)]-v(1,0)&=\sum_{i=1}^n(A_{n,i}-A_{n,i-1})\notag\\
				&=\sum_{i=1}^n\big(A_{n,i}-\hat{\mathbb{E}}\bigl[b_{n,i}(W_{n,i-1})\bigr]\big)+\sum_{i=1}^n\big(\hat{\mathbb{E}}\bigl[b_{n,i}(W_{n,i-1})\bigr]-\hat{\mathbb{E}}\bigl[c_{n,i}(W_{n,i-1})\bigr]\big)\notag\\
				&\leq \sum_{i=2}^n\big|A_{n,i}-\hat{\mathbb{E}}\bigl[b_{n,i}(W_{n,i-1})\bigr]\big|+\sum_{i=1}^n\sup_{x\in \mathbb{R}^d}\big(b_{n,i}(x)-c_{n,i}(x)\big),
		\end{align}
		where 
		\begin{equation}\label{def b and c}
			b_{n,i}(x)=\hat{\mathbb{E}}\Bigl[v\bigl(1-\frac{i}{n}, x+\frac{X_i}{n}\bigr)\Bigr], \ \text{ and } \ c_{n,i}(x)=v\big(1-\frac{i-1}{n},x\big).
		\end{equation}
		Now for any fixed $x\in \mathbb{R}^d$, let $u^x(t,y)=v(t,x+y)$. Then $u^x$ satisfies the following PDE:
		\begin{equation*}
			\begin{cases}
				\partial_tu^x(t,y)-p(\nabla_yu^x(t,y))=f(t,x+y), \ (t,y)\in \mathbb{R}_+\times \mathbb{R}^d,\\
				u^x(0,y)=\phi(x+y).
			\end{cases}
		\end{equation*}
		Then  letting $1-\frac{i}{n}=t<\bar{t}=1-\frac{i-1}{n}$ in Lemma \ref{Lem: Song3}, we have
		\begin{equation}\label{0127-3}
			\begin{split}
				b_{n,i}(x)-c_{n,i}(x)&=\hat{\mathbb{E}}\Big[u^x\Big(1-\frac{i}{n},\frac{X_i}{n}\Big)\Big]-u^x\Big(1-\frac{i-1}{n},0\Big)\\
				&\leq 4n^{-\beta}\int_{1-\frac{i}{n}}^{1-\frac{i-1}{n}}|\nabla_x v(s,\cdot)|_{\beta}ds\times \hat{\mathbb{E}}[|X_i|^{1+\beta}]\\
				&\ \ \ \ +\int_{1-\frac{i}{n}}^{1-\frac{i-1}{n}}\hat{\mathbb{E}}\big[-f\Big(s,x+\big(1-\frac{i-1}{n}-s\big)X_i\Big)\big]ds\\
				&\leq 4n^{-\beta}\int_{1-\frac{i}{n}}^{1-\frac{i-1}{n}}|\nabla_x v(s,\cdot)|_{\beta}ds\times \hat{\mathbb{E}}[|X_1|^{1+\beta}]-\frac{1}{n}\underline{f}.
			\end{split}
		\end{equation}
		Thus, \eqref{0127-1} follows from \eqref{0127-2} and \eqref{0127-3}.
	\end{proof}
	
	\begin{corollary}\label{Coro: 0127}
		Let $\{X_t\}_{t\geq 0}$ (resp. $\{X_k\}_{k\geq 0}$) be a d-dimensional $\alpha$-mixing identical process (resp. sequence) on a sublinear expectation space $(\Omega,\mathcal{H},\hat{\mathbb{E}})$ such that $\hat{\mathbb{E}}[|X_0|^{1+\beta}]<\infty$ for some $\beta\in (0,1]$ and $v\in C_b^{1,1+\beta}(\mathbb{R}_+\times \mathbb{R}^d)$ be the solution to equation \eqref{PDE} with $p(a)=\hat{\mathbb{E}}[\langle a, X_0\rangle]$. Assume that $\{t_k\}_{k\geq 1}$ (resp. $\{m_k\}_{k\geq 1}$) is a $(\gamma,\delta)$-admissible subsequence for some $0<\gamma<1, \delta>0$ and set
		\begin{equation*}
			Y_k:=X_{t_k} \ (resp. \ Y_k:=X_{m_k}),  \ \ S^Y_n=\frac{1}{n}\sum_{i=1}^nY_i.
		\end{equation*}
		Then for $n\geq N_0$, we have
		\begin{equation}\label{0127-4}
			\begin{split}
				v(1,0)-\hat{\mathbb{E}}[\phi(S^Y_n)]&\geq 
				-C(c_1,\gamma,\beta)n^{-\beta}\int_0^1|\nabla_xv(s,\cdot)|_{\beta}ds\times \hat{\mathbb{E}}[|X_0|^{1+\beta}]\\
				&\ \ \ \ -c_X|\nabla_xv|_{\infty}ne^{-\alpha c_2(\ln n)^{1+\delta}}-2c_1l_{\phi}n^{-(1-\gamma)}\hat{\mathbb{E}}[|X_0|]+\underline{f},
			\end{split}
		\end{equation}
  where  $c_X,c_1,c_2$ are from \eqref{equ: alpha-mixing}, \eqref{eq:admissible},
		$\underline{f}=\inf_{\{0\leq t\leq 1,x\in \mathbb{R}^d\}}f(t,x)
	$ with $f$ defined as in \eqref{PDE}.
	\end{corollary}
	
	\begin{proof}
		We only need to show the case $Y_k=X_{t_k}$. For any $n\geq N_0$, let us reorder $\{t_i\}_{1\leq i\leq n, i\notin \Pi_n}$ by $\{\tilde{t}_i\}_{1\leq i\leq n-|\Pi_n|}$.
		Set
		\begin{equation}\label{def: Z_i by reorder}
			Z_i:=X_{\tilde{t}_i},  \ \ S^Z_{n-|\Pi_n|}=\frac{1}{n-|\Pi_n|}\sum_{i=1}^{n-|\Pi_n|}Z_i.
		\end{equation}
		Then $\{Z_i\}_{1\leq i\leq n-|\Pi_n|}$ is an identical sequence. Let
		\begin{equation}\label{def: W sum of Z}
			W_i=\frac{1}{n-|\Pi_n|}\sum_{j=1}^{i}Z_j, \ \text{ for any } \ 1\leq i\leq n-|\Pi_n|.
		\end{equation}
		By Theorem \ref{Thm: 0127}, we have
		\begin{equation}\label{0222-1}
			\begin{split}
				&\ \ \ \ v(1,0)-\hat{\mathbb{E}}[\phi(S^Z_{n-|\Pi_n|})]\\
				&\geq -\sum_{i=2}^{n-|\Pi_n|}\bigg|\hat{\mathbb{E}}\Bigl[v\Bigl(1-\frac{i}{n-|\Pi_n|}, W_i\Bigr)\Bigr]-\hat{\mathbb{E}}\Big[\hat{\mathbb{E}}\Bigl[v\bigl(1-\frac{i}{n-|\Pi_n|}, x+\frac{Z_i}{n-|\Pi_n|}\bigr)\Bigr]\Big|_{x=W_{i-1}}\Big]\bigg|\\
				&\ \ \ \ -4(n-|\Pi_n|)^{-\beta}\int_0^1|\nabla_xv(s,\cdot)|_{\beta}ds\times \hat{\mathbb{E}}[|X_0|^{1+\beta}]+\underline{f}.
			\end{split}
		\end{equation}
		Set
		\begin{equation*}
			g_{i}(x,y):=v\bigl(1-\frac{i}{n-|\Pi_n|}, \frac{i-1}{n-|\Pi_n|}x+\frac{1}{n-|\Pi_n|}y\bigr), \ \ 2\leq i\leq n-|\Pi_n|.
		\end{equation*}
		Then the Lipschitz constant $l_{g_{i}}\leq |\nabla_xv|_{\infty}$, the fact that $|\tilde{t}_{i+1}-\tilde{t}_i|\geq c_2(\ln n)^{1+\delta}$ and the $\alpha$-mixing of $\{X_t\}_{t\ge0}$ give 
		\begin{equation}\label{0127-5}
			\begin{split}
				&\ \ \ \ \sum_{i=2}^{n-|\Pi_n|}\bigg|\hat{\mathbb{E}}\Bigl[v\Bigl(1-\frac{i}{n-|\Pi_n|}, W_i\Bigr)\Bigr]-\hat{\mathbb{E}}\Big[\hat{\mathbb{E}}\Bigl[v\bigl(1-\frac{i}{n-|\Pi_n|}, x+\frac{Z_i}{n-|\Pi_n|}\bigr)\Bigr]\Big|_{x=W_{i-1}}\Big]\bigg|\\
				&=\sum_{i=2}^{n-|\Pi_n|}\big|\hat{\mathbb{E}}[g_{i}(\bar{X}_{\Lambda_{i-1}}, X_{\tilde{t}_i})]-\hat{\mathbb{E}}\big[\hat{\mathbb{E}}[g_{i}(x, X_{\tilde{t}_i})]|_{x=\bar{X}_{\Lambda_{i-1}}}\big]\big|\\
				&\leq c_X|\nabla_xv|_{\infty}ne^{-\alpha c_2(\ln n)^{1+\delta}},
			\end{split}
		\end{equation}
		where $\Lambda_i:=\{\tilde{t}_1,\cdots,\tilde{t}_i\}, \ \ i\geq 2.$
		Note that  $|\Pi_n|\leq c_1n^{\gamma}$. Then
		\begin{equation}\label{0222-2}
			\big|\hat{\mathbb{E}}[\phi(S^Z_{n-|\Pi_n|})]-\hat{\mathbb{E}}[\phi(S^Y_n)]\big|\leq l_{\phi}\hat{\mathbb{E}}[|S^Z_{n-|\Pi_n|}-S^Y_n|]\leq \frac{2|\Pi_n|}{n}l_{\phi}\hat{\mathbb{E}}[|X_0|]\leq 2c_1l_{\phi}n^{-(1-\gamma)}\hat{\mathbb{E}}[|X_0|].
		\end{equation}
		This together with  \eqref{0222-1} and \eqref{0127-5} implies \eqref{0127-4}.
	\end{proof}
	
	For any positive integer $n$ and $\gamma_1, \cdots, \gamma_k>0$, we define $n_{\gamma_1,\cdots,\gamma_k}$ recursively by
	\begin{equation*}
		n_{\gamma_1}:=\lfloor n^{\gamma_1}\rfloor, \ \ n_{\gamma_1,\cdots,\gamma_i}:=\lfloor n_{\gamma_1,\cdots,\gamma_{i-1}}^{\gamma_i}\rfloor, \ \text{ for all } \ i\geq 2.
	\end{equation*}

	We also have the following corollary.
	
	\begin{corollary}\label{Thm: 1224}
		Let $\{X_k\}_{k\geq 1}$ be a d-dimensional $\alpha$-mixing and stationary sequence on a sublinear expectation space $(\Omega,\mathcal{H},\hat{\mathbb{E}})$ such that $\hat{\mathbb{E}}[|X_1|^{1+\beta}]<\infty$ for some $\beta\in (0,1]$. For any $0<\gamma_1,\gamma_2<1$, let $v^{n,\gamma}\in C_b^{1,1+\beta}(\mathbb{R}_+\times\mathbb{R}^d)$ be a solution to equation \eqref{PDE} with $p_{n,\gamma}(a)=\hat{\mathbb{E}}[\langle a, Y_{n,1}^{\gamma}\rangle]$ where
		\begin{equation*}
			Y_{n,1}^\gamma=\frac{1}{n_{\gamma_1}}\sum_{i=1}^{n_{\gamma_1}-n_{\gamma_1,\gamma_2}}X_i.
		\end{equation*}
		Then there exist $C(\gamma_1)>0$ and $C(\gamma_1,\gamma_2)>0$ such that 
		\begin{equation}\label{1220-1}
			\begin{split}
				v^{n,\gamma}(1,0)-\hat{\mathbb{E}}[\phi(S_n)]&\geq -C(\gamma_1)n^{-(1-\gamma_1)\beta}\int_0^1|\nabla_xv^{n,\gamma}(s,\cdot)|_{\beta}ds\times \hat{\mathbb{E}}[|X_1|^{1+\beta}]\\
				&\ \ \ \ -C(\gamma_1,\gamma_2)l_{\phi}\bigl(n^{-(1-\gamma_1)}+n^{-\gamma_1(1-\gamma_2)}\bigr)\hat{\mathbb{E}}[|X_1|]\\
				&\ \ \ \ -c_X|\nabla_xv^{n,\gamma}|_{\infty}n^{1-\gamma_1}e^{-\alpha(n^{\gamma_1\gamma_2}-1)}+\underline{f},
			\end{split}
		\end{equation}
       where $c_X$ is from \eqref{equ: alpha-mixing},
		$
			\underline{f}=\inf_{\{0\leq t\leq 1,x\in \mathbb{R}^d\}}f(t,x)
	$ with $f$ defined as in \eqref{PDE}.
	\end{corollary}
	
	\begin{proof}
		For any fixed $0<\gamma_1,\gamma_2<1$, we construct the double array variables $\{Y^{\gamma}_{n,k}\}$ by
		\begin{equation*}
			Y^{\gamma}_{n,k}=\frac{X_{(k-1)n_{\gamma_1}+1}+X_{(k-1)n_{\gamma_1}+2}+\cdots+X_{kn_{\gamma_1}-n_{\gamma_1,\gamma_2}}}{n_{\gamma_1}}, \quad 1\leq k\leq n_{1-\gamma_1}.
		\end{equation*}
        Since $\{X_k\}_{k\geq 1}$ is a stationary sequence, so is $\{Y^{\gamma}_{n,k}\}_{k=1}^{n_{1-\gamma_1}}$. Then it follows from Theorem \ref{Thm: 0127} that
		\begin{equation}\label{0128-3}
			\begin{split}
				&\ \ \ \ v^{n,\gamma}(1,0)-\hat{\mathbb{E}}\big[\phi\big(S^{Y,\gamma}_{n_{1-\gamma_1}}\big)\big]\\
				&\geq -\sum_{i=2}^{n_{1-\gamma_1}}\bigg|\hat{\mathbb{E}}\Bigl[v^{n,\gamma}\Bigl(1-\frac{i}{n_{1-\gamma_1}}, \sum_{j=1}^i\frac{Y^{\gamma}_{n,j}}{n_{1-\gamma_1}}\Bigr)\Bigr]-\hat{\mathbb{E}}\bigg[\hat{\mathbb{E}}\Bigl[v^{n,\gamma}\bigl(1-\frac{i}{n_{1-\gamma_1}}, x+\frac{Y^{\gamma}_{n,i}}{n_{1-\gamma_1}}\bigr)\Bigr]\Big|_{x=\sum_{j=1}^{i-1}\frac{Y^{\gamma}_{n,j}}{n_{1-\gamma_1}}}\bigg]\bigg|\\
				&\ \ \ \ -4n_{1-\gamma_1}^{-\beta}\int_0^1|\nabla_xv^{n,\gamma}(s,\cdot)|_{\beta}ds\times \hat{\mathbb{E}}[|Y^{\gamma}_{n,1}|^{1+\beta}]+\underline{f},
			\end{split}
		\end{equation}
		where
		\begin{equation*}
			S^{Y,\gamma}_{n_{1-\gamma_1}}=\frac{1}{n_{1-\gamma_1}}\sum_{i=1}^{n_{1-\gamma_1}}Y^{\gamma}_{n,i}.
		\end{equation*}
		Let
		\begin{equation*}
			g^{\gamma}_{n,i}(x,y)=v^{n,\gamma}\Big(1-\frac{i}{n_{1-\gamma_1}},\frac{(i-1)(n_{\gamma_1}-n_{\gamma_1,\gamma_2})}{n_{\gamma_1}n_{1-\gamma_1}}x+\frac{n_{\gamma_1}-n_{\gamma_1,\gamma_2}}{n_{\gamma_1}n_{1-\gamma_1}}y\Big), \ \ 2\leq i\leq n_{1-\gamma_1}.
		\end{equation*}
		Then the Lipschitz constant $l_{g^{\gamma}_{n,i}}$ of $g^{\gamma}_{n,i}$ satisfies
		\begin{equation*}
			l_{g^{\gamma}_{n,i}}\leq \sup_{z\in \mathbb{R}^d}\Big|\nabla_xv^{n,\gamma}\big(1-\frac{i}{n_{1-\gamma_1}},z\big)\Big|\leq |\nabla_xv^{n,\gamma}|_{\infty}.
		\end{equation*}
		For any $1\leq i\leq n_{1-\gamma_1}$, denote $\Lambda_0:=\emptyset$ and
		\begin{equation*}
			\Lambda_i:=\{(i-1)n_{\gamma_1}+1,(i-1)n_{\gamma_1}+2,\cdots,in_{\gamma_1}-n_{\gamma_1,\gamma_2}\}, \ \ \Lambda_i^-:=\cup_{0\leq j\leq i-1}\Lambda_j.
		\end{equation*}
		Then for all $2\leq i\leq n_{1-\gamma_1}$, we know that $\Lambda_i^-<\Lambda_i$ with $d(\Lambda_i^-,\Lambda_i)=n_{\gamma_1,\gamma_2}+1$ and the $\alpha$-mixing property gives
		\begin{equation}\label{0128-2}
			\begin{split}
				&\ \ \ \ \bigg|\hat{\mathbb{E}}\Bigl[v^{n,\gamma}\Bigl(1-\frac{i}{n_{1-\gamma_1}}, \sum_{j=1}^i\frac{Y^{\gamma}_{n,j}}{n_{1-\gamma_1}}\Bigr)\Bigr]-\hat{\mathbb{E}}\bigg[\hat{\mathbb{E}}\Bigl[v^{n,\gamma}\bigl(1-\frac{i}{n_{1-\gamma_1}}, x+\frac{Y^{\gamma}_{n,i}}{n_{1-\gamma_1}}\bigr)\Bigr]\Big|_{x=\sum_{j=1}^{i-1}\frac{Y^{\gamma}_{n,j}}{n_{1-\gamma_1}}}\bigg]\bigg|\\
				&=\Big|\hat{\mathbb{E}}\bigl[g^{\gamma}_{n,i}(\bar{X}_{\Lambda_i^-}, \bar{X}_{\Lambda_i})\bigr]-\hat{\mathbb{E}}\Big[\hat{\mathbb{E}}\bigl[g^{\gamma}_{n,i}(x, \bar{X}_{\Lambda_i})\bigr]\big|_{x=\bar{X}_{\Lambda_i^-}}\Big]\Big|\\
				&\leq c_Xl_{g^{\gamma}_{n,i}}e^{-\alpha (n_{\gamma_1,\gamma_2}+1)}\leq c_X|\nabla_xv^{n,\gamma}|_{\infty}e^{-\alpha(n^{\gamma_1\gamma_2}-1)}.
			\end{split}
		\end{equation}
		Note that 
		$n_{1-\gamma_1}^{-\beta}\leq C(\gamma_1)n^{-(1-\gamma_1)\beta}$ and
		\begin{equation*}
			\hat{\mathbb{E}}[|Y^{\gamma}_{n,1}|^{1+\beta}]\leq \hat{\mathbb{E}}\bigg[\bigg|\frac{\sum_{i=1}^{n_{\gamma_1}}|X_i|}{n_{\gamma_1}}\bigg|^{1+\beta}\bigg]\leq \frac{\sum_{i=1}^{n_{\gamma_1}}\hat{\mathbb{E}}\bigl[|X_i|^{1+\beta}\bigr]}{n_{\gamma_1}}=\hat{\mathbb{E}}[|X_1|^{1+\beta}].
		\end{equation*}
		Then it follows from \eqref{0128-3} and \eqref{0128-2} that
		\begin{equation*}
			\begin{split}
				&\ \ \ \ v^{n,\gamma}(1,0)-\hat{\mathbb{E}}\big[\phi\big(S^{Y,\gamma}_{n_{1-\gamma_1}}\big)\big]\\
				&\geq -c_X|\nabla_xv^{n,\gamma}|_{\infty}n^{1-\gamma_1}e^{-\alpha(n^{\gamma_1\gamma_2}-1)}-C(\gamma_1)n^{-(1-\gamma_1)\beta}\int_0^1|\nabla_xv^{n,\gamma}(s,\cdot)|_{\beta}ds\times \hat{\mathbb{E}}[|X_1|^{1+\beta}]+\underline{f}.
			\end{split}
		\end{equation*}
		Note also that
		\begin{equation}\label{eq:1115-2}
			\begin{split}
				\big|\hat{\mathbb{E}}\big[\phi\big(S^{Y,\gamma}_{n_{1-\gamma_1}}\big)\big]-\hat{\mathbb{E}}[\phi(S_n)]\big|&\leq l_{\phi}\Bigl(\hat{\mathbb{E}}\Bigl[\Big|S^{Y,\gamma}_{n_{1-\gamma_1}}-\frac{n}{n_{\gamma_1}n_{1-\gamma_1}}S_n\Big|\Bigr]+\frac{n-n_{\gamma_1}n_{1-\gamma_1}}{n_{\gamma_1}n_{1-\gamma_1}}\hat{\mathbb{E}}[|S_n|]\Bigr)\\
				&\leq l_{\phi}\frac{2(n-n_{\gamma_1}n_{1-\gamma_1})+n_{1-\gamma_1}n_{\gamma_1,\gamma_2}}{n_{\gamma_1}n_{1-\gamma_1}}\hat{\mathbb{E}}[|X_1|]\\
				&\leq C(\gamma_1,\gamma_2)l_{\phi}\bigl(n^{-(1-\gamma_1)}+n^{-\gamma_1(1-\gamma_2)}\bigr)\hat{\mathbb{E}}[|X_1|].
			\end{split}
		\end{equation}
		Then \eqref{1220-1} follows.
	\end{proof}

 In the case that $\{X_k\}_{k\geq 1}$ is a $d$-dimensional i.i.d. sequence, similar to Theorem 3.4 in \cite{Song2021}, we have the following lemma.
    \begin{lemma}\label{lem:est for iid}
        Let $\{X_k\}_{k\geq 1}$ be a d-dimensional i.i.d. sequence such that $\hat{\mathbb{E}}[|X_1|^{1+\beta}]<\infty$ for some $\beta\in (0,1]$ and $v\in C_b^{1,1+\beta}(\mathbb{R}_+\times \mathbb{R}^d)$ be the solution to equation \eqref{PDE} with $p(a)=\hat{\mathbb{E}}[\langle a, X_1\rangle]$. Set $S_n=\frac{1}{n}\sum_{i=1}^nX_i$. Then we have
        \begin{equation}\label{eq:est for iid}
           \begin{split}
				v(1,0)-\hat{\mathbb{E}}[\phi(S_n)]&\geq -4n^{-\beta}\int_0^1|\nabla_xv(s,\cdot)|_{\beta}ds\times \hat{\mathbb{E}}[|X_1|^{1+\beta}]+\underline{f},
			\end{split}
        \end{equation}
        where 
		$
			\underline{f}=\inf_{\{0\leq t\leq 1,x\in \mathbb{R}^d\}}f(t,x)
	$ with $f$ defined as in \eqref{PDE}.
    \end{lemma}
	
	By Corollary 2.2.13 in \cite{Pengbook}, we know that for any non-empty bounded closed convex subset $\Gamma\subset \mathbb{R}^d$, $v(t,x):=\sup_{y\in \Gamma}\phi(x+ty)$ is the viscosity solution to the following $G$-equation:
	\begin{equation*}
		\begin{cases}
			\partial_tv(t,x)-p_{\Gamma}(\nabla_xv(t,x))=0, \ (t,x)\in \mathbb{R}_+\times \mathbb{R}^d,\\
			v(0,x)=\phi(x),
		\end{cases}
	\end{equation*}
	where $p_{\Gamma}(a):=\sup_{y\in \Gamma}\langle a,y\rangle=\hat{\mathbb E}[\langle a, X\rangle]$ if $X\sim \Gamma$. However, the regularity of the viscosity solution $v$ is not enough ($v$ is only Lipschitz continuous) to apply Theorem \ref{Thm: 0127}. 
	
	Now motivated by \cite{Krylov2020} (see also \cite{Song2021}), we take a non-negative $\zeta\in C^{\infty}(\mathbb{R}^{1+d})$ supported in $\{(t,x): 0<t<1, |x|<1\}$, where $C^{\infty}(\mathbb{R}^{1+d})$ is the collection of all smooth functions, such that
	\begin{equation}\label{eq:huang}
\int_{\mathbb{R}^{1+d}}\zeta(t,x)dtdx=1.
	\end{equation}
	For any $\epsilon>0$, let $\zeta_{\epsilon}(t,x):=\epsilon^{-(1+d)}\zeta(\epsilon^{-1}t, \epsilon^{-1}x)$, and then for any locally integrable $u(t,x)$, the convolution $u_{\epsilon}(t,x)=u*\zeta_{\epsilon}\in C^{\infty}(\mathbb{R}^{1+d})$.
	
 Lemma 4.1 in \cite{Song2021} shows the following result for $d=1$. By a similar argument, we can check it holds for all $d\in\mathbb N$, and thus we omit the proof.
	\begin{lemma}\label{Lem: song4.1}
		For given $\phi\in C_{lip}(\mathbb{R}^d)$ and $\Gamma$ a bounded closed convex subset of $\mathbb{R}^d$, set $v(t,x):=\sup_{y\in \Gamma}\phi(x+ty)$. Then there exists a non-negative function $f^{\epsilon}:\mathbb{R}^d\times\mathbb{R}\to \mathbb{R}_+$ such that $v_{\epsilon}$ satisfies
		\begin{equation*}
			\partial_tv_{\epsilon}(t,x)-p_{\Gamma}(\nabla_xv_{\epsilon}(t,x))=f^{\epsilon}(t,x), \ \text{ for all } \ t>\epsilon, x\in \mathbb{R}^d.
		\end{equation*}
	\end{lemma}
	
	From Lemma 2.3 in \cite{Krylov2020}, we have the following properties of convolutions.
	
	\begin{lemma}\label{Lem: property of convolution}
		For a given $u:\mathbb{R}_+\times\mathbb{R}^d\to \mathbb{R}$, assume that there exist $\beta\in (0,1]$, and $a_1,a_2,a_3\geq 0$ such that
		\begin{equation*}
			\begin{split}
				&|u(t,x)-u(t,y)|\leq a_1|x-y|^{\beta},\\
				&|u(t,x)-u(s,x)|\leq a_2|t-s|^{\beta}+a_3.
			\end{split}
		\end{equation*}
		Then for any $\epsilon\in (0,1)$, $t,s\in(\epsilon,\infty)$ and $x,y\in \mathbb{R}^d$, we have
		\begin{equation*}
			\begin{split}
				&|u_{\epsilon}(t,x)-u(t,x)|\leq (a_1+a_2)\epsilon^{\beta}+a_3,\\
				&|u_{\epsilon}(t,x)-u_{\epsilon}(t,y)|\leq a_1|x-y|^{\beta},\\
				&|u_{\epsilon}(t,x)-u_{\epsilon}(s,x)|\leq a_2|t-s|^{\beta}+a_3,\\
				&|\partial_tu_{\epsilon}(t,x)-\partial_tu_{\epsilon}(t,y)|\leq a_1|\partial_t\zeta|_{L^1}\epsilon^{-1}|x-y|^{\beta},\\
				&|\nabla_xu_{\epsilon}(t,x)-\nabla_xu_{\epsilon}(t,y)|\leq a_1|\nabla_x\zeta|_{L^1}\epsilon^{-1}|x-y|^{\beta},\\
				&|\partial_tu_{\epsilon}(t,x)-\partial_tu_{\epsilon}(s,x)|\leq |\partial_t\zeta|_{L^1}\epsilon^{-1}(a_2|t-s|^{\beta}+a_3),\\
				&|\nabla_xu_{\epsilon}(t,x)-\nabla_xu_{\epsilon}(s,x)|\leq |\nabla_x\zeta|_{L^1}\epsilon^{-1}(a_2|t-s|^{\beta}+a_3),
			\end{split}
		\end{equation*}
		where
		\begin{equation*}
			|\eta|_{L^1}:=\int_{\mathbb{R}^{1+d}}|\eta(t,x)|dtdx, \ \ \eta=\partial_t\zeta, \nabla_x\zeta.
		\end{equation*}
	\end{lemma}
	
	For any given $Z\in \mathcal{H}^d$ on a sublinear expectation space $(\Omega, \mathcal{H}, \hat{\mathbb{E}})$ with $\hat{\mathbb{E}}[|Z|^{2}]<\infty$ and $\phi\in C_{lip}(\mathbb{R}^d)$, we define $v_n^Z$ for any $n\in \mathbb{N}$ by
	\begin{equation}\label{1225-2}
		v_n^Z(0,x):=\phi(x), \ v_n^Z\big(\frac{k}{n},x\big)=\hat{\mathbb{E}}\Bigl[v_n^Z\bigl(\frac{k-1}{n},x+\frac{Z}{n}\bigr)\Bigr], \ k\geq 1,
	\end{equation}
	and
	\begin{equation}\label{1225-3}
		v_n^Z(t,x)=v_n^Z\big(\frac{k}{n},x\big), \ \text{ for all } \ \frac{k-1}{n}<t\leq \frac{k}{n}, \ k\geq 1.
	\end{equation}
	It is easy to see that $v_n^{Z_1}=v_n^{Z_2}$ if $Z_1$ and $Z_2$ are identically distributed. In the following,  we will use this construction multiple times.

	The proof of the next lemma is similar to that of Lemma 4.2 in \cite{Song2021}. So we omit it here.
	\begin{lemma}\label{Lem: 1226-1}
		For $\phi\in C_{lip}(\mathbb{R}^d)$ with Lipschitz constant $l_{\phi}$, let $v_n^Z$ be defined by \eqref{1225-2}-\eqref{1225-3}. Then for $\epsilon>0$, there exists a function $f_n^{(\epsilon)}:\mathbb{R}^d\times \mathbb{R}\to\mathbb{R}$ such that for $ t>\epsilon+\frac{1}{n},$ $ x\in \mathbb{R}^d$, $f_n^{(\epsilon)}(t,x)\geq -\frac{C_{\zeta}l_{\phi}}{n\epsilon}\big(\hat{\mathbb{E}}[|Z|]+\hat{\mathbb{E}}[|Z|^2]\big)$  and 
		\begin{equation*}
			\partial_tv^Z_{n,\epsilon}(t,x)-p_Z(\nabla_xv^Z_{n,\epsilon}(t,x))=f_n^{(\epsilon)}(t,x),
		\end{equation*}
		where 
	$
			p_Z(a):=\hat{\mathbb{E}}[\langle a,Z\rangle], \ a\in \mathbb{R}^d$, $C_{\zeta}=\max\Big\{\frac{5}{2}|\nabla_x\zeta|_{L^1}, 2|\partial_t\zeta|_{L^1}\Big\}
	$ and $\zeta$ is from \eqref{eq:huang}.
	\end{lemma}

	We also need the following lemmas.
	
	\begin{lemma}\label{Lem: identical sequence}
		Assume that $\{X_k\}_{k\geq 1}$ is a $d$-dimensional identical sequence on a sublinear expectation space $(\Omega,\mathcal{H},\hat{\mathbb{E}})$. Set 
		\begin{equation*}
			S_{n,0}=0, \ S_{n,m}=\frac{X_1+\cdots+X_m}{n}, \ m\geq 1,
		\end{equation*}
		Let $Z\deq X_1$ and $v^Z_n$ be defined by \eqref{1225-2}-\eqref{1225-3} with $v^Z_n(0,x)=\phi(x)$ for some $\phi\in C_{lip}(\mathbb{R}^d)$. Then we have
		\begin{equation}\label{ineq: induction 1226}
	\Big|v^Z_n\bigl(1,0)-\hat{\mathbb{E}}[\phi(S_{n,n})]\Big|\le \sum_{k=1}^{n}\Big|\hat{\mathbb{E}}\Bigl[v^Z_n(\frac{n-k}{n},S_{n,k})\Bigr]-\hat{\mathbb{E}}\Bigl[\hat{\mathbb{E}}\Bigl[v^Z_n(\frac{n-k}{n},x+\frac{X_{k}}{n})\Bigr]\Big|_{x=S_{n,k-1}}\Bigr]\Big|.
		\end{equation}
	\end{lemma}
	
	\begin{proof}
		Note that 
		\begin{equation*}
			\begin{split}
				&\ \ \ \ \hat{\mathbb{E}}[\phi(S_{n,n})]-v^Z_n\bigl(1,0)\\
				&=\sum_{k=0}^{n-1}\Bigl(\hat{\mathbb{E}}\Bigl[v^Z_n(\frac{k}{n},S_{n,n-k})\Bigr]-\hat{\mathbb{E}}\Bigl[v^Z_n(\frac{k+1}{n},S_{n,n-k-1})\Bigr]\Bigr)\\
				&=\sum_{k=0}^{n-1}\Bigl(\hat{\mathbb{E}}\Bigl[v^Z_n(\frac{k}{n},S_{n,n-k})\Bigr]-\hat{\mathbb{E}}\Bigl[\hat{\mathbb{E}}\Bigl[v^Z_n(\frac{k}{n},x+\frac{X_{n-k}}{n})\Bigr]\Big|_{x=S_{n,n-k-1}}\Bigr]\Bigr).
			\end{split}
		\end{equation*}
		Then \eqref{ineq: induction 1226} holds.
	\end{proof}
	
	\begin{lemma}\label{Lem: 1226}
		Given $X_1,X_2\in \mathcal{H}^d$ on a sublinear expectation space $(\Omega,\mathcal{H},\hat{\mathbb{E}})$, let $\Gamma_i$ such that $X_i\sim \Gamma_i$, $i=1,2$. For any $\phi\in C_{lip}(\mathbb{R}^d)$ with Lipschitz constant $l_\phi$, define $v_i(t,x):=\sup_{y\in \Gamma_i}\phi(x+ty)$, $i=1,2$. Then for any $(t,x)\in \mathbb{R}_+\times \mathbb R^d$,
		\begin{equation*}
			|v_1(t,x)-v_2(t,x)|\leq tl_{\phi}\hat{\mathbb{E}}[|X_1-X_2|].
		\end{equation*}
	\end{lemma}
	\begin{proof}
		Since $X_i\sim \Gamma_i$, it follows from Lemma \ref{lem:Charact of Gamma} that 
	    $\Gamma_i=\{\mathbb{E}[X_i]: \mathbb{E}\in \Theta\},$
		where $\Theta$ is given in \eqref{Def of Theta}, $i=1,2$. Given any $(t,x)\in \mathbb{R}_+\times \mathbb R^d$, let $y_i\in\Gamma_i$ be such that   $v_i(x,t)=\phi(x+ty_i)$, $i=1,2$. Then for each $i=1,2$, there exists $\mathbb E_i\in\Theta$ such that  $y_i=\mathbb{E}_i[X_i]$. Meanwhile, set $y_1'=\mathbb{E}_1[X_2]\in \Gamma_2, \ y_2'=\mathbb{E}_2[X_1]\in \Gamma_1$. Then  for any given $(t,x)\in \mathbb{R}_+\times \mathbb R^d$,
		\begin{equation*}
			v_1(t,x)-v_2(t,x)\leq \phi(x+ty_1)-\phi(x+ty_1')\leq tl_{\phi}|\mathbb{E}_{1}[X_1]-\mathbb{E}_{1}[X_2]|\leq tl_{\phi}\hat{\mathbb{E}}[|X_1-X_2|],
		\end{equation*}
		and
		\begin{equation*}
			v_1(t,x)-v_2(t,x)\geq \phi(x+ty_2')-\phi(x+ty_2)\geq -tl_{\phi}|\mathbb{E}_{2}[X_1]-\mathbb{E}_{2}[X_2]|\geq -tl_{\phi}\hat{\mathbb{E}}[|X_1-X_2|].
		\end{equation*}
		Thus, we have the desired result.
	\end{proof}
	
	Now we give the proof of Theorem \ref{Thm: LLN for subsequence}.
	
	\begin{proof}[Proof of Theorem \ref{Thm: LLN for subsequence}]
		We only need to prove the case $\{X_{t_k}\}_{k\geq 1}$ and the case $\{X_{m_k}\}_{k\geq 1}$ is similar.
		
		For any $\phi\in C_{lip}(\mathbb{R}^d)$ with the Lipschitz constant $l_{\phi}$, let $v(t,x):=\sup_{y\in \Gamma}\phi(x+ty)$ where $X_0\sim\Gamma$ and $v_n(t,x)$ is defined as in \eqref{1225-2}-\eqref{1225-3} for $Z=X_0$. We need to estimate $|\hat{\mathbb{E}}[\phi(S_n)]-v(1,0)|$. The estimation is divided into two steps (from above and from below).
		
		\textbf{Step 1: The estimation from above}. It is easy to show that $|v(t,x)-v(t,y)|\leq l_{\phi}|x-y|$. Moreover, since $v(r+\delta,x)=\sup_{y\in \Gamma}v(r,x+\delta y)$, then
		\begin{equation*}
			\big|v(t,x)-v(s,x)\big|\leq l_{\phi}\hat{\mathbb{E}}[|X_0|]|t-s|.
		\end{equation*}
		In fact, we can assume $t\geq s$ and Lemma \ref{lem:Charact of Gamma} yields
		\begin{equation*}
			\begin{split}
				\big|v(t,x)-v(s,x)\big|&=\Big|\sup_{y\in \Gamma}v(s,x+(t-s)y)-v(s,x)\Big|\\
				&\leq l_{\phi}|t-s|\sup_{y\in \Gamma}|y|\leq l_{\phi}|t-s|\sup_{\mathbb{E}\in \Theta}|\mathbb{E}[X_0]|\leq l_{\phi}\hat{\mathbb{E}}[|X_0|]|t-s|.
			\end{split}
		\end{equation*}
		Then by Corollary \ref{Coro: 0127}, Lemma \ref{Lem: song4.1} and Lemma \ref{Lem: property of convolution}, we conclude that for any $\epsilon\in (0,1)$ and $n\geq N_0$,
		\begin{equation*}
			\begin{split}
				v_{\epsilon}(1+\epsilon,0)-\hat{\mathbb{E}}\bigl[v_{\epsilon}(\epsilon,S_n)\bigr]&\geq -C(c_1,\gamma)n^{-1}\int_0^1|\nabla_xv_{\epsilon}
				(s+\epsilon,\cdot)|_{1}ds\times \hat{\mathbb{E}}[|X_0|^{2}]\\
				&\ \ \ \ -c_X\big|\nabla_xv_{\epsilon}(\cdot+\epsilon,\cdot)\big|_{\infty}ne^{-\alpha c_2(\ln n)^{1+\delta}}-2c_1l_{\phi}n^{-(1-\gamma)}\hat{\mathbb{E}}[|X_0|]\\
				&\geq -C(c_1,\gamma)|\nabla_x\zeta|_{L^1}\hat{\mathbb{E}}[|X_0|^{2}]l_{\phi}n^{-1}\epsilon^{-1}\\
				&\ \ \ \ -C(c_X,\alpha,c_2,\delta)l_{\phi}n^{-1/2}-2c_1\hat{\mathbb{E}}[|X_0|]l_{\phi}n^{-(1-\gamma)}\\
				&\geq -C(c_X,c_1,c_2,\alpha,\gamma,\delta,\zeta)\bigl(1+\hat{\mathbb{E}}[|X_0|^{2}]\bigr)l_{\phi}\bigl(n^{-1}\epsilon^{-1}+n^{-1/2}+n^{-(1-\gamma)}\bigr).
			\end{split}
		\end{equation*}
		By Lemma \ref{Lem: property of convolution} again, we obtain that
		\begin{equation*}
			\big|v_{\epsilon}(1+\epsilon,0)-v(1,0)\big|\leq \big|v_{\epsilon}(1+\epsilon,0)-v_{\epsilon}(1,0)\big|+\big|v_{\epsilon}(1,0)-v(1,0)\big|\leq (2\hat{\mathbb{E}}[|X_0|]+1)l_{\phi}\epsilon,
		\end{equation*}
		and
		\begin{equation*}
			\begin{split}
				\big|\hat{\mathbb{E}}\bigl[v_{\epsilon}(\epsilon,S_n)\bigr]-\hat{\mathbb{E}}[\phi(S_n)]\big|&\leq \big|\hat{\mathbb{E}}\bigl[v_{\epsilon}(\epsilon,S_n)\bigr]-\hat{\mathbb{E}}\big[v(\epsilon,S_n)\big]\big|+\big|\hat{\mathbb{E}}\bigl[v(\epsilon,S_n)\bigr]-\hat{\mathbb{E}}[\phi(S_n)]\big|\\
				&\leq (2\hat{\mathbb{E}}[|X_0|]+1)l_{\phi}\epsilon.
			\end{split}
		\end{equation*}
		Hence,
		\begin{equation*}
			v(1,0)-\hat{\mathbb{E}}[\phi(S_n)]\geq -C(c_X,c_1,c_2,\alpha,\gamma,\delta,\zeta)\bigl(1+\hat{\mathbb{E}}[|X_0|^{2}]\bigr)l_{\phi}\bigl(\epsilon+n^{-1}\epsilon^{-1}+n^{-1/2}+n^{-(1-\gamma)}\bigr).
		\end{equation*}
		Letting $\epsilon=n^{-1/2}$, we have
		\begin{equation}\label{0222-10}
			v(1,0)-\hat{\mathbb{E}}[\phi(S_n)]\geq -C(c_X,c_1,c_2,\alpha,\gamma,\delta,\zeta)\bigl(1+\hat{\mathbb{E}}[|X_0|^{2}]\bigr)l_{\phi}\bigl(n^{-1/2}+n^{-(1-\gamma)}\bigr).
		\end{equation}
		
		\textbf{Step 2: The estimation from below}. Let $\{\xi_k\}_{k\geq 1}$ be an i.i.d. sequence with $\Gamma$-maximally distributed random variables on $(\Omega, \mathcal{H}, \hat{\mathbb{E}})$, i.e., $\hat{\mathbb{E}}[\phi(\xi_k)]=\max_{y\in \Gamma}\phi(y)$ for all $\phi\in C_{lip}(\mathbb{R}^d), \ k\geq 1$. Set $S_n^{\xi}=\frac{\xi_1+\cdots+\xi_n}{n}$, and then for any $n\geq 1$, $S_n^{\xi}$ is also $\Gamma$-maximally distributed, as $\{\xi_k\}_{k\ge1}$ is i.i.d. 
		
		For any $n\geq N_0$, we consider $v_{n-|\Pi_n|,\epsilon}(\cdot+\epsilon+\frac{1}{n-|\Pi_n|},\cdot)\in C^{1,2}_{b}(\mathbb{R}\times \mathbb{R}^d)$.
		By Lemma \ref{lem:est for iid} and Lemma \ref{Lem: 1226-1}, we know that
		\begin{equation*}
			\begin{split}
				&\ \ \ \ v_{n-|\Pi_n|,\epsilon}\Big(1+\epsilon+\frac{1}{n-|\Pi_n|},0\Big)-\hat{\mathbb{E}}\Bigl[v_{n-|\Pi_n|,\epsilon}\Big(\epsilon+\frac{1}{n-|\Pi_n|},S_{n-|\Pi_n|}^{\xi}\Big)\Bigr]\\
				&\geq -\frac{4}{n-|\Pi_n|}\int_0^1\Big|\nabla_xv_{n-|\Pi_n|,\epsilon}\Big(s+\epsilon+\frac{1}{n-|\Pi_n|},\cdot\Big)\Big|_1ds\times \hat{\mathbb{E}}[|\xi_1|^2]\\
				&\ \ \ \ -\frac{C_{\zeta}l_{\phi}}{(n-|\Pi_n|)\epsilon}\big(\hat{\mathbb{E}}[|X_0|]+\hat{\mathbb{E}}[|X_0|^2]\big).
			\end{split}
		\end{equation*}
		Note that for all $t\geq 0, x,y\in \mathbb{R}^d$,
		\begin{equation}\label{ineq: v_n in x}
			\Big|v_{n-|\Pi_n|}(t,x)-v_{n-|\Pi_n|}(t,y)\Big|\leq l_{\phi}|x-y|,
		\end{equation}
		which combined with Lemma \ref{Lem: property of convolution}, gives 
		\begin{equation*}
			\big|\nabla_xv_{n-|\Pi_n|,\epsilon}\bigl(t,\cdot\bigr)\big|_1\leq l_{\phi}|\nabla_x\zeta|_{L^1}\epsilon^{-1}, \ \text{ for all } \ t\geq \epsilon.
		\end{equation*}
		Note also that $\hat{\mathbb{E}}[|\xi_1|^2]=\max_{y\in \Gamma}|y|^2\leq \hat{\mathbb{E}}[|X_0|^2]$. Then for all $n\geq N_0$,
		\begin{equation}\label{0222-7}
			\begin{split}
				&\ \ \ \ v_{n-|\Pi_n|,\epsilon}\Big(1+\epsilon+\frac{1}{n-|\Pi_n|},0\Big)-\hat{\mathbb{E}}\Bigl[v_{n-|\Pi_n|,\epsilon}\Big(\epsilon+\frac{1}{n-|\Pi_n|},S_{n-|\Pi_n|}^{\xi}\Big)\Bigr]\\
				&\geq -C(c_1,\gamma,\zeta)\big(1+\hat{\mathbb{E}}[|X_0|^2]\big)l_{\phi}n^{-1}\epsilon^{-1}.
			\end{split}
		\end{equation}
		On the other hand, by the construction of  $v_{n-|\Pi_n|}$ in \eqref{1225-2}-\eqref{1225-3}, we conclude that for any $t,s\in \mathbb{R}_+$, $x\in \mathbb{R}^d$,
		\begin{equation}\label{ineq: v_n in t}
			\Big|v_{n-|\Pi_n|}(t,x)-v_{n-|\Pi_n|}(s,x)\Big|\leq l_{\phi}\hat{\mathbb{E}}[|X_0|]\Big(|t-s|+\frac{1}{n-|\Pi_n|}\Big).
		\end{equation}
	 It follows from \eqref{ineq: v_n in x} and \eqref{ineq: v_n in t}  that Lemma \ref{Lem: property of convolution} gives
		\begin{equation}\label{0222-5}
			\begin{split}
				&\ \ \ \ \Big|v_{n-|\Pi_n|,\epsilon}\Big(1+\epsilon+\frac{1}{n-|\Pi_n|},0\Big)-\hat{\mathbb{E}}[\phi(S_n)]\Big|\\
				&\leq \Big|v_{n-|\Pi_n|,\epsilon}\Big(1+\epsilon+\frac{1}{n-|\Pi_n|},0\Big)-v_{n-|\Pi_n|}\Big(1+\epsilon+\frac{1}{n-|\Pi_n|},0\Big)\Big|\\
				&\ \ \ \ +\Big|v_{n-|\Pi_n|}\Big(1+\epsilon+\frac{1}{n-|\Pi_n|},0\Big)-v_{n-|\Pi_n|}(1,0)\Big|\\
				&\ \ \ \ +\big|v_{n-|\Pi_n|}(1,0)-\hat{\mathbb{E}}[\phi(S_n)]\big|\\
				&\leq (1+2\hat{\mathbb{E}}[|X_0|])l_{\phi}\epsilon+3\hat{\mathbb{E}}[|X_0|]l_{\phi}\frac{1}{n-|\Pi_n|}+\big|v_{n-|\Pi_n|}(1,0)-\hat{\mathbb{E}}[\phi(S_n)]\big|.
			\end{split}
		\end{equation}
		Since $\{t_k\}_{k\geq 1}$ is $(\gamma,\delta)$-admissible, for any $n\geq N_0$, we reorder $\{t_i\}_{1\leq i\leq n, i\notin \Pi_n}$ by $\{\tilde{t}_i\}_{1\leq i\leq n-|\Pi_n|}$ and define $S^Z_{n-|\Pi_n|}$ and $Z_i, W_i$ for $1\leq i\leq n-|\Pi_n|$ as in \eqref{def: Z_i by reorder}-\eqref{def: W sum of Z}. Similar to \eqref{0127-5}, it follows from Lemma \ref{Lem: identical sequence} and \eqref{ineq: v_n in x} that
		\begin{equation}\label{0222-4}
			\begin{split}
				&\ \ \ \ \big|v_{n-|\Pi_n|}(1,0)-\hat{\mathbb{E}}\big[\phi\big(S^Z_{n-|\Pi_n|}\big)\big]\big|\\
				&\leq \sum_{i=1}^{n-|\Pi_n|}\bigg|\hat{\mathbb{E}}\Bigl[v_{n-|\Pi_n|}\big(1-\frac{i}{n-|\Pi_n|},W_{i}\big)\Bigr]\\
                &\qquad\qquad \ \ -\hat{\mathbb{E}}\Bigl[\hat{\mathbb{E}}\Bigl[v_{n-|\Pi_n|}\big(1-\frac{i}{n-|\Pi_n|},x+\frac{X_{\tilde{t}_i}}{n-|\Pi_n|}\big)\Bigr]\Big|_{x=W_{i-1}}\Bigr]\bigg|\\
				&\leq c_Xl_{\phi}ne^{-\alpha c_2(\ln n)^{1+\delta}}\leq C(c_2,c_X,\alpha,\delta)l_{\phi}n^{-1/2}.
			\end{split}
		\end{equation}
		Similar to \eqref{0222-2}, we have
		\begin{equation}\label{0222-3}
			\big|\hat{\mathbb{E}}[\phi(S^Z_{n-|\Pi_n|})]-\hat{\mathbb{E}}[\phi(S_n)]\big|\leq 2c_1l_{\phi}n^{-(1-\gamma)}\hat{\mathbb{E}}[|X_0|].
		\end{equation}
		Then we conclude from \eqref{0222-5}-\eqref{0222-3} that
		\begin{equation}\label{0222-8}
			\begin{split}
				&\ \ \ \ \Big|v_{n-|\Pi_n|,\epsilon}\Big(1+\epsilon+\frac{1}{n-|\Pi_n|},0\Big)-\hat{\mathbb{E}}[\phi(S_n)]\Big|\\
				&\leq C(c_X,c_1,c_2,\alpha,\gamma,\delta)\bigl(1+\hat{\mathbb{E}}[|X_0|]\bigr)l_{\phi}\bigl(\epsilon+n^{-1/2}+n^{-(1-\gamma)}\bigr).
			\end{split}
		\end{equation}
		Note that $S_{n-|\Pi_n|}^{\xi}$ is $\Gamma$-maximally distributed. Then by Lemma \ref{Lem: property of convolution}, 
		\begin{equation}\label{0222-6}
			\begin{split}
				&\ \ \ \ \Big|\hat{\mathbb{E}}\Bigl[v_{n-|\Pi_n|,\epsilon}\Big(\epsilon+\frac{1}{n-|\Pi_n|},S_{n-|\Pi_n|}^{\xi}\Big)\Bigr]-v(1,0)\Big|\\
				&=\Big|\hat{\mathbb{E}}\Bigl[v_{n-|\Pi_n|,\epsilon}\Big(\epsilon+\frac{1}{n-|\Pi_n|},S_{n-|\Pi_n|}^{\xi}\Big)\Bigr]-\hat{\mathbb{E}}\bigl[\phi\bigl(S_{n-|\Pi_n|}^{\xi}\bigr)\bigr]\Big|\\
				&\leq \Big|\hat{\mathbb{E}}\Bigl[v_{n-|\Pi_n|,\epsilon}\Big(\epsilon+\frac{1}{n-|\Pi_n|},S_{n-|\Pi_n|}^{\xi}\Big)\Bigr]-\hat{\mathbb{E}}\Bigl[v_{n-|\Pi_n|}\Big(\epsilon+\frac{1}{n-|\Pi_n|},S_{n-|\Pi_n|}^{\xi}\Big)\Bigr]\Big|\\
				&\ \ \ \ +\Big|\hat{\mathbb{E}}\Bigl[v_{n-|\Pi_n|}\Big(\epsilon+\frac{1}{n-|\Pi_n|},S_{n-|\Pi_n|}^{\xi}\Big)\Bigr]-\hat{\mathbb{E}}\bigl[v_{n-|\Pi_n|}\big(0,S_{n-|\Pi_n|}^{\xi}\big)\bigr]\Big|\\
				&\leq (1+2\hat{\mathbb{E}}[|X_0|])l_{\phi}\epsilon+3\hat{\mathbb{E}}[|X_0|]l_{\phi}\frac{1}{n-|\Pi_n|}.
			\end{split}
		\end{equation}
		Summarizing \eqref{0222-7}, \eqref{0222-8} and \eqref{0222-6}, we get for any $n\geq N_0$,
		\begin{equation*}
			\begin{split}
				\hat{\mathbb{E}}[\phi(S_n)]-v(1,0)&\geq -C(c_X,c_1,c_2,\alpha,\gamma,\delta,\zeta)\bigl(1+\hat{\mathbb{E}}[|X_0|^{2}]\bigr)l_{\phi}\bigl(\epsilon+n^{-1}\epsilon^{-1}+n^{-1/2}+n^{-(1-\gamma)}\bigr).
			\end{split}
		\end{equation*}
		Taking $\epsilon=n^{-1/2}$, we have
		\begin{equation}\label{0222-9}
			\begin{split}
				\hat{\mathbb{E}}[\phi(S_n)]-v(1,0)&\geq -C(c_X,c_1,c_2,\alpha,\gamma,\delta,\zeta)\bigl(1+\hat{\mathbb{E}}[|X_0|^{2}]\bigr)l_{\phi}\bigl(n^{-1/2}+n^{-(1-\gamma)}\bigr).
			\end{split}
		\end{equation}
		
		Now according to \eqref{0222-10} and \eqref{0222-9}, we obtain \eqref{Convergence rate of LLN subsequence} for $n\geq N_0$. For $n<N_0$, let $x^*$ be the maximum point of $\phi$ in $\Gamma$. We derive
		\begin{equation}\label{eq:new1117}
				|\hat{\mathbb{E}}[\phi(S_n)]-\max_{x\in \Gamma}\phi(x)|=|\hat{\mathbb{E}}[\phi(S_n)]-\phi(x^*)|\leq l_{\phi}(\hat{\mathbb{E}}[|S_n|]+|x^*|)\leq 2l_{\phi}\hat{\mathbb{E}}[|X_0|].
		\end{equation}
		Then \eqref{Convergence rate of LLN subsequence} holds for all $n\geq 1$.
	\end{proof}

  Prior to proving Theorem \ref{thm:LLN}, we establish the following lemma.
    \begin{lemma}
		Suppose that $\{X_k\}_{k\geq 1}$ is a d-dimensional $\alpha$-mixing stationary sequence on a sublinear expectation space $(\Omega,\mathcal{H},\hat{\mathbb{E}})$ for some $\alpha>0$ with $\hat{\mathbb{E}}[|X_1|^{2}]<\infty$. Let
		$$S_n=\frac{1}{n}\sum_{i=1}^nX_i, \ \text{ and } \ \Gamma_n:=\big\{\E[S_n]: \mathbb{E}\in \Theta\big\}\subset \mathbb{R}^d.$$
		Then for any $0<\gamma,\gamma'<1$, there exists $C>0$ depending only on $(\gamma,\gamma',\alpha)$ and the constant $c_X$ appearing in Definition \ref{Def:alpha-mixing} such that for any $\phi\in C_{lip}(\mathbb{R}^d)$ with Lipschitz constant $l_{\phi}> 0$  and for all $n\in \mathbb{N}$,
		\begin{equation}\label{eq:intro main1}
			\left|\hat{\mathbb{E}}\left[\phi\left(\frac{1}{n}\sum_{k=1}^nX_k\right)\right]-\max_{x\in \Gamma_{n_{\gamma}}}\phi(x) \right|\leq C\big(1+\hat{\mathbb{E}}[|X_1|^2]\big)l_{\phi}\bigl(n^{-(1-\gamma)/2}+n^{-\gamma(1-\gamma')}\bigr),
		\end{equation}
    where $n_{\gamma}:=\lfloor n^{\gamma}\rfloor$ is the greatest integer less than or equal to $n^{\gamma}$.
	\end{lemma}
	
	\begin{proof}
		For any $\phi\in C_{lip}(\mathbb{R}^d)$ with the Lipschitz constant $l_{\phi}$, any bounded closed convex subset $\Gamma\subset \mathbb{R}^d$ and any $Z\in \mathcal{H}^d$ on a sublinear expectation space $(\Omega, \mathcal{H}, \hat{\mathbb{E}})$ with $\hat{\mathbb{E}}[|Z|^{2}]<\infty$, let $v^{\Gamma}(t,x):=\sup_{y\in \Gamma}\phi(x+ty)$ and $v_n^Z(t,x)$ be defined as in \eqref{1225-2}-\eqref{1225-3}. 
		
		For any fixed $0<\gamma, \gamma'<1$, there exist $N_{\gamma,\gamma'}>0$ such that $n_{\gamma}>n_{\gamma,\gamma'}$ for all $n\geq N_{\gamma,\gamma'}$. Let
		\begin{equation*}
			Y_{n}^\gamma=\frac{1}{n_{\gamma}}\sum_{i=1}^{n_{\gamma}-n_{\gamma,\gamma'}}X_i, \ \text{ for } \ n\geq N_{\gamma,\gamma'}.
		\end{equation*}
		Then by Jensen's inequality,
		\begin{equation*}
			\hat{\mathbb{E}}[|Y_{n}^\gamma|^2]\leq\hat{\mathbb{E}}\biggl[\bigg(\frac{\sum_{i=1}^{n_{\gamma}}|X_i|}{n_{\gamma}}\bigg)^2\biggr]\leq \frac{\sum_{i=1}^{n_{\gamma}}\hat{\mathbb{E}}[|X_i|^2]}{n_{\gamma}}=\hat{\mathbb{E}}[|X_1|^2].
		\end{equation*}
		Let $\Gamma^{\gamma}_{n}:=\{\mathbb{E}[Y_{n}^\gamma]: \mathbb{E}\in \Theta\}$ and $\Gamma_{n_{\gamma}}:=\{\mathbb{E}[S_{n_{\gamma}}]: \mathbb{E}\in \Theta\}$. Then by Lemma \ref{Lem: 1226}, for all $n\geq N_{\gamma,\gamma'}$,
		\begin{equation}\label{est LLN ineq 1}
			\big|v^{\Gamma^{\gamma}_{n}}(1,0)-v^{\Gamma_{n_{\gamma}}}(1,0)\big|\leq l_{\phi}\hat{\mathbb{E}}\bigl[|S_{n_{\gamma}}-Y_{n}^\gamma|\bigr]\leq C(\gamma,\gamma')l_{\phi}n^{-\gamma(1-\gamma')}.
		\end{equation}
		Thus, we only need to estimate $|\hat{\mathbb{E}}[\phi(S_n)]-v^{\Gamma^{\gamma}_{n}}(1,0)|$. The estimation is also divided into two steps (from above and from below).
		
		\textbf{Step 1: The estimation from above}. Similar to Step 1 in the proof of Theorem \ref{Thm: LLN for subsequence}, by Corollary \ref{Thm: 1224}, Lemma \ref{Lem: song4.1} and Lemma \ref{Lem: property of convolution}, we have
		\begin{equation}\label{est LLN ineq 2}
			v^{\Gamma_{n}^{\gamma}}(1,0)-\hat{\mathbb{E}}[\phi(S_n)]\geq -C(c_X,\gamma,\gamma',\alpha,\zeta)(1+\hat{\mathbb{E}}[|X_1|^{2}])l_{\phi}\bigl(n^{-\frac{1-\gamma}{2}}+n^{-\gamma(1-\gamma')}\bigr).
		\end{equation}
		
		\textbf{Step 2: The estimation from below}. Consider an i.i.d. $\Gamma_{n}^{\gamma}$-maximally distributed sequence $\{\xi_k\}_{k\geq 1}$. Set $S_n^{\xi}=\frac{\xi_1+\cdots+\xi_n}{n}$. Similar to \eqref{0222-7} and \eqref{0222-6}, and using the fact that $\hat{\mathbb{E}}[|Y_n^{\gamma}|^2]\leq \hat{\mathbb{E}}[|X_1|^2]$, we conclude that for all $n\geq N_{\gamma,\gamma'}$,
		\begin{equation}\label{1128-1}
			\begin{split}
				&\ \ \ \ v^{Y_n^{\gamma}}_{n_{1-\gamma},\epsilon}\Big(1+\epsilon+\frac{1}{n_{1-\gamma}},0\Big)-\hat{\mathbb{E}}\Bigl[v^{Y_n^{\gamma}}_{n_{1-\gamma},\epsilon}\Big(\epsilon+\frac{1}{n_{1-\gamma}},S_{n_{1-\gamma}}^{\xi}\Big)\Bigr]\\
				&\geq -C(\gamma,\gamma',\zeta)\big(1+\hat{\mathbb{E}}[|X_1|^2]\big)l_{\phi}n^{-(1-\gamma)}\epsilon^{-1},
			\end{split}
		\end{equation}
        and
        \begin{equation}\label{1128-3}
			\begin{split}
				\Big|\hat{\mathbb{E}}\Bigl[v^{Y_n^{\gamma}}_{n_{1-\gamma},\epsilon}\Big(\epsilon+\frac{1}{n_{1-\gamma}},S_{n_{1-\gamma}}^{\xi}\Big)\Bigr]-v^{\Gamma_{n}^{\gamma}}(1,0)\Big|\leq (1+2\hat{\mathbb{E}}[|X|])l_{\phi}\epsilon+3\hat{\mathbb{E}}[|X|]l_{\phi}\frac{1}{n_{1-\gamma}}.
			\end{split}
		\end{equation}
        Now define
		\begin{equation*}
			Y_{n,k}^{\gamma}=\frac{X_{(k-1)n_{\gamma}+1}+X_{(k-1)n_{\gamma}+2}+\cdots+X_{kn_{\gamma}-n_{\gamma,\gamma'}}}{n_{\gamma}}, \quad 1\leq k\leq n_{1-\gamma}, \ \ S^Y_{n_{1-\gamma}}=\frac{1}{n_{1-\gamma}}\sum_{k=1}^{n_{1-\gamma}}Y_{n,k}^{\gamma}.
		\end{equation*}
		Then $Y_n^{\gamma}=Y_{n,1}^{\gamma}$ and $\{Y_{n,k}^{\gamma}\}_{1\leq k\leq n_{1-\gamma}}$ is a stationary sequence. Note that
        \begin{equation}
			\Big|v_{n_{1-\gamma}}^{Y_n^{\gamma}}(t,x)-v_{n_{1-\gamma}}^{Y_n^{\gamma}}(t,y)\Big|\leq l_{\phi}|x-y|,
		\end{equation}
		we derive from Lemma \ref{Lem: identical sequence} and the $\alpha$-mixing of $X$ that
		\begin{equation}\label{1227-2}
			\begin{split}
				\Big|v^{Y_n^{\gamma}}_{n_{1-\gamma}}(1,0)-\hat{\mathbb{E}}\big[\phi\bigl(S^Y_{n_{1-\gamma}}\bigr)\big]\Big|\leq \sum_{k=0}^{n_{1-\gamma}-1}c_Xl_{\phi}e^{-\alpha(n_{\gamma,\gamma'}+1)}\leq C(c_X,\gamma,\gamma',\alpha)l_{\phi}n^{-(1-\gamma)}.
			\end{split}
		\end{equation}
        Similar to \eqref{eq:1115-2} and \eqref{0222-5}, we have
        \begin{equation}\label{1227-3}
				\Big|\hat{\mathbb{E}}\big[\phi\bigl(S^Y_{n_{1-\gamma}}\bigr)\big]-\hat{\mathbb{E}}[\phi(S_n)]\Big|\leq C(\gamma,\gamma')\hat{\mathbb{E}}[|X_1|]l_{\phi}\bigl(n^{-(1-\gamma)}+n^{-\gamma(1-\gamma')}\bigr),
		\end{equation}
        and
		\begin{equation}\label{1227-1}
			\begin{split}
				&\ \ \ \ \Big|v^{Y_n^{\gamma}}_{n_{1-\gamma},\epsilon}\Big(1+\epsilon+\frac{1}{n_{1-\gamma}},0\Big)-\hat{\mathbb{E}}[\phi(S_n)]\Big|\\
				&\leq (1+2\hat{\mathbb{E}}[|X_1|])l_{\phi}\epsilon+3\hat{\mathbb{E}}[|X_1|]l_{\phi}\frac{1}{n_{1-\gamma}}+\Big|v^{Y_n^{\gamma}}_{n_{1-\gamma}}(1,0)-\hat{\mathbb{E}}[\phi(S_n)]\Big|.
			\end{split}
		\end{equation}
		Combining \eqref{1227-2}-\eqref{1227-1}, we conclude for any $n\geq N_{\gamma,\gamma'}$ and $\epsilon>0$,
		\begin{equation}\label{1128-2}
			\Big|v^{Y_n^{\gamma}}_{n_{1-\gamma},\epsilon}\Big(1+\epsilon+\frac{1}{n_{1-\gamma}},0\Big)-\hat{\mathbb{E}}[\phi(S_n)]\Big|\leq C(c_X,\gamma,\gamma',\alpha)\bigl(\epsilon+n^{-(1-\gamma)}+n^{-\gamma(1-\gamma')}\bigr).
		\end{equation}
		Then by \eqref{1128-1}, \eqref{1128-3}, \eqref{1128-2} and
		taking $\epsilon=n^{-\frac{1-\gamma}{2}}$, we have for any $n\geq N_{\gamma,\gamma'}$ and $\epsilon>0$,
		\begin{equation}\label{est LLN ineq 3}
			\hat{\mathbb{E}}[\phi(S_n)]-v^{\Gamma_{n}^{\gamma}}(1,0)\geq -C(c_X,\gamma,\gamma',\alpha,\zeta)(1+\hat{\mathbb{E}}[|X_1|^{2}])l_{\phi}\bigl(n^{-\frac{1-\gamma}{2}}+n^{-\gamma(1-\gamma')}\bigr).
		\end{equation}
		
		Now, according to \eqref{est LLN ineq 1}, \eqref{est LLN ineq 2} and \eqref{est LLN ineq 3}, we obtain \eqref{eq:intro main1} for $n\geq N_{\gamma,\gamma'}$. For $n<N_{\gamma,\gamma'}$, we can prove this using a method similar to that of \eqref{eq:new1117}.
		Thus, \eqref{eq:intro main1} holds for all $n\geq 1$.
	\end{proof}

    \begin{proof}[Proof of Theorem \ref{thm:LLN}]
        For any $k\geq 1$, by the convexity, we have that $\Gamma_{2k}\subset \Gamma_k$. Hence, $\{\Gamma_{2^{n-1}k}\}_{n\geq 1}$ is a decreasing sequence of sets, i.e.,
        \begin{equation*}
            \Gamma_{2^{n}k}\subset \Gamma_{2^{n-1}k}, \ \text{ for all } \ n\geq 1.
        \end{equation*}
        Denote
        \begin{equation*}
            \Gamma^k_*:=\lim_{n\to \infty}\Gamma_{2^{n-1}k}=\bigcap_{n\geq 1}\Gamma_{2^{n-1}k}, \ \text{ for all } \ k\geq 1.
        \end{equation*}
        Since $\Gamma_{2^{n-1}k}$ is convex and closed for all $n\geq 1$, we know that $\Gamma_*^k$ is also convex and closed.
        
        Let $S_n=\frac{1}{n}\sum_{i=1}^nX_i$ and fix $0<\delta<1/4$. Taking $\gamma'=1/2$ in \eqref{eq:intro main1} and maximizing the right side of equation \eqref{eq:intro main1} with respect to the parameter $\gamma\in [1/2-2\delta,1/2+2\delta]$, we conclude that there exists $C>0$ (which may vary between lines in the following) such that for any $\phi\in C_{lip}(\mathbb{R}^d)$ with Lipschitz constant $l_{\phi}$,
        \begin{equation}\label{eq:1117}
            \big|\hat{\E}[\phi(S_n)]-\max_{x\in \Gamma_m}\phi(x)\big|\leq Cl_{\phi}n^{-1/4+\delta}, \ \text{ for all } \ \lfloor n^{1/2-2\delta}\rfloor \leq m\leq \lfloor n^{1/2+2\delta}\rfloor.
        \end{equation}
        Thus, for any $\lfloor n^{1/2-2\delta}\rfloor\leq m_1,m_2\leq \lfloor n^{1/2+2\delta}\rfloor$, we have
        \begin{equation}\label{eq:1117-1}
            \big|\max_{x\in \Gamma_{m_1}}\phi(x)-\max_{x\in \Gamma_{m_2}}\phi(x)\big|\leq Cl_{\phi}n^{-1/4+\delta}.
        \end{equation}
        Fix $k\geq 1$. Note that for any $\phi\in C_{lip}(\mathbb{R}^d)$,
        \begin{equation*}
            \max_{x\in \Gamma_{2^{n-1}k}}\phi(x)\downarrow \max_{x\in \Gamma^k_*}\phi(x), \ \text{ as } \ n\to \infty, 
        \end{equation*}
        Then there exists $l_{k,\delta}>0$ such that for any $l\geq l_{k,\delta}$ and choosing $n=\lfloor (2^lk)^{2/(1-4\delta)}\rfloor$, we have $\lfloor n^{1/2-2\delta}\rfloor\leq2^lk<2^{l+1}k\leq \lfloor n^{1/2+2\delta}\rfloor$.
        Hence, it follows from \eqref{eq:1117-1} that
        \begin{equation}\label{eq:1117-4}
            \big|\max_{x\in \Gamma_m}\phi(x)-\max_{x\in \Gamma_{2^lk}}\phi(x)\big|\leq Cl_{\phi}(2^lk)^{-1/2}, \ \text{ for all $2^lk\leq m\leq 2^{l+1}k$.}
        \end{equation}
        Thus,
        \begin{equation}\label{eq:1117-2}
            \lim_{m\to \infty}\max_{x\in \Gamma_m}\phi(x)=\max_{x\in \Gamma_*^k}\phi(x).
        \end{equation}
       Since the left hand of \eqref{eq:1117-2} is independent of $k$,  for any $k_1,k_2\geq 1$, we conclude 
       \begin{equation}\label{eq:14.192024}
           \text{$\max_{x\in \Gamma_*^{k_1}}\phi(x)=\max_{x\in \Gamma_*^{k_2}}\phi(x)$ for all $\phi\in C_{lip}(\mathbb{R}^d)$}.
       \end{equation} Let $\phi_i(x)=\inf_{y\in \Gamma_*^{k_i}}d(x,y)$, $i=1,2$. Then $\phi_i(x)\in  C_{lip}(\mathbb{R}^d)$, $i=1,2$. Applying \eqref{eq:14.192024} to $\phi_1$ and $\phi_2$, we obtain that 
       \[0=\sup_{x\in \Gamma_*^{k_1}}\inf_{y\in \Gamma_*^{k_2}}d(x,y)=\sup_{x\in \Gamma_*^{k_2}}\inf_{y\in \Gamma_*^{k_1}}d(x,y),\]
       that is, the Hausdorff metric between two compact subsets  $\Gamma_*^{k_1}$ and $\Gamma_*^{k_2}$ is zero. The arbitrariness of $k_1,k_2\ge 1 $ ensures us to denote
        \[
        \Gamma_*:=\Gamma_*^{k}\text{ for some (and for any) }k\ge 1.
        \]
        Note that $\Gamma^k_*\subset \Gamma_k$ for all $k\geq 1$, and hence
        \begin{equation}\label{eq:1117-3}
            \bigcap_{n\geq 1}\Gamma_n\subset \Gamma_*^1=\Gamma_*\subset\bigcap_{n\geq 1}\Gamma_n.
        \end{equation}
        So $\Gamma_*=\cap_{n\geq 1}\Gamma_n$. Moreover, according to \eqref{eq:1117-4}-\eqref{eq:1117-3}, we have for all $l\geq l_{1,\delta}$,
        \begin{equation*}
                \big|\max_{x\in \Gamma_{2^l}}\phi(x)-\max_{x\in \Gamma_*}\phi(x)\big|\leq \sum_{k=1}^{\infty}\big|\max_{x\in \Gamma_{2^{l+k-1}}}\phi(x)-\max_{x\in \Gamma_{2^{l+k}}}\phi(x)\big|\leq Cl_{\phi}2^{-l/2},
        \end{equation*}
        and thus 
        \begin{equation}\label{eq:1117-6}
                \big|\max_{x\in \Gamma_{m}}\phi(x)-\max_{x\in \Gamma_*}\phi(x)\big|\leq Cl_{\phi}2^{-l/2}, \ \text{ for all } \ 2^l\leq m\leq 2^{l+1}.
        \end{equation}
        Then \eqref{eq:thm-LLN} follows from \eqref{eq:1117} and \eqref{eq:1117-6}.
    \end{proof}

 \subsection{Strong law of large numbers}
In this subsection, we prove the following SLLN, that is, Theorem \ref{Thm:SLLN}.

 We first give some notations. For any subset $\Gamma\subset \mathbb{R}^d$ and any point $\lambda=(\l_1,\ldots,\l_d)\in \mathbb{R}^d$, set
 \begin{equation*}
     \Gamma^{\lambda}:=\Big\{x^{\lambda}=\langle\lambda, x\rangle=\sum_{i=1}^d\lambda_ix_i: x\in \Gamma\Big\}.
 \end{equation*}
 Denote by $\mathbb Q^d_1:=\{\lambda\in \mathbb{R}^d: \lambda_i \text{ is rational and } |\lambda_i|\leq 1,~1\le i\le d\}$. Then $\mathbb Q^d_1$ is countable and we have the following lemma.
 \begin{lemma}\label{Lem:equiv Gamma and Gamma-lambda}
     Suppose that $\Gamma\subset \mathbb{R}^d$ is a non-empty bounded closed convex subset. Then $\Gamma$ is uniquely determinated by $\{\Gamma^{\lambda}\}_{\lambda\in \mathbb Q^d_1}$, i.e.,
     \begin{equation*}
         x\in \Gamma \text{ if and only if } x^{\lambda}\in \Gamma^{\lambda} \text{ for all } \lambda\in \mathbb Q^d_1.
     \end{equation*}
 \end{lemma}
 \begin{proof}
     ($\Rightarrow$) By the definition of $\Gamma^{\lambda}$, we know that $x^{\lambda}\in \Gamma^{\lambda}$ if $x\in \Gamma$ and $\lambda\in \mathbb Q^d_1$.

     ($\Leftarrow$) We only need to show that if $x\notin \Gamma$, there exists $\lambda\in \mathbb Q^d_1$ such that $x^{\lambda}\notin \Gamma^{\lambda}$. Since $\Gamma$ is a closed convex set and $x\notin \Gamma$, by the hyperplane separation theorem, there exixts a unit vector $v\in \mathbb{R}^d$ and $c\in \mathbb R$ such that
     \begin{equation*}
         x^{v}=\langle v, x\rangle>c>\langle v, y\rangle=y^{v}, \text{ for all } y\in \Gamma.
     \end{equation*}
     Note that $\Gamma$ is bounded and $v$ can be approximately by elements in $\mathbb Q^d_1$, and then there exists a $\lambda\in \mathbb Q^d_1$ such that
     \begin{equation*}
         x^{\lambda}>c>y^{\lambda}, \text{ for all } y\in \Gamma.
     \end{equation*}
     Hence, $x^{\lambda}\notin \Gamma^{\lambda}$.
 \end{proof}
 Now we give the proof of Theorem \ref{Thm:SLLN}.
 \begin{proof}[Proof of Theorem \ref{Thm:SLLN}]
     Let $\mathcal{P}$ be given by \eqref{Prob set for X} and $S_n=\frac{1}{n}\sum_{i=1}^nX_i$. We only need to show that for any $P\in\mathcal{P}$,
    \begin{equation}\label{Ineq:SLLN1}
         \lim_{n\to\infty}dist(S_n,\Gamma_*)=0, \ \text{$P$-a.s.}
     \end{equation}
     
     By Theorem \ref{thm:LLN} and Corollary \ref{coro:Charact of Gamma}, there exists $C>0$  such that for any $\phi\in C_{lip}(\mathbb R^d)$ with Lipschitz constant $l_{\phi}$,
		\begin{equation}\label{0605-1}
			\Big|\hat{\mathbb{E}}[\phi(S_n)]-\max_{x\in \Gamma_{*}}\phi(x) \Big|\leq Cl_{\phi}n^{-1/6}.
		\end{equation}
      
     Our proof will be divided into the cases $d=1$ and $d>1$.

     \noindent\textbf{Case $d=1$:} Notice that \eqref{Ineq:SLLN1} is equivalent to 
     \begin{equation}\label{eq:SLLN-1-dim}
        \inf_{x\in \Gamma_*}x\leq \liminf_{n\to\infty}S_{n}\leq \limsup_{n\to\infty}S_{n}\leq \sup_{x\in \Gamma_*}x, \ \text{$P$-a.s., for any $P\in\mathcal{P}$.}
     \end{equation}
     The proof borrows from the idea in  \cite[Theorem 3.1]{Song2022}. For any $m,n\in \mathbb N$, set $S^{m}_n:=\frac{X_{m+1}+\cdots+X_{m+n}}{n}$. Let $\beta_n$ be the biggest integer less than or equal to $\beta^{n}$ for some $\beta>1$ and $m_n$ be a positive integer for each $n\in\mathbb{N}$. Note that $\{X_{m+k}\}_{k\geq 1}$ is still a $\alpha$-mixing stationary sequence with the same  positive parameter $c_X$ for all $m\geq 1$, then for any $P\in \mathcal{P}$ and $\epsilon>0$, we conclude from \eqref{0605-1} by setting $\phi(x)=\sup_{y\in \Gamma_*}|x-y|$ that
     \begin{equation*}
         \sum_{n=1}^{\infty}P\bigg\{\sup_{y\in [\underline{\mu}, \overline{\mu}]}|S^{m_n}_{\beta_n}-y|>\epsilon\bigg\}\leq \frac{1}{\epsilon}\sum_{n=1}^{\infty}\hat{\mathbb E}\big[\phi\big(S^{m_n}_{\beta_n}\big)\big]\leq \frac{C}{\epsilon}\sum_{n=1}^{\infty}\beta_n^{-1/6}<\infty.
     \end{equation*}
     Hence, it follows from Borel–Cantelli lemma and the arbitrary of $\epsilon>0$ that
     \begin{equation*}
         \underline{\mu}\leq \liminf_{n\to\infty}S^{m_n}_{\beta_n}\leq \limsup_{n\to\infty}S^{m_n}_{\beta_n}\leq \overline{\mu}, \ \text{$P$-a.s.}
     \end{equation*}
     Then following ideas in \cite[STEP 1 and STEP 2 in Theorem 3.1]{Song2022}, we obtain \eqref{eq:SLLN-1-dim}.

     \noindent\textbf{Case $d>1$:} For any $\lambda\in \mathbb R^d$, set $X_k^{\lambda}:=\langle\lambda, X_k\rangle$. It is easy to check that $\{X_k^{\lambda}\}_{k\geq 1}$ is a $1$-dimensional $\alpha$-mixing stationary sequence. Moreover, choosing $\phi(x)=\varphi(\langle \lambda,x\rangle)$ for all $\varphi\in C_{lip}(\mathbb{R})$ in \eqref{0605-1}, we deduce
     \begin{equation*}
		\Big|\hat{\mathbb{E}}[\varphi(S^{\lambda}_n)]-\max_{x\in \Gamma^{\lambda}_{*}}\varphi(x) \Big|\leq Cl_{\varphi}n^{-1/6},
	\end{equation*}
    where $S_n^{\lambda}:=\frac{1}{n}\sum_{i=1}^nX_i^{\lambda}$ and $\Gamma^{\lambda}_{*}:=\{\langle \lambda,x\rangle: x\in \Gamma_*\}$. 
     By the result in Case $d=1$, for any $P\in \mathcal{P}$, we have
     \begin{equation*}
         \inf_{x\in\Gamma_*^{\lambda}}x\leq \liminf_{n\to\infty}\langle\lambda, S_n\rangle\leq \limsup_{n\to\infty}\langle\lambda, S_n\rangle\leq   \sup_{x\in\Gamma_*^{\lambda}}x, \ \text{$P$-a.s.}
     \end{equation*}
     Assume that \eqref{Ineq:SLLN1} does not hold true, then there exists a measurable set $A$ with $P(A)>0$, such that
     \begin{equation*}
         \limsup_{n\to \infty}dist(S_n,\Gamma_*)>0, \ \text{ for all } \ \omega\in A,
     \end{equation*}
     i.e., for any $\omega\in A$, there exist  $\epsilon(\omega)>0$ and a subsequence $\{n_k(\omega)\}_{k\geq 1}\subset \mathbb N$ such that 
     \begin{equation}\label{0605-3}
         dist(S_{n_k(\omega)}(\omega),\Gamma_*)>\epsilon(\omega), \ \text{ for all } \ k\geq 1.
     \end{equation}
     Since $\mathbb Q^d_1$ is countable, there exists a measurable set $A_0\subset \Omega$ with $P(A_0)=0$ such that for all $\lambda\in \mathbb Q^d_1, \omega\notin A_0$,
     \begin{equation}\label{0605-2}
         \inf_{x\in\Gamma_*^{\lambda}}x\leq \liminf_{n\to\infty}\langle\lambda, S_n(\omega)\rangle\leq \limsup_{n\to\infty}\langle\lambda, S_n(\omega)\rangle\leq \sup_{x\in\Gamma_*^{\lambda}}x.
     \end{equation}
     Now choose $\omega\in A\setminus A_0$. By \eqref{0605-2}, we conclude, by letting $\lambda=e_i$ the $i$-th standard basis of $\mathbb R^d$, that the $i$-th component of $\{S_n(\omega)\}_{n\geq 1}$ is bounded. Hence, the sequence $\{S_n(\omega)\}_{n\geq 1}$ is bounded. We choose a subsequence of $\{n_k(\omega)\}_{k\geq 1}$, which is still denoted by $\{n_k(\omega)\}_{k\geq 1}$, so that $\lim_{k\to \infty}S_{n_k(\omega)}(\omega)=x_0$. Then \eqref{0605-3} yields that $x_0\notin \Gamma_*$. By Lemma \ref{Lem:equiv Gamma and Gamma-lambda}, there exists $\lambda_0\in \mathbb Q^d_1$ such that $\langle \lambda_0,x_0\rangle\notin \Gamma_*^{\lambda_0}$. As $\Gamma_*^{\lambda_0}$ is convex, this contradicts \eqref{0605-2}, and therefore \eqref{Ineq:SLLN1} holds true.
 \end{proof}
	
	\section{Sublinear Markovian systems generated by $G$-SDEs}\label{sec:GSDE}
In this section, we review sublinear Markovian systems in a general setting and establish a condition that ensures the system is $\alpha$-mixing. We then demonstrate that sublinear Markovian systems generated by $G$-SDEs under some dissipative condition meet this criterion, thereby confirming their $\alpha$-mixing property. Consequently, all the aforementioned results can be applied to these $G$-SDEs. Finally, we prove that a broad class of iterative sequences derived from the aforementioned equations are not independent (but exhibit $\alpha$-mixing behavior). Thus, even when we shift our focus from abstract dynamical systems to the study of $G$-SDEs, our results allow the LLN and SLLN to hold within a more general framework.

	\subsection{Sublinear Markovian systems and its $\a$-mixing}\label{Sec 6.1}
	In this subsection, we recall the definition of sublinear Markovian systems, and provide a condition that guarantees their $\alpha$-mixing property.
 
	Let $(\mathbb{X},d)$ be a metric space. For any $m\geq 1$, denote by $C_{lip}(\mathbb{X}^m)$  the collection of all Lipschitz continuous functions from $\mathbb{X}^m$ to $\mathbb{R}$, where a function $h$ is Lipschitz continuous means
	\begin{equation}\label{2020-2}
		|h(\mathbf{x})-h(\mathbf{y})|\leq \sum_{i=1}^ml^i_h d(x_i,y_i), \ \text{ for all } \mathbf{x}=(x_1,\cdots,x_m), \mathbf{y}=(y_1,\cdots,y_m)\in \mathbb{X}^m.
	\end{equation}
	Here  $l^i_{h}$ is the smallest non-negative real number such that \eqref{2020-2} holds. We call $l_h^i$ the Lipschitz constant of $h$ with respect to the $i$-th component, and $l_h:=\max\{l^1_h,\cdots,l^m_h\}$ the Lipschitz constant of $h$.

Let us recall the definition of sublinear Markovian systems (see Peng \cite{Peng2005}). 
	
	\begin{definition}\label{def: markov}
		A map $T: \mathbb{R}_+\times C_{lip}(\mathbb{X})\to C_{lip}(\mathbb{X})$ is called a sublinear Markovian system if
		\begin{itemize}
			\item [(i)] for each fixed $(t,x)\in \mathbb{R}_+\times\mathbb{X}$, $T_t[\cdot](x)$ is a sublinear expectation on $C_{lip}(\mathbb{X})$;
			\item [(ii)] $T_0[\phi](x)=\phi(x)$ for each $\phi\in C_{lip}(\mathbb{X})$;
			\item [(iii)] $\{T_t\}_{t\geq 0}$ satisfies the following Chapman semigroup formula:
			\begin{equation*}
				T_t\circ T_s[\phi]=T_{t+s}[\phi], \ \text{ for all } \ \phi\in C_{lip}(\mathbb{X})\text{ and }s,t\ge0.
			\end{equation*}
		\end{itemize}
	\end{definition}
Now we give the definition of invariant sublinear expectation with respect to a sublinear Markovnian system.
	\begin{definition}
		A sublinear expectation $\tilde{T}: C_{lip}(\mathbb{X})\to \mathbb{R}$ is said to be an invariant sublinear expectation under a sublinear Markovian system $\{T_t\}_{t\geq 0}$ if $\tilde{T}(T_t[\phi])=\tilde{T}(\phi)$ for all $t\geq 0, \phi\in C_{lip}(\mathbb{X})$.
	\end{definition}
	
For a given invariant sublinear expectation $\tilde{T}$ of a sublinear Markovian system $\{T_t\}_{t\geq 0}$. Let $\Omega:=C(\mathbb{R};\mathbb{X})$  be a metric space equipped with the metric defined by
	$$\rho(\omega^1, \omega^2):=\sum_{i=1}^{\infty}2^{-i}[(\max_{t\in[-i,i]}d(\omega^1_t,\omega^2_t))\wedge 1], \ \text{ for all } \ \omega^1, \omega^2\in \Omega.$$
Define the shift $\theta: \mathbb{R}\times \Omega\to \Omega$ by $\theta_t(\omega)(s)=\omega(t+s)$ for all $t,s\in \mathbb{R}, \omega\in \Omega$. Set
	\begin{equation*}
		Lip(\Omega):=\{\xi: \xi(\omega)=\phi(\omega_{t_1},\cdots,\omega_{t_n}), \ n\geq 1, t_1<\cdots<t_n, \phi\in C_{lip}(\mathbb{X}^{n}), \omega\in \Omega\}.
	\end{equation*}
	For any fixed $\xi(\omega)=\phi(\omega_{t_1},\cdots,\omega_{t_n})\in Lip(\Omega)$ with $t_1<\cdots<t_n$, we define $\phi_k\in C_{lip}(\mathbb{X}^{n-k})$ as follows: 
	\begin{equation*}
		\phi_0:=\phi, \ \ \phi_k(x_1,\cdots,x_{n-k}):=T_{t_{n-k+1}-t_{n-k}}[\phi_{k-1}(x_1,\cdots,x_{n-k},\cdot)](x_{n-k}), \ \ 1\leq k\leq n-1.
	\end{equation*}
	Then for the given invariant sublinear expectation $\tilde{T}$, we define 
	\begin{equation*}
		\hat{\mathbb{E}}^{\tilde{T}}[\xi]=\tilde{T}(\phi_{n-1}(\cdot)), \ \text{ for any } \ \xi(\omega)=\phi(\omega_{t_1},\cdots,\omega_{t_n})\in Lip(\Omega).
	\end{equation*}
	Since $\tilde{T}$ is an invariant sublinear expectation and $T_t$ satisfies the Chapman semigroup property, we can easily prove that $\hat{\mathbb{E}}^{\tilde{T}}$ is a sublinear expectation on $Lip(\Omega)$.
	Let $U_t$ be the linear operator generated by the shift $\theta_t$, i.e., $U_t: Lip(\Omega)\mapsto Lip(\Omega)$ is given by
	\begin{equation*}
		U_t\xi(\omega)=\phi(\omega_{t+t_1},\cdots,\omega_{t+t_n}), \ \text{ for any } \ \xi(\omega)=\phi(\omega_{t_1},\cdots,\omega_{t_n})\in Lip(\Omega).
	\end{equation*}
	Then $\tilde{T}$ being invariant and $\theta$ satisfying the semigroup property imply that $U_t$ preserves the sublinear expectation $\hat{\mathbb{E}}^{\tilde{T}}$ and satisfies $U_t\circ U_s=U_{t+s}$. Now for any $p\geq 1$, we denote by $L_G^p(\Omega)$ the completion of $Lip(\Omega)$ under norm $\big(\hat{\mathbb{E}}^{\tilde{T}}[|\cdot|^p]\big)^{\frac{1}{p}}$.
	Then we have the following proposition.
	
	\begin{proposition}
		If $\tilde{T}$ is an invariant sublinear expectation of a sublinear Markovian system $\{T_t\}_{t\geq 0}$, then $(\Omega, Lip(\Omega), (U_t)_{t\geq 0}, \hat{\mathbb{E}}^{\tilde{T}})$ and $(\Omega, L^p_G(\Omega), (U_t)_{t\geq 0}, \hat{\mathbb{E}}^{\tilde{T}})$ are sublinear expectation systems.
	\end{proposition}

   The following theorem provides a sufficient condition to guarantee the $\alpha$-mixing (resp. mixing)  of $(\Omega, Lip(\Omega), (U_t)_{t\geq 0}, \hat{\mathbb{E}}^{\tilde{T}})$ (resp. $(\Omega, L^p_G(\Omega), (U_t)_{t\geq 0}, \hat{\mathbb{E}}^{\tilde{T}})$).
	
	\begin{theorem}\label{Thm: alpha-mixing system by T_t}
		Let $\tilde{T}$ be an invariant sublinear expectation of a sublinear Markovian system $\{T_t\}_{t\geq 0}$. If $T_t$, $\tilde{T}$ satisfy
		\begin{itemize}
			\item [(i)] $|T_t[\phi](x)-T_t[\phi](y)|\leq l_{\phi}d(x,y)$, for all $ t\geq 0, \ x,y\in \mathbb{X}$ and $\phi\in C_{lip}(\mathbb{X})$ with the Lipschitz constant $l_{\phi}>0$;
			\item [(ii)] there exist $c>0, \alpha>0$ and $x_0\in \mathbb{X}$ such that
			\begin{equation*}
				|T_t[\phi](x)-\tilde{T}(\phi)|\leq cl_{\phi}(1+d(x,x_0))e^{-\alpha t}, \ \text{ for all } \ t\geq 0, \ x\in \mathbb{X} \ \text{ and } \ \phi\in C_{lip}(\mathbb{X}).
			\end{equation*}
		\end{itemize}
		Then for any $m\geq 1$, $(\Omega, (Lip(\Omega))^m, (U_t)_{t\geq 0}, \hat{\mathbb{E}}^{\tilde{T}})$  is an $m$-dimensional $\alpha$-mixing sublinear expectation system and hence $(\Omega, L^p_G(\Omega), (U_t)_{t\geq 0}, \hat{\mathbb{E}}^{\tilde{T}})$ is an $m$-dimensional mixing sublinear expectation system.
	\end{theorem}
	
	\begin{proof}
		For any fixed $\xi\in (Lip(\Omega))^m$, there exist $t_1<\cdots<t_n$ and $\phi_1,\cdots,\phi_m\in C_{lip}(\mathbb{X}^n)$ such that
		\begin{equation*}
			\xi(\omega)=(\phi_k(\omega_{t_1},\cdots,\omega_{t_n}))_{k=1}^m.
		\end{equation*}
		We only need to show that there exists $C_{\xi}>0$ depending on $\xi$ such that for any $\phi\in C_{lip}(\mathbb{R}^{2m})$, and $\Lambda^1,\Lambda^2\subset \mathbb{R}_+$ with $\Lambda^1\leq \Lambda^2$, we have
		\begin{equation}\label{Ineq: alpha-mixing}
			\begin{split}
				|\hat{\mathbb{E}}^{\tilde{T}}[\phi(\bar{\xi}_{\Lambda_1}, \bar{\xi}_{\Lambda_2})] - \hat{\mathbb{E}}^{\tilde{T}}[\hat{\mathbb{E}}^{\tilde{T}}[\phi(x, \bar{\xi}_{\Lambda_2})]|_{x=\bar{\xi}_{\Lambda_1}}]| \leq C_{\xi} l_{\phi} e^{-\alpha |\Lambda^2_{min}-\Lambda^1_{max}|},
			\end{split}
		\end{equation}
		where $\bar{\xi}_{\Lambda_i}=\frac{1}{|\Lambda_i|}\sum_{u_i\in \Lambda_i}U_{u_i}\xi$, $i=1,2$.
		
		Note that
		\begin{equation}\label{2020-1}
			\begin{split}
				&\ \ \ \ \phi(\bar{\xi}_{\Lambda_1}, \bar{\xi}_{\Lambda_2})(\omega)\\
				&=\phi\Bigl(\frac{1}{|\Lambda_1|}\sum_{u_1\in \Lambda_1}(\phi_k(\omega_{t_1+u_1},\cdots,\omega_{t_n+u_1}))_{k=1}^m, \frac{1}{|\Lambda_2|}\sum_{u_2\in \Lambda_2}(\phi_k(\omega_{t_1+u_2},\cdots,\omega_{t_n+u_2}))_{k=1}^m\Bigr).
			\end{split}
		\end{equation}
		We first prove \eqref{Ineq: alpha-mixing} in the case $|\Lambda^2_{min}-\Lambda^1_{max}|>t_n-t_1$. In this case, $t_1+u_2>t_n+u_1$ for any $u_i\in \Lambda_i, \ i=1,2$. We reorder $(t_l+u_1)_{1\leq l\leq n, u_1\in \Lambda_1}$ and $(t_l+u_2)_{1\leq l\leq n, u_2\in \Lambda_2}$ respectively by
		\begin{equation*}
			(r_1,r_2,\cdots,r_{|\Lambda_1|n}) \ \text{ and } \ (s_1,s_2,\cdots,s_{|\Lambda_2|n}).
		\end{equation*}
		Then we have $r_1\leq \cdots\leq r_{|\Lambda_1|n}<s_1\leq \cdots\leq s_{|\Lambda_2|n}$ with 
		\begin{equation}\label{difference s-r}
			s_1-r_{|\Lambda_1|n}=|\Lambda^2_{min}-\Lambda^1_{max}|-(t_n-t_1).
		\end{equation}
		Let
		\begin{equation*}
			\psi(\omega_{r_1},\cdots,\omega_{r_{|\Lambda_1|n}}, \omega_{s_1},\cdots,\omega_{s_{|\Lambda_2|n}})=\phi(\bar{\xi}_{\Lambda_1}, \bar{\xi}_{\Lambda_2})(\omega)
		\end{equation*}
		Then \eqref{2020-1} implies that for any $\mathbf{x}=(x_1,\cdots,x_{|\Lambda_1|n})\in \mathbb{X}^{|\Lambda_1|n}$ and $\mathbf{y}=(y_1,\cdots,y_{|\Lambda_2|n}), \mathbf{z}=(z_1,\cdots,z_{|\Lambda_2|n}) \in \mathbb{X}^{|\Lambda_2|n}$,
		\begin{equation*}
			\begin{split}
				|\psi(\mathbf{x},\mathbf{y})-\psi(\mathbf{x},\mathbf{z})|\leq \frac{1}{|\Lambda_2|}\big(\max_{1\leq k\leq m}l_{\phi_k}\big)l_{\phi}\sum_{i=1}^{|\Lambda_2|n}d(y_i,z_i).
			\end{split}
		\end{equation*}
		This gives 
		\begin{equation*}
			l_{\psi(\mathbf{x},\cdot)}\leq \frac{1}{|\Lambda_2|}\big(\max_{1\leq k\leq m}l_{\phi_k}\big)l_{\phi} \ \text{ for all } \ \mathbf{x}\in \mathbb{X}^{|\Lambda_1|n}.
		\end{equation*}
		Notice
		\begin{equation}\label{0227-3}
			\hat{\mathbb{E}}^{\tilde{T}}[\phi(\bar{\xi}_{\Lambda_1}, \bar{\xi}_{\Lambda_2})]=\hat{\mathbb{E}}^{\tilde{T}}[\psi(\omega_{r_1},\cdots,\omega_{r_{|\Lambda_1|n}}, \omega_{s_1},\cdots,\omega_{s_{|\Lambda_2|n}})],
		\end{equation}
		and 
		\begin{equation}\label{0227-4}
			\hat{\mathbb{E}}^{\tilde{T}}[\hat{\mathbb{E}}^{\tilde{T}}[\phi(x, \bar{\xi}_{\Lambda_2})]|_{x=\bar{\xi}_{\Lambda_1}}]=\hat{\mathbb{E}}^{\tilde{T}}\Bigl[\hat{\mathbb{E}}^{\tilde{T}}[\psi(\mathbf{x}, \omega_{s_1},\cdots,\omega_{s_{|\Lambda_2|n}})]\big|_{\mathbf{x}=(\omega_{r_1},\cdots,\omega_{r_{|\Lambda_1|n}})}\Bigr].
		\end{equation}
		Now, for any $\mathbf{x}\in \mathbb{X}^{|\Lambda_1|n}$, we define $\psi_{k}(\mathbf{x},\cdot)\in C_{lip}(\mathbb{X}^{|\Lambda_2|n-k})$ by
		\begin{equation*}
			\begin{split}
				\psi_0(\mathbf{x},y_1,\cdots,y_{|\Lambda_2|n})=\psi(\mathbf{x},y_1,\cdots,y_{|\Lambda_2|n})
			\end{split}
		\end{equation*}
		and for any $1\leq k\leq |\Lambda_2|n-1$,
		\begin{equation*}
			\psi_k(\mathbf{x},y_1,\cdots,y_{|\Lambda_2|n-k})=T_{s_{|\Lambda_2|n-k+1}-s_{|\Lambda_2|n-k}}[\psi_{k-1}(\mathbf{x},y_1,\cdots,y_{|\Lambda_2|n-k},\cdot)](y_{|\Lambda_2|n-k}).
		\end{equation*}
		Let
		\begin{equation*}
			\psi_{|\Lambda_2|n}(\mathbf{x})=T_{s_1-r_{|\Lambda_1|n}}[\psi_{|\Lambda_2|n-1}(\mathbf{x},\cdot)](x_{|\Lambda_1|n}),
		\end{equation*}
		and
		\begin{equation*}
			\tilde{\psi}_{|\Lambda_2|n}(\mathbf{x})=\tilde{T}\bigl(\psi_{|\Lambda_2|n-1}(\mathbf{x},\cdot)\bigr).
		\end{equation*}
		Then we conclude
		\begin{equation}\label{0227-5}
			\hat{\mathbb{E}}^{\tilde{T}}[\psi(\omega_{r_1},\cdots,\omega_{r_{|\Lambda_1|n}}, \omega_{s_1},\cdots,\omega_{s_{|\Lambda_2|n}})]=\hat{\mathbb{E}}^{\tilde{T}}[\psi_{|\Lambda_2|n}(\omega_{r_1},\cdots,\omega_{r_{|\Lambda_1|n}})],
		\end{equation}
		and
		\begin{equation}\label{0227-6}
			\hat{\mathbb{E}}^{\tilde{T}}\Bigl[\hat{\mathbb{E}}^{\tilde{T}}[\psi(\mathbf{x}, \omega_{s_1},\cdots,\omega_{s_{|\Lambda_2|n}})]\big|_{\mathbf{x}=(\omega_{r_1},\cdots,\omega_{r_{|\Lambda_1|n}})}\Bigr]=\hat{\mathbb{E}}^{\tilde{T}}[\tilde{\psi}_{|\Lambda_2|n}(\omega_{r_1},\cdots,\omega_{r_{|\Lambda_1|n}})].
		\end{equation}
		Note that for any $1\leq k\leq |\Lambda_2|n-1$, any $\mathbf{x}\in \mathbb{X}^{|\Lambda_1|n}$ and $(y_1,\cdots,y_{|\Lambda_2|n-k}), (z_1,\cdots,z_{|\Lambda_2|n-k})\in \mathbb{X}^{|\Lambda_2|n-k}$,
		\begin{equation*}
			\begin{split}
				&\ \ \ \ |\psi_k(\mathbf{x},y_1,\cdots,y_{|\Lambda_2|n-k})-\psi_k(\mathbf{x},z_1,\cdots,z_{|\Lambda_2|n-k})|\\
				&\leq \big|T_{\Delta_k s}[\psi_{k-1}(\mathbf{x},y_1,\cdots,y_{|\Lambda_2|n-k},\cdot)](y_{|\Lambda_2|n-k})-T_{\Delta_k s}[\psi_{k-1}(\mathbf{x},y_1,\cdots,y_{|\Lambda_2|n-k},\cdot)](z_{|\Lambda_2|n-k})\big|\\
				&\ \ \ \ +\big|T_{\Delta_k s}[\psi_{k-1}(\mathbf{x},y_1,\cdots,y_{|\Lambda_2|n-k},\cdot)](z_{|\Lambda_2|n-k})-T_{\Delta_k s}[\psi_{k-1}(\mathbf{x},z_1,\cdots,z_{|\Lambda_2|n-k},\cdot)](z_{|\Lambda_2|n-k})\big|\\
				&\leq l^{|\Lambda_2|n-k+1}_{\psi_{k-1}(\mathbf{x},\cdot)}d(y_{|\Lambda_2|n-k},z_{|\Lambda_2|n-k})\\
				&\ \ \ \ +\sup_{y\in \mathbb{X}}|\psi_{k-1}(\mathbf{x},y_1,\cdots,y_{|\Lambda_2|n-k},y)-\psi_{k-1}(\mathbf{x},z_1,\cdots,z_{|\Lambda_2|n-k},y)|\\
				&\leq \sum_{i=1}^{|\Lambda_2|n-k-1}l^i_{\psi_{k-1}(\mathbf{x},\cdot)}d(y_i,z_i)+\big(l^{|\Lambda_2|n-k}_{\psi_{k-1}(\mathbf{x},\cdot)}+l^{|\Lambda_2|n-k+1}_{\psi_{k-1}(\mathbf{x},\cdot)}\big)d(y_{|\Lambda_2|n-k},z_{|\Lambda_2|n-k}),
			\end{split}
		\end{equation*}
		where $\Delta_k s:=s_{|\Lambda_2|n-k+1}-s_{|\Lambda_2|n-k}$. Hence, for any $1\leq k\leq |\Lambda_2|n-1$,
		\begin{equation}\label{0227-1}
			l^i_{\psi_k(\mathbf{x},\cdot)}\leq l^i_{\psi_{k-1}(\mathbf{x},\cdot)}, \ 1\leq i\leq |\Lambda_2|n-k-1, \ \text{ and } \ l^{|\Lambda_2|n-k}_{\psi_k(\mathbf{x},\cdot)}\leq l^{|\Lambda_2|n-k}_{\psi_{k-1}(\mathbf{x},\cdot)}+l^{|\Lambda_2|n-k+1}_{\psi_{k-1}(\mathbf{x},\cdot)}.
		\end{equation}
		Note that 
		\begin{equation}\label{0227-2}
			l^i_{\psi_0(\mathbf{x},\cdot)}=l^i_{\psi(\mathbf{x},\cdot)}\leq l_{\psi(\mathbf{x},\cdot)} \ \text{ for all } \ 1\leq i\leq |\Lambda_2|n.
		\end{equation}
		Then \eqref{0227-1} and \eqref{0227-2} imply that for any $1\leq k\leq |\Lambda_2|n-1$,
		\begin{equation*}
			l^i_{\psi_k(\mathbf{x},\cdot)}\leq l_{\psi(\mathbf{x},\cdot)},  \ 1\leq i\leq |\Lambda_2|n-k-1, \ \text{ and } \ l^{|\Lambda_2|n-k}_{\psi_k(\mathbf{x},\cdot)}\leq (k+1)l_{\psi(\mathbf{x},\cdot)}.
		\end{equation*}
		Thus,
		\begin{equation*}
			l_{\psi_{|\Lambda_2|n-1}(\mathbf{x},\cdot)}=l^1_{\psi_{|\Lambda_2|n-1}(\mathbf{x},\cdot)}\leq |\Lambda_2|n l_{\psi(\mathbf{x},\cdot)}\leq n\big(\max_{1\leq k\leq m}l_{\phi_k}\big)l_{\phi}.
		\end{equation*}
		Therefore, for any $\mathbf{x}\in \mathbb{X}$,
		\begin{equation}\label{0227-7}
			\begin{split}
				|\psi_{|\Lambda_2|n}(\mathbf{x})-\tilde{\psi}_{|\Lambda_2|n}(\mathbf{x})|&=\big|T_{s_1-r_{|\Lambda_1|n}}[\psi_{|\Lambda_2|n-1}(\mathbf{x},\cdot)](x_{|\Lambda_1|n})-\tilde{T}\bigl(\psi_{|\Lambda_2|n-1}(\mathbf{x},\cdot)\bigr)\big|\\
				&\leq cn\big(\max_{1\leq k\leq m}l_{\phi_k}\big)l_{\phi}\bigl(1+d(x_{|\Lambda_1|n},x_0)\bigr)e^{-\alpha(s_1-r_{|\Lambda_1|n})}.
			\end{split}
		\end{equation}
		Then according to \eqref{difference s-r}, \eqref{0227-3}-\eqref{0227-6} and \eqref{0227-7}, we conclude that for $|\Lambda^2_{min}-\Lambda^1_{max}|>t_n-t_1$,
		\begin{equation}\label{0227-8}
			\begin{split}
				&\ \ \ \ \big|\hat{\mathbb{E}}^{\tilde{T}}[\phi(\bar{\xi}_{\Lambda_1}, \bar{\xi}_{\Lambda_2})] - \hat{\mathbb{E}}^{\tilde{T}}[\hat{\mathbb{E}}^{\tilde{T}}[\phi(x, \bar{\xi}_{\Lambda_2})]|_{x=\bar{\xi}_{\Lambda_1}}]\big|\\
				&\leq cn\big(\max_{1\leq k\leq m}l_{\phi_k}\big)l_{\phi}\bigl(1+\tilde{T}(d(\cdot,x_0))\bigr)e^{-\alpha\big(|\Lambda^2_{min}-\Lambda^1_{max}|-(t_n-t_1)\big)}.
			\end{split}
		\end{equation}
   It is easy to see that $d(\cdot,x_0)\in C_{lip}(\mathbb{X})$, and hence $\tilde{T}(d(\cdot,x_0))<\infty$.
		Note that for $|\Lambda^2_{min}-\Lambda^1_{max}|\leq t_n-t_1$, 
		\begin{equation}\label{0227-9}
			\begin{split}
				&\ \ \ \ \big|\hat{\mathbb{E}}^{\tilde{T}}[\phi(\bar{\xi}_{\Lambda_1}, \bar{\xi}_{\Lambda_2})] - \hat{\mathbb{E}}^{\tilde{T}}[\hat{\mathbb{E}}^{\tilde{T}}[\phi(x, \bar{\xi}_{\Lambda_2})]|_{x=\bar{\xi}_{\Lambda_1}}]\big|\\
				&\leq \big|\hat{\mathbb{E}}^{\tilde{T}}[\phi(\bar{\xi}_{\Lambda_1}, \bar{\xi}_{\Lambda_2})] - \phi(0,0)\big|+\hat{\mathbb{E}}^{\tilde{T}}\big[\big|\hat{\mathbb{E}}^{\tilde{T}}[\phi(x, \bar{\xi}_{\Lambda_2})]-\phi(0,0)\big|_{x=\bar{\xi}_{\Lambda_1}}\big]\\
				&\leq 2l_{\phi}\bigl({\mathbb{E}}^{\tilde{T}}[|\bar{\xi}_{\Lambda_1}|]+{\mathbb{E}}^{\tilde{T}}[|\bar{\xi}_{\Lambda_2}|]\bigr)\\
				&\leq 4l_{\phi}{\mathbb{E}}^{\tilde{T}}[|\xi|].
			\end{split}
		\end{equation}
		By \eqref{0227-8} and \eqref{0227-9}, we obtain \eqref{Ineq: alpha-mixing} by choosing 
		\begin{equation*}
			C_{\xi}=\max\Big\{cn\big(\max_{1\leq k\leq m}l_{\phi_k}\big)\bigl(1+\tilde{T}(d(\cdot,x_0))\bigr)e^{\alpha(t_n-t_1)}, 4{\mathbb{E}}^{\tilde{T}}[|\xi|]e^{\alpha(t_n-t_1)}\Big\}<\infty.
		\end{equation*}
  The proof is completed.
	\end{proof}
\subsection{The $\a$-mixing of Markovian systems generated by $G$-SDEs}\label{subsec:application to GSDE}
In this subsection, we prove that Markovian systems generated by $G$-SDEs, under certain dissipative conditions, are $\alpha$-mixing. This ensures that the aforementioned results, including Birkhoff's ergodic theorem, LLN and SLLN, are applicable.

We will begin this section by introducing some notations and concepts. Let $G: \mathbb{S}_d\to \mathbb{R}$ be a given monotonic and sublinear function, i.e., 
    \begin{equation}\label{Set for G-BM}
       G(A)=\frac{1}{2}\sup_{q\in Q}Tr[Aq], \text{ for some bounded convex and closed subset } Q\subset \mathbb{S}_d^+
    \end{equation}
     where $\mathbb{S}_d$ is the collection of all $d\times d$ symmetric matrices, $\mathbb{S}_d^+$ is the collection of all positive semi-definite matrices in $\mathbb{S}_d$ and $B_t=\{B^i_t\}_{i=1}^d$ is the corresponding $d$-dimensional two-sided $G$-Brownian motion on a $G$-expectation space $(\hat{\Omega},L^p_G(\hat{\Omega}), (L^p_G(\hat{\Omega}_t))_{t\in \mathbb R}, \hat{\mathbb E})$ for any $p\geq 1$,  where $\hat{\Omega}:=C_0(\mathbb R; \mathbb R^d)$, $\hat{\Omega}_t:=\{\hat{\omega}\in \hat{\Omega}: \hat{\omega}_{\cdot}=\hat{\omega}'_{\cdot\wedge t}, \text{ for some } \hat{\omega}'\in \hat{\Omega}\}$ and $L^p_G(\hat{\Omega}), L^p_G(\hat{\Omega}_t)$ are given as in Example \ref{Example:G-BM}. We consider the following $G$-SDE: for any $\xi\in (L_G^p(\hat{\Omega}_s))^n$,
	\begin{equation}\label{G-SDE}
		X_t^{s,\xi}=\xi+\int_s^tb(X_r^{s,\xi})dr+\sum_{i,j=1}^d\int_s^th_{ij}(X_r^{s,\xi})d\langle B^i,B^j\rangle_r+\int_0^t\sigma(X_r^{s,\xi})dB_r, \ t\geq s,
	\end{equation}
	where $b,h_{ij}=h_{ji}: \mathbb{R}^n\to \mathbb{R}^n$ and $\sigma: \mathbb{R}^n\to \mathbb{R}^{n\times d}$ are deterministic continuous functions  and the mutual variation process $\langle B^i,B^j\rangle_t$ of $B^i, B^j$ is defined by
 \begin{equation}\label{eq:23.15}
      \langle B^i,B^j\rangle_t:{=}L_G^2-\lim_{|\pi|\to 0}\sum_{k=0}^{N-1}(B^i_{t_{k+1}}-B^i_{t_k})(B^j_{t_{k+1}}-B^j_{t_k})=B_t^iB_t^j-\int_0^tB_s^idB_s^j-\int_0^tB_s^jdB_s^i,
 \end{equation}
 where $\pi=\{0=t_0<t_1<\cdots<t_n=t\}$ is a partition of $[0,t]$ with $|\pi|:=\max_{0\leq k\leq N-1}|t_{k+1}-t_k|$.

    Under Assumption \ref{Assumption G-SDE}, according to Theorem 4.5 in \cite{Li-Lin-Lin2016} (choose the Lyapunov function $V(t,x)=|x|^p$), we know that for any $p\geq 2$ and $\xi\in (L^p_G(\hat{\Omega}_s))^n$ with starting time $s\in \mathbb R$, the equation \eqref{G-SDE} has a unique solution $X_t^{s,\xi}\in (L^p_G(\hat{\Omega}_t))^n$. Moreover, we have the following lemma.

    \begin{lemma}\label{lem:est for G-SDE}
        Suppose Assumption \ref{Assumption G-SDE} holds. Then we have the following results:
        \begin{itemize}
            \item [(i)] For any $\xi,\eta\in (L_G^2(\hat{\Omega}_s))^n$ and $t\geq s$,
            \begin{equation}\label{eq:contract of solution}
		      \hat{\mathbb{E}}[|X_t^{s,\xi}-X_t^{s,\eta}|^2]\leq e^{-2\alpha(t-s)}\hat{\mathbb{E}}[|\xi-\eta|^2].
	        \end{equation}
            \item [(ii)] For any $p\geq 2$, there exists $C_{\alpha,p}>0$ depending also on $|b(0)|,|h_{ij}(0)|$ and the bound of $\sigma$ such that 
            \begin{equation}\label{eq:p-bdd solution}
		      \hat{\mathbb{E}}[|X_t^{s,\xi}|^p]\leq C_{\alpha,p}(1+\hat{\mathbb{E}}[|\xi|^p]), \text{ for all } t\geq s, \ \xi\in (L_G^p(\hat{\Omega}_s))^n.
	        \end{equation}
        \end{itemize}
    \end{lemma}
    \begin{proof}
        \eqref{eq:contract of solution} is follows from (i) in \cite[Lemma 3.2]{Hu-Li-Wang-Zheng2015}.  Next we show that \eqref{eq:p-bdd solution} holds true for all $p\geq 2$. Applying $G$-It{\^o}'s formula to $e^{p\alpha(t-s)/2}|X_t^{s,\xi}|^p$ yields
        \begin{equation}\label{eq:0616-1}
            \begin{split}
                e^{p\alpha(t-s)/2}|X_t^{s,\xi}|^p&=|\xi|^p+\frac{p\alpha}{2}\int_s^te^{p\alpha(r-s)/2}|X_r^{s,\xi}|^pdr\\
                &\ \ \ \ +p\int_s^te^{p\alpha(r-s)/2}|X_r^{s,\xi}|^{p-2}\langle X_r^{s,\xi}, b(X_r^{s,\xi})\rangle dr
                \\
                &\ \ \ \ +\frac{p}{2}\int_s^te^{p\alpha(r-s)/2}|X_r^{s,\xi}|^{p-2}Tr[\zeta_r d\langle B\rangle_r]\\
                &\ \ \ \ +p\int_s^te^{p\alpha(r-s)/2}|X_r^{s,\xi}|^{p-2}\langle X_r^{s,\xi}, \sigma(X_r^{s,\xi})dB_r\rangle,
            \end{split}
        \end{equation}
        where $\langle B\rangle_t:=[\langle B^i,B^j\rangle_t]_{i,j=1}^d$ and
        \begin{equation*}
            \zeta_r=\frac{p-2}{|X_r^{s,\xi}|^2}\sigma^{\top}(X_r^{s,\xi})\big(X_r^{s,\xi}\cdot(X_r^{s,\xi})^{\top}\big)\sigma(X_r^{s,\xi})+\sigma^{\top}(X_r^{s,\xi})\sigma(X_r^{s,\xi})+2[\langle X_r^{s,\xi}, h_{ij}(X_r^{s,\xi})\rangle]_{i,j=1}^d.
        \end{equation*}
        Note that 
        \begin{equation}\label{eq:0616-2}
            \begin{split}
                &\ \ \ \ \int_s^te^{p\alpha(r-s)/2}|X_r^{s,\xi}|^{p-2}Tr(\zeta_r d\langle B\rangle_r)\\
            &\leq 2\int_s^te^{p\alpha(r-s)/2}|X_r^{s,\xi}|^{p-2}G(\zeta_r) dr\\
            &\leq 2\int_s^te^{p\alpha(r-s)/2}|X_r^{s,\xi}|^{p-2}\big(G(2[\langle X_r^{s,\xi}, h_{ij}(X_r^{s,\xi})\rangle]_{i,j=1}^d)+C_{p,|\sigma|_{\infty}}\big) dr,
            \end{split}
        \end{equation}
        where $|\sigma|_{\infty}:=\sup_{x\in \mathbb R^d}Tr(\sigma^{\top}(x)\sigma(x))$. On the other hand, since $G:\mathbb S_d\to \mathbb R$ is a sublinear expectation, \eqref{eq:dissipative assump} in Assumption \ref{Assumption G-SDE} gives
        \begin{equation}\label{eq:0616-3}
            \begin{split}
               &\ \ \ \ \langle x,b(x)\rangle+G(2[\langle x, h_{ij}(x)\rangle]_{i,j=1}^d)\\
               &\leq -\alpha |x|^2+\langle x,b(0)\rangle+G(2[\langle x, h_{ij}(x)\rangle]_{i,j=1}^d-(\sigma(x)-\sigma(0))^{\top}(\sigma(x)-\sigma(0)))\\
               &\leq -\alpha |x|^2+C_{|b(0)|,|h_{ij}(0)|}|x|+C_{|\sigma|_{\infty}}
               \\
               &\leq -\frac{2\alpha}{3} |x|^2+C_{\alpha,|b(0)|,|h_{ij}(0)|,|\sigma|_{\infty}}.
            \end{split}
        \end{equation}
        It follows from \eqref{eq:0616-1}-\eqref{eq:0616-3} that
        \begin{equation}\label{eq:0616-4}
            \begin{split}
                e^{p\alpha(t-s)/2}|X_t^{s,\xi}|^p&\leq|\xi|^p-\frac{p\alpha}{6}\int_s^te^{p\alpha(r-s)/2}|X_r^{s,\xi}|^pdr\\
                &\ \ \ \ +C_{\alpha,p,|b(0)|,|h_{ij}(0)|,|\sigma|_{\infty}}\int_s^te^{p\alpha(r-s)/2}|X_r^{s,\xi}|^{p-2}dr\\
                &\ \ \ \ +p\int_s^te^{p\alpha(r-s)/2}|X_r^{s,\xi}|^{p-2}\langle X_r^{s,\xi}, \sigma(X_r^{s,\xi})dB_r\rangle\\
                &\leq |\xi|^p+C_{\alpha,p,|b(0)|,|h_{ij}(0)|,|\sigma|_{\infty}}e^{p\alpha(r-s)/2}\\
                &\ \ \ \ +p\int_s^te^{p\alpha(r-s)/2}|X_r^{s,\xi}|^{p-2}\langle X_r^{s,\xi}, \sigma(X_r^{s,\xi})dB_r\rangle.
            \end{split}
        \end{equation}
        Then we derive \eqref{eq:p-bdd solution} by taking sublinear expectation $\hat{\mathbb E}$ on both side of \eqref{eq:0616-4}.
    \end{proof}

    Recall that $C_{l,lip}(\mathbb R^n)$ denotes the collection of all functions $\phi: \mathbb R^n\to \mathbb R$ satisfying there exist $L>0, k\geq 1$ such that
    \[
    |\phi(x)-\phi(y)|\leq L(1+|x|^k+|y|^k)|x-y|.
    \]
	For any $t\geq 0$, define $T_t$ by
	\begin{equation}\label{def: T_t by G-SDE}
		T_t\phi(x):=\hat{\mathbb{E}}[\phi(X_t^{0,x})], \ \text{ for all } \ x\in \mathbb{R}^n \ \text{ and  } \ \phi\in C_{l,lip}(\mathbb{R}^n).
	\end{equation}
    By \eqref{eq:contract of solution} and \eqref{eq:p-bdd solution} in Lemma \ref{lem:est for G-SDE}, 
    \begin{equation*}
        \begin{split}
            |T_t\phi(x)-T_t\phi(y)|&\leq \hat{\mathbb{E}}\big[\big|\phi\big(X_t^{0,x}\big)-\phi\big(X_t^{0,y}\big)\big|\big]\\
            &\leq L\hat{\mathbb{E}}\big[\big(1+\big|X_t^{0,x}\big|^{k}+\big|X_t^{0,y}\big|^{k}\big)\big|X_t^{0,x}-X_t^{0,y}\big|\big]\\
            &\leq 3L\Big(1+\big(\hat{\mathbb{E}}\big[\big|X_t^{0,x}\big|^{2k}\big]\big)^{\frac{1}{2}}+\big(\hat{\mathbb{E}}\big[\big|X_t^{0,y}\big|^{2k}\big]\big)^{\frac{1}{2}}\Big)\Big(\hat{\mathbb{E}}\big[\big|X_t^{0,x}-X_t^{0,y}\big|^2\big]\Big)^{\frac{1}{2}}\\
            &\leq C_{L,\alpha,k}(1+|x|^k+|y|^k)|x-y|,
        \end{split}
    \end{equation*}
     where $C_{L,\alpha,k}=3L(2\sqrt{C_{\alpha,2k}}+1)$ and $C_{\alpha,2k}$ is from \eqref{eq:p-bdd solution} for $p=2k$.
	Hence, $\{T_t\}_{t\geq 0}$ is a sublinear Markovian system on $C_{l,lip}(\mathbb{R}^n)$.

    Set
    \[
    \mathcal{M}:=\bigcap_{p\geq 2}\mathcal{M}_p, \text{ where } \mathcal{M}_p:=\big\{\{\xi_{t}\}_{t\in \mathbb R}: \xi_t\in (L_G^p(\hat{\Omega}_t))^n \text{ for all } t\in \mathbb R \text{ and } \sup_{t\in \mathbb R} \hat{\mathbb{E}}[|\xi_t|^p]<\infty\big\}.
    \]
    
    \begin{definition}
        A (two-sided) process $\{\xi_{t}\}_{t\in \mathbb R}\in \mathcal{M}_2$ is called a stationary solution of \eqref{G-SDE} if 
        \[
        \{\xi_t\}_{t\in \mathbb R} \text{ is a $n$-dimensional stationary process and } X_t^{s,\xi_s}=\xi_t, \ \text{$\hat{\mathbb E}$-q.s. for all }  t\geq s.
        \]
        If in addition, $\xi_{t}$ is continuous in $t$ $\hat{\mathbb E}$-q.s., we call $\{\xi_{t}\}_{t\in \mathbb R}$ is a continuous stationary solution of \eqref{G-SDE}.
    \end{definition}

    Set
    \[
    \mathcal{T}:=\{\text{all sublinear expectations } \tilde T \text{ on } C_{lip}(\mathbb R^n)\}.
    \]
    
    \begin{theorem}\label{Thm:stationary solution}
        Suppose Assumption \ref{Assumption G-SDE} holds. Then equation \eqref{G-SDE} has a unique continuous stationary solution $\{\xi_{t}\}_{t\in \mathbb R}$ in $\mathcal{M}_2$. Moreover, $\{\xi_{t}\}_{t\in \mathbb R}\in \mathcal{M}$ and 
        \begin{equation}\label{eq:construction of ise}
            \tilde T[\phi]:=\hat{\mathbb E}[\phi(\xi_0)], \ \text{ for all } \ \phi\in C_{lip}(\mathbb R^n)
        \end{equation}
        is the unique invariant sublinear expectation of $\{T_t\}_{t\geq 0}$ in $\mathcal{T}$.
    \end{theorem}

    \begin{proof}
        \textbf{Existence of continuous stationary solution:} For any fixed $t\in \mathbb R$, we consider the process $\{X_t^{s,0}\}_{s\leq t}$. By \eqref{eq:contract of solution} and \eqref{eq:p-bdd solution} in Lemma \ref{lem:est for G-SDE}, for any $s_1\leq s_2\leq t$,
        \begin{equation*}
                \hat{\mathbb E}[|X_t^{s_1,0}-X_t^{s_2,0}|^2]=\hat{\mathbb E}\big[\big|X_t^{s_2,X_{s_2}^{s_1,0}}-X_t^{s_2,0}\big|^2\big]\leq e^{-2\alpha(t-s_2)}\hat{\mathbb E}[|X_{s_2}^{s_1,0}|^2]\leq C_{\alpha}e^{-2\alpha(t-s_2)}.
        \end{equation*}
        Hence $\{X_t^{s,0}\}_{s\leq t}$ is a Cauchy process in $(L_G^2(\hat{\Omega}_t))^n$. Therefore, there exists $\xi_t\in (L_G^2(\hat{\Omega}_t))^n$ such that
        \begin{equation}\label{0617-1}
                \hat{\mathbb E}[|X_t^{s,0}-\xi_t|^2]\leq C_{\alpha}e^{-2\alpha(t-s)}, \ \text{ for all } \ t\geq s.
        \end{equation}
        Then for all $r\leq s\le t$, \eqref{eq:p-bdd solution} and \eqref{0617-1} give
        \begin{equation*}
                \hat{\mathbb E}[|X_t^{s,\xi_s}-\xi_t|^2]\leq 2\hat{\mathbb E}\big[\big|X_t^{s,\xi_s}-X_t^{s,X_{s}^{r,0}}\big|^2\big]+2\hat{\mathbb E}[|X_t^{r,0}-\xi_t|^2]\leq C_{\alpha}e^{-2\alpha(t-r)}.
        \end{equation*}
        Letting $r\to -\infty$, we conclude $\hat{\mathbb E}[|X_t^{s,\xi_s}-\xi_t|^2]=0$, i.e., $X_t^{s,\xi_s}=\xi_t$, $\hat{\mathbb E}$-q.s. On the other hand, by $G$-Fatou's lemma,
        \[
        \hat{\mathbb E}[|\xi_t|^p]\leq \liminf_{s\to -\infty}\hat{\mathbb E}[|X_t^{s,0}|^p]\leq C_{\alpha,p}.
        \]
        Thus, $\{\xi_t\}_{t\in \mathbb R}\in \mathcal{M}$. Since the equation \eqref{G-SDE} is time-homogeneous,  for any $(t_1,\cdots,t_n)\in \mathbb{R}^n$ and $t\in \mathbb{R}$, we know that
        \begin{equation}\label{eq:mynew1117}
            (X_{t_1}^{t_1-T,0}, \cdots, X_{t_n}^{t_n-T,0})\deq (X_{t_1+t}^{t_1+t-T,0}, \cdots, X_{t_n+t}^{t_n+t-T,0}), \ \text{ for all } \ T>0.
        \end{equation}
        Letting $T\to \infty$ in \eqref{eq:mynew1117}, together with \eqref{0617-1}, we conclude
        \begin{equation*}
            (\xi_{t_1}, \cdots, \xi_{t_n})\deq (\xi_{t_1+t}, \cdots, \xi_{t_n+t}).
        \end{equation*}
        Hence, $\{\xi_t\}_{t\in \mathbb{R}}$ is a stationary solution. 

        It remains to show that $\xi_{\cdot}$ is continuous $\hat{\E}$-q.s. Noting that $\xi_t=X_t^{s,\xi_s}$ for all $t\geq s$, 
        \begin{align}\label{eq:0617n}
                    \hat{\mathbb E}[|\xi_t-\xi_s|^4]&\leq C_d\Big(\hat{\mathbb E}\Big[\Big|\int_s^tb(\xi_r)dr\Big|^4\Big]+\sum_{i,j=1}^d\hat{\mathbb E}\Big[\Big|\int_s^th_{ij}(\xi_r)d\langle B\rangle_r\Big|^4\Big]+\hat{\mathbb E}\Big[\Big|\int_s^t\sigma(\xi_r)dB_r\Big|^4\Big]\Big)\notag\\
                    &\leq C_d\Big(|t-s|^3\int_s^t\big(1+\hat{\mathbb E}[|\xi_r|^{4\kappa}]\big)dr+|t-s|^2\Big)\\
                    &\leq C_{d,\alpha,\kappa}(1+|t-s|^2)|t-s|^2.\notag
        \end{align}
        By Kolmogorov's continuity criterion under sublinear expectation framework (see e.g. \cite[Theorem 6.1.40]{Pengbook}), $\{\xi_t\}_{t\in \mathbb R}$ admits a continuous modification, which still denotes $\{\xi_t\}_{t\in \mathbb R}$. Hence, $\{\xi_t\}_{t\in \mathbb R}$ is a continuous stationary solution.

        \noindent\textbf{Uniqueness of continuous stationary solution:} Assume there is another continuous stationary solution $\{\tilde \xi_t\}_{t\in \mathbb R}\in \mathcal{M}_2$.
        Then for any $s\le t$,
        \[
        \hat{\mathbb E}[|\xi_t-\tilde \xi_t|^2]=\hat{\mathbb E}[|X_t^{s,\xi_s}-X_t^{s,\tilde \xi_s}|^2]\leq 2e^{-2\alpha(t-s)}\sup_{r\in \mathbb R}\big(\hat{\mathbb E}[|\xi_r|^2]+\hat{\mathbb E}[|\tilde \xi_r|^2]\big).
        \]
        Letting $s\to -\infty$, we conclude that $\xi_t=\tilde \xi_t$, $\hat{\mathbb E}$-q.s. Then the continuity of $\xi, \tilde \xi$ gives the uniqueness.

        Now we only need to prove that $\tilde T$ defined in \eqref{eq:construction of ise} is the unique invariant sublinear expectation of $\{T_t\}_{t\geq 0}$ in $\mathcal{T}$. Note that for any $t\geq 0$,
        \[
        \tilde T[T_t\phi]=\hat{\mathbb E}[T_t\phi(\xi_0)]=\hat{\mathbb E}[\phi(X_t^{0,\xi_0})]=\hat{\mathbb E}[\phi(\xi_t)]=\hat{\mathbb E}[\phi(\xi_0)]=\tilde T[\phi], \ \text{ for all } \phi\in C_{lip}(\mathbb R^n).
        \]
        Hence, $\tilde T$ is an invariant sublinear expectation of $\{T_t\}_{t\geq 0}$ in $\mathcal{T}$.

        If there is another such a sublinear expectation $\tilde T'$,
        according to \eqref{eq:contract of solution} and \eqref{0617-1}, we know that for any $\phi\in C_{lip}(\mathbb R^n)$ with Lipschitz constant $l_{\phi}$,
        \begin{equation}\label{0617-2}
            |T_t\phi(x)-\tilde T[\phi]|=\big|\hat{\mathbb E}[\phi(X_0^{-t,x})]-\hat{\mathbb E}[\phi(\xi_0)]\big|\leq l_{\phi}\hat{\mathbb E}[|X_0^{-t,x}-\xi_0|]\leq C_{\alpha}l_{\phi}(1+|x|)e^{-\alpha t}.
        \end{equation}
        Then for any $\phi\in C_{lip}(\mathbb R^n)$,
        \[
        |\tilde T'[\phi]-\tilde T[\phi]|=|\tilde T'[T_t\phi]-\tilde T[\phi]|\leq \tilde T'[|T_t\phi-\tilde T[\phi]|]\leq C_{\alpha}l_{\phi}(1+\tilde T'[|\cdot|])e^{-\alpha t}.
        \]
        We derive $\tilde T'[\phi]=\tilde T[\phi]$ by letting $t\to \infty$.
    \end{proof}

    Let  $(\Omega, Lip(\Omega), (U_t)_{t\geq 0}, \hat{\mathbb{E}}^{\tilde{T}})$ and $(\Omega, L_G^p(\Omega), (U_t)_{t\geq 0}, \hat{\mathbb{E}}^{\tilde{T}})$ be sublinear expectation systems generated by the sublinear Markovian system $\{T_t\}_{t\geq 0}$ defined by \eqref{def: T_t by G-SDE} and its invariant sublinear expectation $\tilde{T}$ as in Section \ref{Sec 6.1}. We have the following result.

	\begin{theorem}\label{Thm:mixing of G-SDE}
		Suppose Assumption \ref{Assumption G-SDE} holds. Then for any $m\geq 1$, the $m$-dimensional sublinear expectation system $(\Omega, (Lip(\Omega))^m, (U_t)_{t\geq 0}, \hat{\mathbb{E}}^{\tilde{T}})$ (resp. $(\Omega, (L_G^p(\Omega))^m, (U_t)_{t\geq 0}, \hat{\mathbb{E}}^{\tilde{T}})$ for each $p\geq 1$) is $\alpha$-mixing (resp. mixing), continuous and regular.
	\end{theorem}

	\begin{proof}
        \textbf{The $\alpha$-mixing (mixing):} Note that for any $\phi\in C_{lip}(\mathbb R^n)$ with Lipschitz constant $l_{\phi}$, we have
        \begin{equation*}
            |T_t\phi(x)-T_t\phi(y)|\leq l_{\phi}\hat{\mathbb E}[|X_t^{0,x}-X_t^{0,y}|]\leq l_{\phi}|x-y|, \ \text{ for all } x,y\in \mathbb R^n,
        \end{equation*}
        which together with  \eqref{0617-2}, implies the conditions (i) and  (ii) in Theorem \ref{Thm: alpha-mixing system by T_t} hold. Thus the system $(\Omega, (Lip(\Omega))^m, (U_t)_{t\geq 0}, \hat{\mathbb{E}}^{\tilde{T}})$ (resp. $(\Omega, (L_G^p(\Omega))^m, (U_t)_{t\geq 0}, \hat{\mathbb{E}}^{\tilde{T}})$ for each $p\geq 1$) is $\alpha$-mixing (resp. mixing).
        
		\noindent\textbf{Continuity:} We only  show that $(\Omega, (L_G^1(\Omega))^m, (U_t)_{t\geq 0}, \hat{\mathbb{E}}^{\tilde{T}})$ is continuous, as $Lip(\O), L_G^p(\O)\subset L_G^1(\O)$, $p\ge 2$, and moreover, we only consider the case for $m=1$. Fix $X\in L_G^1(\Omega)$. For any $\epsilon>0$, there exists $\eta\in Lip(\Omega)$ with 
        \[
        \eta(\omega)=\phi(w_{t_1},\cdots,\omega_{t_k}), \ \text{ for some } \ k\geq 1, -\infty<t_1<\cdots<t_k<\infty,\ \phi\in C_{lip}(\mathbb R^{n\times k}),
        \]
        such that $\hat{\mathbb{E}}^{\tilde{T}}[|X-\eta|]<\epsilon$.
        By the definition of $\hat{\mathbb E}^{\tilde T}$, we have
        \begin{equation}\label{0617-3}
           \hat{\mathbb E}^{\tilde T}[\eta]=\hat{\mathbb E}[\phi(\xi_{t_1},\cdots,\xi_{t_k})],
        \end{equation}
        where $\{\xi_t\}_{t\in \mathbb R}$ is the unique continuous stationary solution in Theorem \ref{Thm:stationary solution}. 
        Then by \eqref{eq:0617n},
        \begin{equation*}
            \begin{split}
                \hat{\mathbb E}^{\tilde T}[|U_t\eta-\eta|]&=\hat{\mathbb E}[|\phi(\xi_{t+t_1},\cdots,\xi_{t+t_k})-\phi(\xi_{t_1},\cdots,\xi_{t_k})|]\\
                &\leq l_{\phi}\sum_{i=1}^k\hat{\mathbb E}[|\xi_{t+t_i}-\xi_{t_i}|]\\
                &\leq l_{\phi}\sum_{i=1}^k(\hat{\mathbb E}[|\xi_{t+t_i}-\xi_{t_i}|^4])^{\frac{1}{4}}\\
                &\leq C_{l_{\phi},k,d,\alpha,\kappa}(t+t^{\frac{1}{2}}).
            \end{split}
        \end{equation*}
        Hence
        \[
        \limsup_{t\to 0}\hat{\mathbb E}^{\tilde T}[|U_tX-X|]\leq 2\epsilon+\limsup_{t\to 0}\hat{\mathbb E}^{\tilde T}[|U_t\eta-\eta|]=2\epsilon.
        \]
        As $\epsilon>0$ is arbitrary, so
        \[
        \lim_{t\to 0}\hat{\mathbb E}^{\tilde T}[|U_tX-X|]=0.
        \]

        To prove the regularity, we need the following Kolmogorov's criterion, which is similar to \cite[Theorem (1.8) in Page 517]{Revuz-Yor1999} (see also \cite[Theorem B.2.9]{Pengbook})
    \begin{lemma}\label{lem:Kolmogorov criterion}
    Let $\{(X^n_{t})_{t\in \mathbb R}\}_{n\geq 1}$ be a sequence of $\mathbb R^d$-valued continuous processes defined on  probability spaces $(\Omega^n,\mathcal{F}^n,P^n)$ such that
    \begin{itemize}
        \item the family $\{P^n\circ (X^n_0)^{-1}\}_{n\geq 1}$ is tight \footnote{A family $\{P_i\}_{i\in I}$ of probabilities on $\mathbb{R}^d$ is said to be tight if for any $\epsilon>0$, there exists a compact subset $K$ of $\mathbb{R}^d$ such that $\sup_{i\in I}P_i(K^c)<\epsilon$.} on $\mathbb R^d$;
        \item there exist $C>0,\alpha>0,\beta>0$ such that for any $t,s\in \mathbb R$ and any $n\geq 1$
        \[
        \mathbb E_{P^n}[|X_t^n-X_s^n|^{\alpha}]\leq C|t-s|^{1+\beta}.
        \]
    \end{itemize}
    Then the family $\{P^n\circ (X^n)^{-1}\}_{n\geq 1}$ is a   relatively compact subset of the space of probability measures on $C(\mathbb{R},\mathbb{R}^d)$ with respect to the weak convergence topology.
    \end{lemma}

    Now we back to the proof of the regularity in Theorem \ref{Thm:mixing of G-SDE}.
    
    \noindent\textbf{Regularity:} We first show that $(\Omega, Lip(\Omega), \hat{\mathbb{E}}^{\tilde{T}})$ is regular. Note that the $G$-expectation space $(\hat{\Omega},L_G^1(\hat{\Omega}),\hat{\mathbb E})$ is regular. Let $\hat{\mathcal{F}}$ be the $\sigma$-algebra generated by $L_G^1(\hat{\Omega})$ and denote by $\hat{\mathcal{P}}$ the collection of all probabilities $\hat{P}$ on $(\hat{\Omega},\hat{\mathcal{F}})$ such that the linear expectation $\mathbb{E}_{\hat{P}}$ is dominated by $\hat{\mathbb E}$. By \eqref{0617-3}, we know that
        \begin{equation*}
            \hat{\mathbb{E}}^{\tilde{T}}[X]=\sup_{P\in \mathcal{P}_1}\mathbb E_P[X], \ \text{ for all } \ X\in Lip(\Omega),
        \end{equation*}
        where 
        \[
        \mathcal{P}_1:=\{\hat{P}\circ \xi_{\cdot}^{-1}: \hat{P}\in \hat{\mathcal{P}}\}.
        \]
        Note that
        \[
        \sup_{\hat{P}\in \hat{\mathcal{P}}}\mathbb E_{\hat{P}\circ \xi_{0}^{-1}}[|\cdot|]=\hat{\mathbb E}[|\xi_0|]=\tilde T[|\cdot|]<\infty.
        \]
        Hence, $\{\hat{P}\circ \xi_{0}^{-1}: \hat{P}\in \hat{\mathcal{P}}\}$ is tight. Then according to \eqref{eq:0617n} and Lemma \ref{lem:Kolmogorov criterion}, we know that $\mathcal{P}_1$ is weakly relatively compact. Let $\mathcal{P}=\overline{\mathcal{P}_1}$ the closure of  $\mathcal{P}_1$ under the topology of weak convergence. Then $\mathcal{P}$ is weakly compact.   Let 
        \[
        Lip_b(\Omega):=\{\psi(\xi): \text{for all } \psi\in C_{b,lip}(\mathbb{R}), \ \xi\in Lip(\Omega)\}.
        \]
        Then it is easy to see that $Lip_b(\Omega)\subset Lip(\Omega)\cap C_b(\Omega)$. By Lemma \ref{lem:regualr<=>compact}, we know that $(\Omega, Lip_b(\Omega), \hat{\mathbb{E}}^{\tilde{T}})$ is regular. Then we conclude from Lemma \ref{lem:extend-sub-expectation} that $(\Omega, Lip(\Omega), \hat{\mathbb{E}}^{\tilde{T}})$ is regular.

        Now we show that $(\Omega, L_G^1(\Omega), \hat{\mathbb{E}}^{\tilde{T}})$ is regular. For any $X\in Lip(\Omega)$, we know that $|X-X_n|\downarrow 0$ as $n\to \infty$ with $X_n=(X\wedge n)\vee (-n)\in Lip_b(\Omega)$ and $|X-X_n|\in Lip(\Omega)$. The regularity of $(\Omega, Lip(\Omega), \hat{\mathbb{E}}^{\tilde{T}})$ gives $\hat{\mathbb{E}}^{\tilde{T}}[|X-X_n|]\downarrow 0$. Hence, $L_G^1(\Omega)$ is also the completion of $Lip_b(\Omega)$ under $\hat{\mathbb{E}}^{\tilde{T}}[|\cdot|]$. Then \cite[Theorem 6.1.35]{Pengbook} entails that $(\Omega, L_G^1(\Omega), \hat{\mathbb{E}}^{\tilde{T}})$ is regular.
	\end{proof}

 According to the Theorem \ref{Thm:mixing of G-SDE}, we know that $(\Omega, L_G^1(\Omega), \hat{\mathbb{E}}^{\tilde{T}})$ is regular. Let $\Omega:=C(\mathbb{R};\mathbb{R}^n)$, $\mathcal{F}:=\sigma(L_G^1(\Omega))=\mathcal{B}(\Omega)$ and
 \begin{equation}\label{eq:def-P-GSDE}
     \mathcal{P}:=\{\text{all probabilities $P$ on $(\Omega,\mathcal{F})$ such that $\mathbb{E}_p$ is dominated by $\hat{\mathbb{E}}^{\tilde{T}}$}\}.
 \end{equation}
 Recall $\theta_t:\Omega\to \Omega$ is defined by $\theta_t\omega (s)=\omega(t+s)$ for all $\omega\in \Omega, t,s\in \mathbb{R}$.
 The following result indicates that there are  $\theta$-mixing probabilities in $\mathcal{P}$.

 \begin{theorem}\label{thm: infinitely many ergodic probablities in P}
     Suppose Assumption \ref{Assumption G-SDE} holds and $\mathcal{P}$ is given as in \eqref{eq:def-P-GSDE}. Then there is at least one $\theta$-mixing probability in $\mathcal{P}$.

    If the dimensions $n=d=1$ in \eqref{G-SDE} and the set $Q$ from \eqref{Set for G-BM} can be written as $Q=[\underline{\sigma}^2, \overline{\sigma}^2]$. Assume $0<\underline{\sigma}^2<\overline{\sigma}^2$, $b(x)\neq 0$ for some $x\in \mathbb{R}$ and $\inf_{x\in \mathbb{R}}\sigma^2(x)>0$. Then there are uncountably many $\theta$-mixing probabilities in $\mathcal{P}$.
 \end{theorem}

 \begin{remark}
     For general dimension $n\geq 2$, similar to the proof of 1-dimensional case, we can also prove there are many equations, whose $\mathcal{P}$  has uncountably many $\theta$-mixing probabilities. For example,
     the following gradient systems generated by $G$-Brownian motion have uncountably many $\theta$-mixing probabilities in $\mathcal{P}$:
     \[
     d X_t=-\nabla V(X_t)d t+\sigma(X_t)d B_t,
     \]
     where $Q$ is not a singleton and contains a positive definite matrix, and $\s$ is uniformly non-degenerate.
 \end{remark}

 \begin{proof}[Proof of Theorem \ref{thm: infinitely many ergodic probablities in P}]
     According to Theorem \ref{thm:uncountable many mixing measure}, for any $q\in Q$, we know that there exists a probability $\hat{P}_q$ on $(\hat{\Omega},\mathcal{B}(\hat{\Omega}))$ such that $\{B_t\}_{t\geq 0}$ is a Brownian motion under $\hat{P}_q$ with the following quadratic variance process $\langle B\rangle_t=qt$, $\hat{P}_q$-a.s. Hence, equation \eqref{G-SDE} is an SDE under $\hat{P}_q$ which can be written as the following form
     \begin{equation}\label{eq:SDE}
		X_t^{s,\xi}=\xi+\int_s^t\Big(b(X_r^{s,\xi})+Tr\big([h_{ij}(X_r^{s,\xi})]_{i,j=1}^dq\big)\Big)dr+\int_s^t\sigma(X_r^{s,\xi})dB_t, \ t\geq s\geq 0.
	\end{equation}
    Moreover, SDE \eqref{eq:SDE} is strongly dissipative since we deduce from \eqref{eq:dissipative assump} that for any $x,y\in \mathbb{R}^n$,
    \begin{equation}\label{eq:str-dissip}
        \langle x-y, b(x)+ Tr([h_{ij}(x)-h_{ij}(y)]_{i,j=1}^dq)\rangle +\frac{1}{2} Tr\big((\sigma(x)-\sigma(y))^{\top}(\sigma(x)-\sigma(y))q\big)\leq -\alpha |x-y|^2.
    \end{equation}
    Under the strongly dissipative condition \eqref{eq:str-dissip}, it is well-known that SDE \eqref{eq:SDE} has a unique mixing probability measure $\mu_q$ of the (linear) Markov transition semigroup $P^q_t: C_{lip}(\mathbb{R}^n)\to C_{lip}(\mathbb{R}^n)$ which is defined by
    \[
    P^q_t\phi (x):=\mathbb{E}_{\hat{P}_q}[\phi(X_t^{0,x})], \ t\geq 0.
    \]
    Then $\mathbb{P}^{\mu_q}$ is a $\theta$-mixing probability on $(\Omega,\mathcal{F})$ which is given by the unique  linear expectation $\mathbb{E}_{\mathbb{P}^{\mu_q}}$ in the following way: for any fixed $\xi(\omega)=\phi(\omega_{t_1},\cdots,\omega_{t_n})\in Lip(\Omega)$ with $t_1<\cdots<t_n$, 
    \begin{equation*}
		\mathbb{E}_{\mathbb{P}^{\mu_q}}[\xi]=\int_{\mathbb{R}^n}\phi_{n-1}(x_1)d\mu^q(x_1), \ \text{ for any } \ \xi(\omega)=\phi(\omega_{t_1},\cdots,\omega_{t_n})\in Lip(\Omega),
	\end{equation*}
    where $\{\phi_{k}\}_{0\leq k\leq n-1}$ is defined recursively by 
	\begin{equation*}
		\phi_0:=\phi, \ \ \phi_k(x_1,\cdots,x_{n-k}):=P^q_{t_{n-k+1}-t_{n-k}}[\phi_{k-1}(x_1,\cdots,x_{n-k},\cdot)](x_{n-k}), \ \ 1\leq k\leq n-1.
	\end{equation*}
    Note that $P^q_t\phi(x)\leq T_t\phi(x)$ for all $\phi\in C_{lip}(\mathbb{R}^n)$, $x\in \mathbb{R}^n$ and $t\geq 0$. Then by \cite[Theorem 3.3]{Hu-Li-Wang-Zheng2015},
    \[
    \int_{\mathbb{R}^n}\phi(y)d\mu^q(y)=\lim_{t\to \infty}P^q_t\phi(x)\leq \lim_{t\to \infty}T_t\phi(x)=\tilde{T}[\phi], \ \text{ for all } \ x\in \mathbb{R}^n.
    \]
    Thus, the construction of $\hat{\mathbb{E}}^{\tilde{T}}$ implies that $\mathbb{E}_{\mathbb{P}^{\mu_q}}[\xi]\leq \hat{\mathbb{E}}^{\tilde{T}}[\xi]$ for all $\xi\in Lip(\Omega)$. Hence, the $\theta$-mixing probability $\mathbb{P}^{\mu_q}$ is in $\mathcal{P}$.

    In the $1$-dimensional case, since $Q=[\underline{\sigma}^2,\overline{\sigma}^2]$ has uncountable elements, it is enough to show that $\mathbb{P}^{\mu_{q_1}}\neq \mathbb{P}^{\mu_{q_2}}$ for all $q_1,q_2\in [\underline{\sigma}^2,\overline{\sigma}^2]$ with $q_1\neq q_2$. It suffices to prove $\mu_{q_1}\neq \mu_{q_2}$ for all $q_1,q_2\in [\underline{\sigma}^2,\overline{\sigma}^2]$ with $q_1\neq q_2$. It is well known that for any $q\in [\underline{\sigma}^2,\overline{\sigma}^2]$, the unique invariant measure $\mu_q$ satisfies the following Fokker-Planck equation:
    \[
    \mathcal{L}^*\mu_q=0,
    \]
    where $\mathcal{L}^*$ is the dual operator of the following $\mathcal{L}$:
    \[
    \mathcal{L}\varphi(x)=\frac{1}{2}\sigma^2(x)q\partial^2_x \varphi(x)+\big(b(x)+h(x)q\big)\partial_x\varphi(x).
    \]
    By a simple calculation, $\mu_q$ has the following density
    \begin{equation*}
        \mu_q(dx)=\frac{1}{Z_q}\frac{1}{\sigma^2(x)}\exp\bigg\{2\int_0^x\frac{b(y)+h(y)q}{\sigma^2(y)q}d y\bigg\}dx,
    \end{equation*}
    where
    \begin{equation*}
        Z_q=\int_{\mathbb{R}}\frac{1}{\sigma^2(x)}\exp\bigg\{2\int_0^x\frac{b(y)+h(y)q}{\sigma^2(y)q}d y\bigg\}dx.
    \end{equation*}
   Since the term $b$ is not a zero function, by a simple calculation, we conclude that $\mu_{q_1}\neq \mu_{q_2}$ for all $q_1,q_2\in [\underline{\sigma}^2,\overline{\sigma}^2]$ with $q_1\neq q_2$. This shows that there are uncountable $\theta$-mixing probabilities $\{\mathbb{P}^{\mu_q}\}_{q\in [\underline{\sigma}^2,\overline{\sigma}^2]}$ in $\mathcal{P}$.
 \end{proof}

\subsection{Systems of $\a$-mixing but not independent generated by $G$-SDEs}

In this subsection, we prove that there exists a class of $\a$-mixing process generated by the above equations  that are not i.i.d.

    According to Theorem \ref{Thm:mixing of G-SDE}, we know that the canonical process  $\tilde \xi_t(\omega):=\omega_t$ for all $t\in \mathbb R$ on $(\Omega, Lip(\Omega), \hat{\mathbb E}^{\tilde T})$ is an $\alpha$-mixing process and $\{\tilde \xi_t\}_{t\in \mathbb R}$ and the stationary solution $\{\xi_t\}_{t\in \mathbb R}$ has the same distribution, i.e., for any $t_1<\cdots<t_k$,
    \begin{equation}\label{new0617-1}
        \hat{\mathbb E}^{\tilde T}[\phi(\tilde \xi_{t_1},\cdots, \tilde \xi_{t_k})]=\hat{\mathbb E}[\phi(\xi_{t_1},\cdots, \xi_{t_k})], \ \text{ for all } \ \phi\in C_{lip}(\mathbb R^{n\times k}).
    \end{equation}
    Since $\{\xi_t\}_{t\in \mathbb R}\in \mathcal{M}$, then \eqref{new0617-1} holds for all $\phi\in C_{l,lip}(\mathbb R^{n\times k})$. The following theorem demonstrates that there is a class of $\alpha$-mixing processes arising from the aforementioned equations that do not exhibit independence. 
    
    \begin{theorem}\label{Thm:not independent}
        Suppose Assumption \ref{Assumption G-SDE}  and the following statments hold:
        \begin{itemize}
            \item $b,h_{ij}$ are Lipschitz continuous;
            \item there is a positive definite $q\in Q$ where $Q$ is given in \eqref{Set for G-BM};
            \item $\sigma$ is uniformly nondegenerate, i.e., there exists  a positive definite matrix $\Sigma\in \mathbb S_d$ such that
        \[
        \sigma^{\top}(x)\sigma(x)\geq \Sigma, \text{ for all } x\in \mathbb R^n.
        \]
        \end{itemize}
        Then for any $t>s$, $\tilde \xi_t$ is NOT independent from $\tilde \xi_s$.
    \end{theorem}
    \begin{proof}
        Assume to find  a contradiction that $\tilde \xi_t$ is independent from $\tilde \xi_s$ for some $t>s$. Then \eqref{new0617-1}  shows that the stationary solution $\xi_t$ is independent from $\xi_s$. Since $\{\xi_t\}_{t\in \mathbb R}\in \mathcal{M}$, it follows that
        \begin{equation}\label{new0617-2}
            \hat{\mathbb E}[\phi(\xi_{s}, \xi_{t})]=\hat{\mathbb E}\big[\hat{\mathbb E}[\phi(x, \xi_{t})]\big|_{x=\xi_s}\big], \ \text{ for all } \ \phi\in C_{l,lip}(\mathbb R^{2n}).
        \end{equation}
        On the other hand, by the strong Markovian property of solution (see e.g. \cite[Theorem 4.2]{Hu-Ji-Liu2021}),
        \begin{equation}\label{new0617-3}
            \hat{\mathbb E}[\phi(\xi_{s}, \xi_{t})]=\hat{\mathbb E}\big[\phi(\xi_{s}, X_t^{s,\xi_s})\big]=\hat{\mathbb E}\big[\hat{\mathbb E}[\phi(x, X_t^{s,x})]\big|_{x=\xi_s}\big], \ \text{ for all } \ \phi\in C_{b,lip}(\mathbb R^{2n}).
        \end{equation}
        Fix a non-negative $\varphi\in C_{b,lip}(\mathbb R^n)$. For any non-negative $\phi_1\in C_{b,lip}(\mathbb R^n)$, letting $\phi(x,y)=\phi_1(x)\varphi(y)$ in \eqref{new0617-2} and \eqref{new0617-3} and since $(\xi_t)_{t\in \mathbb{R}}$ is a stationary solution, we have
        \begin{equation*}
            \hat{\mathbb E}[\phi_1(\xi_{s})T_{t-s}\varphi(\xi_{s})]=\hat{\mathbb E}[\phi_1(\xi_{s})\varphi(\xi_{t})]=\hat{\mathbb E}[\phi_1(\xi_{s})]\hat{\mathbb E}[\varphi(\xi_{t})]=\hat{\mathbb E}[\phi_1(\xi_{s})]\hat{\mathbb E}[T_{t-s}\varphi(\xi_{s})].
        \end{equation*}
        Since $\xi_s, \xi_t$ have the same distribution $\tilde T$ and  $\tilde T$ is invariant with respect to $\{T_t\}_{t\geq 0}$, we conclude that
        \begin{equation}\label{new0617-4}
            \tilde T[\phi_1T_{t-s}\varphi]=\tilde T[\phi_1]\tilde T[T_{t-s}\varphi], \ \text{ for all non-negative } \ \phi_1\in C_{b,lip}(\mathbb R^n).
        \end{equation}
        Let $\phi_2:=T_{t-s}\varphi$. We first prove the following equality by induction: for any $k$-polynomial $p(x)=a_kx^{k}+a_{k-1}x^{k-1}+\cdots+a_0$ with $a_k>0$, we have
        \begin{equation}\label{new0617-5}
            \tilde T[p(\phi_2)]=p(\tilde T[\phi_2]).
        \end{equation}
        It is obvious that \eqref{new0617-5} holds for $1$-polynomial $p(x)=a_1x+a_0$ with $a_1>0$. 
        Assume that \eqref{new0617-5} holds for all $k$-polynomials with positive coefficient $a_k$, then for any $p(x)=a_{k+1}x^{k+1}+a_{k}x^{k}+\cdots+a_0$ with $a_{k+1}>0$,
        \begin{equation*}
             \tilde T[p(\phi_2)]=\tilde T[\phi_2\tilde p(\phi_2)]+a_0=\tilde T[\phi_2]\tilde T[\tilde p(\phi_2)]+a_0=\tilde T[\phi_2]\tilde p(\tilde T[\phi_2])+a_0=p(\tilde T[\phi_2]),
        \end{equation*}
        where $\tilde p(x)=a_{k+1}x^{k}+a_{k}x^{k-1}+\cdots+a_1$. Then \eqref{new0617-5} holds for all polynomials with positive leading coefficient by induction. 
        
        Now we prove that \eqref{new0617-5} holds for all polynomials $p$ (may be with negative leading coefficient). Suppose that the degree of $p$ is $k$. For any $\epsilon>0$, define
        \[
        p_{\epsilon}(x)=\epsilon x^{k+1}+p(x).
        \]
        It follows that 
        \begin{equation}\label{new0617-6}
            \tilde T[p_\epsilon(\phi_2)]=p_{\epsilon}(\tilde T[\phi_2])=\epsilon (\tilde T[\phi_2])^{k+1}+p(\tilde T[\phi_2]).
        \end{equation}
        Since $\phi_2=T_{t-s}\varphi$ is bounded, i.e., $|\phi_2|\leq M$ for some $M>0$, then $|\tilde T[p_\epsilon(\phi_2)]-\tilde T[p(\phi_2)]|\leq \epsilon M^{k+1}$ and $|p_{\epsilon}(\tilde T[\phi_2])-p(\tilde T[\phi_2])|\leq \epsilon M^{k+1}$, we conclude \eqref{new0617-5} for all polynomial $p$ by letting $\epsilon\to 0$ in \eqref{new0617-6}. Since for any $h\in C(\mathbb R^n)$, Stone–Weierstrass theorem yields that there exist polynomials $\{p_n\}_{n\geq 1}$ converges to $h$ uniformly on $[0,M]$. We derive
        \begin{equation*}
            \tilde T[h(\phi_2)]=\lim_{n\to \infty}\tilde T[p_n(\phi_2)]=\lim_{n\to \infty}p_n(\tilde T[\phi_2])=h(\tilde T[\phi_2]).
        \end{equation*}
        Then for any $\varphi_1,\varphi_2\in C(\mathbb R^n)$,
        \begin{equation*}
            \tilde T[\varphi_1(\phi_2)\varphi_2(\phi_2)]=\varphi_1(\tilde T[\phi_2])\varphi_2(\tilde T[\phi_2])=\tilde T[\varphi_1(\phi_2)]\tilde T[\varphi_2(\phi_2)].
        \end{equation*}
        By Remark \ref{rem:X is constant}, we have
        \begin{equation*}
            T_{t-s}\varphi=\phi_2=\tilde T[\phi_2]=\tilde T[\varphi], \text{ $\tilde T$-q.s.}
        \end{equation*}
        Hence, for all non-negative $\varphi\in C_{b,lip}(\mathbb R^n)$,
        \begin{equation*}
            \hat{\mathbb E}[\varphi(X_{t-s}^{0,x})]=\hat{\mathbb E}[\varphi(\xi_0)], \text{ $\tilde T$-q.s. } x\in \mathbb{R}^n. 
        \end{equation*}
        Since $X_{t-s}^{0,x}, \xi_0\in L_G^p(\hat{\Omega})$ for all $p\geq 1$, choosing $\varphi(y)=|y|^p\wedge N$ and letting $N\to \infty$ give
        \begin{equation}\label{0618-1}
            \hat{\mathbb E}[|X_{t-s}^{0,x}|^p]=\hat{\mathbb E}[|\xi_0|^p]\leq C_{\alpha,p}, \text{ $\tilde T$-q.s. } x\in \mathbb{R}^n.
        \end{equation}
        As $b, h_{ij}$ are Lipschitz continuous, there exists $L_1>0,L_2>0$ such that
        \begin{equation*}
				\langle x, b(x)\rangle-G\bigl(-\sigma^{\top}(x)\sigma(x)-2[\langle x, h_{ij}(x)\rangle]_{i,j=1}^d\bigr)\geq -L_1|x|^2-L_2, \ \text{ for any } \ x\in \mathbb{R}^n.
		\end{equation*}
    Applying $G$-It{\^o}'s formula to $e^{2L_1u}|X_{u}^{0,x}|^2$ on interval $[0,t-s]$ yields
        \begin{equation}\label{0618-2}
			\hat{\mathbb E}[|X_{t-s}^{0,x}|^2]\geq e^{-2L_1(t-s)}|x|^2-\frac{L_2}{L_1}, \ \text{ for all } \ x\in \mathbb R^n.
		\end{equation}
  It follows from \eqref{0618-1} and \eqref{0618-2} that there exists $C>0$ such that $|x|\leq C$, $\tilde{T}$-q.s., i.e., $\tilde{T}((|\cdot|-C)^+)=0$.
    Note that $\xi_t$ has distribution $\tilde{T}$ for all $t\in \mathbb{R}$, and then we have
    \[
    \hat{\mathbb E}[(|\xi_t|-C)^+]=\tilde{T}((|\cdot|-C)^+)=0,
    \]
    i.e.,
    \begin{equation}\label{0618-3}
        |\xi_t|\leq C, \text{ $\hat{\mathbb E}$-q.s.}, \text{ for all } t\in \mathbb R.
    \end{equation}
    Consider the positive definite $q\in Q$ and let $P_{q}$ be given in \eqref{eq: represent G_BM by BM}. By Theorem \ref{thm:uncountable many mixing measure}, we know that $\{B_t\}_{t\geq 0}$ is a classical Brownian motion with covariance $\{q t\}_{t\geq 0}$. Note that \eqref{0618-3} implies 
    \begin{equation}\label{0618-4}
        |\xi_t|\leq C, \text{ $P_{q}$-a.s.}  \text{ for all } t\in \mathbb R.
    \end{equation}
   However, since equation \eqref{G-SDE} is uniformly nondegenerate under $P_{q}$, it is well known that the solution is irreducible, and for any $t>0$
    \[
    P_{q}(\xi_t\in \mathcal{O})=P_{q}(X_t^{0,\xi_0}\in \mathcal{O})=\int_{\mathbb R^n}P_{q}(X_t^{0,x}\in \mathcal{O})d P_{q}\circ \xi_0^{-1}(x)>0, \ \text{ for any open set } \ \mathcal{O}\subset \mathbb R^n.
    \]
    This gives a contradiction to \eqref{0618-4}. The proof is completed.
    \end{proof}

    \begin{remark}
       Let $Q$ contain a positive definite matrix $q$. Then the $G$-OU process in Remark \ref{Rem:example G-OU} satisfies all conditions in Theorem \ref{Thm:not independent}.
    \end{remark}

	\section{Applications}\label{sec:applications}
 In this section, we apply the results obtained through sublinear expectation systems to classical ergodic theory and capacity theory. Although these topics are secondary to our main objectives, we believe they offer significant independent research interest.
	\subsection{Application to classical ergodic theory}\label{application:ergodic}
	With the help of consequences in sublinear expectation, we can obtain some new results and new proofs in classical ergodic theory. Let us begin with some notations. Let $(\O,f)$ be a topological dynamical system, where $\O$ is a compact metric space and $f:\O\to\O$  is a homeomorphism. A topological dynamical system is said to be point transitive if there exists $\o\in\O$ such that $\mathrm{O}(f,\o):=\{f^n\o:n\in\mathbb{Z}\}$ is dense in $\O$.
	
It is well known that for any topological dynamical system $(\O,f)$, there exist Borel probabilities which are invariant with respect to $f$. If the invariant probability is unique, then this system is said to be uniquely ergodic.	As the first application to classical ergodic theory, we provide a characterization of the unique ergodicity of topological dynamical systems, which can be viewed as an extension of  \cite[Theorem 4.10]{Ward2011}.
	\begin{theorem}\label{thm:char. unique ergodic}
		Let $(\O,f)$ be a topological dynamical system, which is point transitive. Then the following two statements are equivalent:
		\begin{enumerate}[(i)]
			\item the system $(\O,f)$ is uniquely ergodic;
			\medskip
			\item for any $X\in C(\O)$, we have $\overline{X}\in C(\O)$.
		\end{enumerate}
	\end{theorem}
	\begin{proof}
		(i) $\Rightarrow$ (ii). It is a direct consequence of  \cite[Theorem 4.10]{Ward2011}.
		
		(ii) $\Rightarrow$ (i). We apply Theorem \ref{thm:ergodic theorem for some function} on $U_f$ induced by $f$, i.e., $U_fX:=X\circ f$, and $\mathcal{H}=C_b(\O)$ to prove this result.

		Let $\mathcal{P}$ be the set of all Borel probabilities on $\O$. Then it is weakly compact, as $\O$ is compact. Define an sublinear expectation $\hat{\mathbb{E}}$ on $C_b(\O)$ via
		\[\hat{\mathbb{E}}[X]:=\sup_{P\in\mathcal{P}}\int XdP,\text{ for any }X\in C_b(\O).\] 
		Since $\mathcal{P}$ is weakly compact, it follows from Lemma \ref{lem:regualr<=>compact} that $\hat{\mathbb{E}}$ is regular. Note that for any $\o\in\O$, $\delta_\o\in\mathcal{P}$, where $\d_\o$ is the Dirac measure at $\o$. Thus, 
		\[\hat{\mathbb{E}}[X]=\sup_{\o\in\O}X(\o).\]
		So, $L^1_{C_b(\O)}=C_b(\O)=C(\O)$. Now we prove the sublinear expectation $\hat{\E}$ is invariant and ergodic with respect to $U_f$. Indeed, for any $X\in C(\O)$, $\hat{\mathbb{E}}[X\circ f]=\sup_{\o\in\O}X(f\o)=\sup_{\o\in\O}X(\o)=\hat{\mathbb{E}}[X].$
	 Thus, $\hat{\mathbb{E}}$ is invariant. 
For any $X\in\mathcal{I}$, i.e., $X\circ f=X$, $\hat{\mathbb{E}}$-q.s., it follows that $X(f\o)=X(\o)$ for any $\o\in \O$. Since $(\O,f)$ is point transitive, there exists $\o_0\in \O$ such that $\mathrm{O}(f,\o_0)$ is dense in $\O$.  On the one hand, if $\o\in\mathrm{O}(f,\o_0)$ then $X(\o)=X(f^n\o_0)$ for some $n\in\mathbb{N}$, and hence by invariance of $X$, we have $X(\o)=X(\o_0)$. On the other hand, if $\o\notin \mathrm{O}(f,\o_0)$, there exists a strictly increasing sequence $\{n_i\}_{i=1}^\infty $ such that $\lim_{i\to\infty}f^{n_i}\o_0=\o$, which together with the continuity and invariance of $X$, implies that 
		\[X(\o)=\lim_{i\to\infty}X(f^{n_i}\o_0)=X(\o_0).\]
		That is, $X$ is constant. Therefore, $(\O,C(\O),U_f,\hat{\mathbb{E}})$ is an ergodic sublinear expectation system. 
 
For any $X\in C(\O)$, we have $-X\in C(\O)$, and hence by (ii), we obtain that 
\[\underline{X}=-\limsup_{n\to\infty}\frac{1}{n}\sum_{i=0}^{n-1}(-X)\circ f^i\in C(\O).\]
Thus, by Corollary \ref{cor:P unique},  there is a unique ergodic probability in $\mathcal{P}$. However, $\mathcal{P}$ is the set of all probabilities on $\O$, which implies that $(\O,f)$ is uniquely ergodic.
	\end{proof}
	
	The following result was obtained by Downarowicz and Weiss \cite[Theorem 4.9]{downarowicz2020all}.
	\begin{theorem}\label{thm:continuous ergodic }
		Let $(\O,f)$ be a topological dynamical system. Then the following statements are equivalent:
		\begin{enumerate}[(i)]
			\item  for each $X\in C(\O)$, the limit $\lim_{n\to\infty}\frac{1}{n}\sum_{i=0}^{n-1}X(f^i\o)$ exists for all $\o\in\O$, and belongs to $C(\O)$;
			\medskip
			\item for each $X\in C(\O)$, the limit $\lim_{n\to\infty}\frac{1}{n}\sum_{i=0}^{n-1}X(f^i\o)$ converges uniformly on $\O$.
		\end{enumerate}
	\end{theorem}
	
	Applying the Theorem \ref{thm:char. unique ergodic} to Theorem \ref{thm:continuous ergodic }, we have a little extension of Theorem \ref{thm:continuous ergodic }.
	\begin{theorem}
		Let $(\O,f)$ be a topological dynamical system. Then the following statements are equivalent:
		\begin{enumerate}[(i)]
			\item  for each $X\in C(\O)$, we have $\overline{X}\in C(\O)$;
			\medskip
			\item for each $X\in C(\O)$, the limit $\lim_{n\to\infty}\frac{1}{n}\sum_{i=0}^{n-1}X\circ f^i$ converges uniformly on $\O$.
		\end{enumerate}
	\end{theorem}
	\begin{proof}
		(i) $\Rightarrow$ (ii). Fix $X\in C(\O)$. Given $\o\in\O$, let $\O_\o=\overline{\mathrm{O}(f,\o)}$. Let $X|_{\O_\o}=Y$. Since $\O_\o$ is a $f$-invariant closed subset, it follows that $\overline{Y}=\overline{X}|_{\O_\o}\in C(\O_\o)$. Since $\O_\o$ is point transitive, by Theorem \ref{thm:char. unique ergodic}, we have $(\O_\o,f)$ is uniquely ergodic. Therefore, 
		\[\overline{X}(\o)=\overline{Y}(\o)=\underline{Y}(\o)=\underline{X}(\o).\]
		As $\o\in\O$ is arbitrary, we have $\overline{X}=\underline{X}\in C(\O)$, which together with Theorem \ref{thm:continuous ergodic }, implies (ii).
		
		(ii) $\Rightarrow$ (i). This is a direct corollary of Theorem \ref{thm:continuous ergodic }.
	\end{proof}
	
	It is well known that there are many parallels between topological dynamics and ergodic theory. In particular, ergodicity reflects the indivisibility of a measure-preserving system, corresponding to minimality reflecting the indivisibility of a topological dynamical system. Specifically, a system is called minimal if there is no non-empty, proper closed invariant subset. Now we prove that the ergodicity of a sublinear expectation system is equivalent to the minimality of a topological dynamical system, when we take the following sublinear expectation system: given a compact metric space $\O$, we take a candidate vector lattice
	\begin{equation}\label{123}
	    \mathcal{H}:=\{X-Y:X,Y:\O\to\mathbb{R}\text{ are bounded and upper semi-continuous}\},
	\end{equation}
	and the sublinear expectation \begin{equation}\label{eq:7.2621:37}
	    \hat{\mathbb{E}}[X]:=\sup_{P\in\mathcal{P}}\int XdP,\text{ for any }X\in \mathcal{H},
	\end{equation}
	where $\mathcal{P}$ is the set of all Borel probabilities on $\O$.

    \begin{proposition}
        $\mathcal{H}$ is a vector lattice.
    \end{proposition}
    \begin{proof}
    It is obvious that $\mathcal{H}$ is a linear space and $c\in \mathcal{H}$ for all $c\in \mathbb R$. Thus, we only need to prove $|X|\in \mathcal{H}$ if $X\in \mathcal{H}$. Note that $|X|=X\vee 0+(-X)\vee 0$ and $\mathcal{H}$ is a linear space, and so it suffices to prove $X\vee 0\in \mathcal{H}$ if $X\in \mathcal{H}$.
    Let
    \[
    \mathcal{H}^+:=\{X\in \mathcal{H}:X\text{ is bounded and upper semi-continuous}\}.
    \]
   Then $$\mathcal{H}=\mathcal{H}^+-\mathcal{H}^+:=\{X_1-X_2: X_1,X_2\in \mathcal{H}^+\}.$$
    We first show that $X_1\vee X_2\in \mathcal{H}^+$ if $X_1,X_2\in \mathcal{H}^+$.
    Note that $X$ is upper semi-continuous if and only if 
    \[
    \limsup_{n\to \infty}X(\omega_n)\leq X(\omega), \ \text{ for any } \ \omega\in \Omega  \text{ and } \omega_n\to \omega \text{ as } n\to \infty.
    \]
    Then if $X_1,X_2\in \mathcal{H}^+$, for any $\omega\in \Omega$ and $\omega_n\to \omega$ as $n\to \infty$,
    \[
    \limsup_{n\to \infty}(X_1\vee X_2)(\omega_n)\leq (\limsup_{n\to \infty}X_1(\omega_n))\vee (\limsup_{n\to \infty}X_2(\omega_n))\leq (X_1\vee X_2)(\omega),
    \]
    which gives $X_1\vee X_2\in \mathcal{H}^+$.

    Now we prove that $X\vee 0\in \mathcal{H}$ if $X\in \mathcal{H}$. Suppose $X=X_1-X_2$ for some $X_1,X_2\in \mathcal{H}^+$, then
    \[
    X\vee 0=(X_1-X_2)\vee (X_2-X_2)=(X_1\vee X_2)-X_2\in \mathcal{H}.
    \]
    Therefore, $\mathcal{H}$ is a vector lattice.
    \end{proof}
	
	\begin{proposition}
		Let $(\O,f)$ be a topological dynamical system. Then the following two statements are equivalent:
		\begin{enumerate}[(i)]
			\item $(\O,f)$ is minimal;
			\item $(\O,\mathcal{H},U_f,\hat{\mathbb{E}})$ is ergodic, where  $\mathcal{H}$ and $\hat{\mathbb{E}}$ are defined as in \eqref{123} and \eqref{eq:7.2621:37}, respectively.
		\end{enumerate}
	\end{proposition}
	\begin{proof}
		(i) $\Rightarrow$ (ii). Suppose that $(\O,f)$ is minimal. Given any function $X\in \mathcal{H}$ with $U_fX=X$, let $X_1,X_2$ be two bounded upper semi-continuous functions  such that $X=X_1-X_2$. By \cite[Theorem 1]{fort1951points}, $X$ has a continuous point, denoted  by $\o_0$.  Now we prove that for any $\o\in\O$, $X(\o)=X(\o_0)$. Indeed, as $(\O,f)$ is minimal, there exists an infinite sequence $\{n_i\}_{i=1}^\infty$ such that $\lim_{i\to\infty}f^{n_i}\o=\o_0$. Since $\o_0$ is the continuous point of $X$, it follows that $X(\o)=\lim_{i\to\infty}X(f^{n_i}\o)=X(\o_0)$. Thus, $X$ is constant, which shows that  $(\O,\mathcal{H},U_f,\hat{\mathbb{E}})$ is ergodic.
		
		(ii) $\Rightarrow$ (i).  By the definition of minimality, we only need to prove that for any $f$-invariant closed subset $A$, one has $A=\emptyset$ or $\O$. Indeed, as $\one_A$ is a bounded upper semi-continuous function, i.e., $\one_A\in\mathcal{H}$. Since  $(\O,\mathcal{H},U_f,\hat{\mathbb{E}})$ is ergodic, it follows that $\one_A=c$, $\hat\E$-q.s., and hence by the definition of $\hat\E$, we have $\one_A(\o)=c$ for any $\o\in\O$. It is easy to see that $c=0$ or $1$. Thus, $A$ is equal to either $\O$ or $\emptyset$. According to the arbitrariness of $A$, we finish the proof.
	\end{proof}

 \subsection{Application to capacity theory}\label{application:capacity}
In this subsection, we apply our results to capacities, and obtain Birkhoff's ergodic theorem for capacities induced by regular sublinear expectations. Let us begin with some notations. For a given measurable space $(\Omega,\mathcal{F})$, a set-valued function $\mu:\mathcal{F}\to[0,1]$ is called a capacity if
$\mu(\O)=1$, $\mu(\emptyset)=0$, and for any $A,B\in\mathcal{F}$ with $A\subset B$, $\mu(A)\le \mu(B)$. Furthermore, if a capacity satisfies that for any $A,B\in\mathcal{F}$, $\mu(A\cup B)\le \mu(A)+\mu(B)$, then it is said to be subadditive. 
 In \cite{FWZ2020}, Feng, Wu and Zhao introduced the ergodicity of capacities.   For a given measurable space $(\Omega,\mathcal{F})$ and a measurable transformation $f: (\Omega,\mathcal{F})\to (\Omega,\mathcal{F})$, we call a subadditive capacity $C$ on $(\Omega,\mathcal{F})$ is $f$-ergodic if for any $f$-invariant set $A\in \mathcal{F}$, we have $C(A)=0$ or $C(A^c)=0$. Meanwhile, we call capacity is continuous from above (resp. below) if for any $A_n,A\in\mathcal{F}$ with $A_n\downarrow A$ (resp. with $A_n\uparrow A$), then $\lim_{n\to\infty}\mu(A_n)=\mu(A)$. If it continuous from above and below, it is said to be continuous. In particular, a subadditive capacity is continuous if and only if for any $A_n\downarrow \emptyset$, as $n\to\infty$, $\lim_{n\to\infty}\mu(A_n)=0$.
 
Note that each regular sublinear expectation $\hat{\mathbb E}[X]=\sup_{P\in\mathcal{P}}\int XdP$ on $(\O,\mathcal{H})$ induces a subadditive capacity by
 \[C(A)=\sup_{P\in\mathcal{P}}P(A),\text{ for any }A\in\s(\mathcal{H}).\]
Conversely, if a  capacity $\mu$ on a measurable space $(\O,\mathcal{\mathcal{F}})$ with the form $\mu(A)=\sup_{P\in\mathcal{P}}P(A)$ for any $A\in\mathcal{F}$, then it induces an sublinear expectation on $B_b(\O,\mathcal{F})$,
\[\hat{\mathbb E}[X]=\sup_{P\in\mathcal{P}}\int XdP, \text{ for any }X\in B_b(\O,\mathcal{F}).\]
In this case, $\mu$ is continuous if and only if $\hat{\mathbb E}$ is regular on $B_b(\O,\mathcal{F})$.
 
 The authors in \cite{FWZ2020} proved that if a  capacity $\mu$  with the form $\mu=\sup_{P\in\mathcal{P}}P$ is continuous, then $\mu$ is ergodic if and only if for any $X\in B_b(\O,\mathcal{F})$, there exists a constant $a_X\in\mathbb{R}$ such that 
 \[\mu(\{\o\in\O:\lim_{n\to\infty}\frac{1}{n}\sum_{i=0}^{n-1}X(f^i\o)=a_X\}^c)=0.\]
 Under the same condition, Feng, Huang, Liu and Zhao \cite{FHLZ2023} proved that there exists an ergodic probability $P$ such that for any $X\in B_b(\O,\mathcal{F})$, $a_X=\int XdP$. By the argument above, this result can be rewritten under the sublinear expectation setting as follows. Let $\hat{\mathbb E}$ be a regular sublinear expectation on $(\O,B_b(\O,\mathcal{F}))$. Then $\hat{\mathbb E}$ is ergodic if and only if  there exists an ergodic probability $P$ on $\mathcal{F}$ such that for any $X\in B_b(\O,\mathcal{F})$,
 \[\lim_{n\to\infty}\frac{1}{n}\sum_{i=0}^{n-1}X(f^i\o)=\int XdP, \text{ }\hat{\mathbb E}\text{-q.s.}\]
Recall that the capacity induced by regular sublinear expectation in Example \ref{ex:not continuous} is not continuous, i.e., not regular on the vector lattice $B_b(\O,\mathcal F)$, but the Birkhoff's ergodic theorem holds for this capacity. 

More generally, the following result shows that the Birkhoff's ergodic theorem holds for all capacities generated by regular sublinear expectation on any vector lattice.
 \begin{theorem}\label{thm:ergodic theorem for capacity}
     Let $(\O,\mathcal{H},U_f,\hat{\mathbb E})$ be a regular sublinear expectation system, where $U_f$ is induced by a measurable transform $f:\O\to\O$. Let $C$ be the capacity induced by $\hat{\mathbb E}$ on $(\O,\s(\mathcal{H}))$. Then the following two statements are equivalent:
     \begin{enumerate}[(i)]
         \item $C$ is ergodic;
         \item there exists an ergodic probability $P\in\mathcal{P}$ such that for any $X\in L^1(\O,\mathcal{\s(\mathcal{H})},P)$, 
 \[C\bigg(\Big\{\o\in\O:\lim_{n\to\infty}\frac{1}{n}\sum_{i=0}^{n-1}X(f^i\o)=\int XdP\Big\}^c\bigg)=0.\]
     \end{enumerate}
  When the equivalent statements (i) and (ii) hold, one has $\mathcal{P}\cap \mathcal{M}(T)=\mathcal{P}\cap \mathcal{M}^e(T)=\{P\}$. 
 \end{theorem}
 \begin{proof}
    (i) $\Rightarrow$ (ii). Since $\hat{\mathbb E}$ is regular, it follows from Proposition \ref{re:existence of invariant measure} that there exists an invariant probability $P\in\mathcal{P}$. Furthermore, $P$ is ergodic. Indeed, for any $A\in\mathcal{I}$, we have that $C(A)=0$ or $C(A^c)=0$, as $C$ is ergodic. Thus, $P(A)=0$ or $P(A^c)=0$, which implies that $P$ is ergodic. 

  Let $P\in \mathcal{P}$ be the ergodic probability given above and
  \begin{equation}\label{eq:1118}
      D_X=\Big\{\o\in\O:\lim_{n\to\infty}\frac{1}{n}\sum_{i=0}^{n-1}X(f^i\o)=\int XdP\Big\}, \ \text{ for any $X\in L^1(\O,\mathcal{\s(\mathcal{H})},P)$.}
  \end{equation}
  Now we only need to prove $C(D_X^c)=0$ for all $X\in L^1(\O,\mathcal{\s(\mathcal{H})},P)$.
 By contradiction, we may assume that there exists $X\in L^1(\O,\mathcal{\s(\mathcal{H})},P)$ such that $C(D_X^c)>0$.
 Since $C$ is ergodic and $D_X\in\mathcal{I}$, we have that $C(D_X)=0$. Hence, $P(D_X)\leq C(D_X)=0$ as $P\in\mathcal{P}$.
    However, by Theorem \ref{thm:Birkhoff for measures}, we know that $P(D_X)=1$, which is a contradiction. Thus, the statement (ii) holds.

      (ii) $\Rightarrow$ (i). For any $A\in\s(\mathcal{H})$ with $f^{-1}A=A$, by (ii), we have 
      \[C(\{\o\in\O:\one_A(\o)=P(A)\}^c)=0.\]
      Since $P$ is ergodic, we know that $P(A)\in\{0,1\}$. If $P(A)=0$, then $C(A)=0$; if $P(A)=1$, then $C(A^c)=0$. Thus, $C$ is ergodic. 

      Now we prove the uniqueness. If there exists another invariant probability $P'\in\mathcal{P}$, then by (ii), we have $P'(D_X^c)=0$ for all $X\in L^1(\O,\mathcal{\s(\mathcal{H})},P)$. This together with Theorem \ref{thm:Birkhoff for measures} implies that $P=P'$.
 \end{proof}
 According to Theorem \ref{thm:uncountable many mixing measure}, we know that the capacity induced by the sublinear expectation generated by a $G$-Brownian motion is not continuous, as its $\mathcal{P}$ has uncountablely many ergodic probabilities. Since we have proven that the  sublinear expectation generated by a $G$-Brownian motion is ergodic. Thus, we know that the ergodicity of an sublinear expectation cannot imply the ergodicity of its corresponding capacity. However, the converse is true.
 \begin{theorem}
     Let $(\O,\mathcal{H},U_f,\hat{\mathbb E})$ be a sublinear expectation system, where $U_f$ is induced by a measurable transform $f:\O\to\O$. Let $C$ be the capacity induced by $\hat{\mathbb E}$ on $(\O,\s(\mathcal{H}))$. If $C$ is ergodic, then $\hat{\mathbb E}$ is ergodic.
 \end{theorem}
 \begin{proof}
Since $C$ is ergodic, by Theorem \ref{thm:ergodic theorem for capacity}, there exists an ergodic probability $P$ in $\mathcal{P}$ such that for any $X\in L^1(\O,\mathcal{\s(\mathcal{H})},P)$, 
 \[C\bigg(\Big\{\o\in\O:\lim_{n\to\infty}\frac{1}{n}\sum_{i=0}^{n-1}X(f^i\o)=\int XdP\Big\}^c\bigg)=0.\]
 Note that for any $X\in\mathcal{H}$, $X\in  L^1(\O,\mathcal{\s(\mathcal{H})},P)$, as $\int|X|dP\le \hat{\mathbb E}[|X|]<\infty.$ Thus, for any $X\in\mathcal{I}$, 
 \[C\bigg(\Big\{\o\in\O:X=\int XdP\Big\}^c\bigg)=0,\]
 which implies that $X$ is constant, $\hat{\mathbb E}$-q.s. Thus, $\hat{\mathbb E}$ is ergodic.  
 \end{proof}
 
\begin{appendices}
\section{Mixing of Brownian motion}
\label{Sec: appendix}
We prove the mixing of systems generated by the classical Brownian motion, that is, Theorem \ref{Thm: strong-mixing-of-BM}. 
\begin{proof}[Proof of Theorem \ref{Thm: strong-mixing-of-BM}]
    In the case that $\mathbb{T}=\mathbb{R}_+$, we can extend the system  $(\Omega,\mathcal{F}, (\theta_t)_{t\geq 0}, P)$ to a   system  $(\tilde \Omega,\tilde{\mathcal{F}}, (\tilde{\theta}_t)_{t\geq 0}, \tilde P)$ generated by a two-sided Brownian motion $(\tilde{B}_t)_{t\in \mathbb R}$ with the same covariance matrix $\Sigma$. We consider the map $I:\tilde \Omega\to \Omega$ by
    \[
    I((\tilde{\omega}_t)_{t\in \mathbb R})=(\tilde{\omega}_t)_{t\geq 0}.
    \]
    Then $I: (\tilde \Omega,\tilde{\mathcal{F}})\to (\Omega,\mathcal{F})$ is measurable and
    \[
    I\circ \tilde{\theta}_t=\theta_t\circ I, \ \text{ for all }  \ t\geq 0, \ \text{ and } \ P=\tilde P\circ I^{-1}.
    \]
    Then the strongly mixing of $(\tilde \Omega,\tilde{\mathcal{F}}, (\tilde{\theta}_t)_{t\geq 0}, \tilde P)$ implies that of $(\Omega,\mathcal{F}, (\theta_t)_{t\geq 0}, P)$. Hence, we only need to prove the two-sided Brownian motion case $\mathbb{T}=\mathbb{R}$.  To prove the strongly mixing of $(\Omega,\mathcal{F}, (\theta_t)_{t\geq 0}, P)$, it is equivalent to show that 
	\begin{equation}
		\label{ergodic equivalent def}
		\lim_{t\to \infty}\mathbb E[X\circ \theta_t\cdot Y]=\mathbb E[X]\mathbb E[Y], \ \text{ for all } \ X,Y\in L^2(\Omega,\mathcal{F},P).
	\end{equation}
    Let
    \[
    \mathcal{F}^{t}_{-\infty}:=\sigma(B_u-B_v: u,v\leq t), \ \ \mathcal{F}_{t}^{\infty}:=\sigma(B_u-B_v: u,v\geq t), \ \text{ for any } \ t\in \mathbb R.
    \]
   For any fixed $X,Y\in L^2(\Omega,\mathcal{F},P)$, denote
	$$X_m=\mathbb E[X|\mathcal{F}_{-m}^{\infty}], \ \ Y_m=\mathbb E[Y|\mathcal{F}_{-\infty}^m], \ \ m\geq 1.$$
	Since $\mathcal{F}_{-m}^{\infty} \uparrow\mathcal{F}$ and $ \mathcal{F}_{-\infty}^m \uparrow\mathcal{F}$ as $m \to \infty$ it follows that  $\big(X_m\big)_{m\geq 1}, \big(Y_m\big)_{m\geq 1}$ are martingales with respect to filtrations $\big(\mathcal{F}_{-m}^{\infty}\big)_{m\geq 1}, \big(\mathcal{F}_{-\infty}^m\big)_{m\geq 1}$ respectively. Note that
    \begin{equation*}
        \sup_{m\geq 1}\mathbb E[|X_m|^2]\leq \mathbb E[|X|^2]<\infty, \ \text{ and } \ \sup_{m\geq 1}\mathbb E[|Y_m|^2]\leq \mathbb E[|Y|^2]<\infty.
    \end{equation*}
    Then the martingale convergence theorem gives
    \begin{equation*}
        \lim_{m\to \infty}\|X-X_m\|_{L^2}=\lim_{m\to \infty}\|Y-Y_m\|_{L^2}=0.
    \end{equation*}
    Notice for any $m\geq 1$,
    \begin{equation*}
        \begin{split}
            \mathbb E[X\circ \theta_t\cdot Y]&=\mathbb E[X_m\circ \theta_t\cdot Y_m]+\mathbb E[(X-X_m)\circ \theta_t\cdot Y_m]+\mathbb E[X\circ \theta_t\cdot (Y-Y_m)],
        \end{split}
    \end{equation*}
    and
    \begin{equation*}
        \mathbb E[X]\mathbb E[Y]=\mathbb E[X_m]\mathbb E[Y_m]+\mathbb E[X-X_m]\mathbb E[Y_m]+\mathbb E[X]\mathbb E[Y-Y_m].
    \end{equation*}
    Since $\theta_t$ preserves $P$, then we have
    \begin{equation*}
		\begin{split}
			\big|\mathbb E[X\circ \theta_t\cdot Y]-\mathbb E[X]\mathbb E[Y]\big|&\leq \big|\mathbb E[X_m\circ \theta_t\cdot Y_m]-\mathbb E[X_m]\mathbb E[Y_m]\big|\\
			&\ \ \ \ +2\big(\|X-X_m\|_{L^2}\|Y\|_{L^2}+\|Y-Y_m\|_{L^2}\|X\|_{L^2}\big).
		\end{split}
	\end{equation*}
	Note that $X_m\circ\theta_t\in (\theta_t)^{-1}\mathcal{F}_{-m}^{\infty}=\mathcal{F}_{-m+t}^{\infty}, \ Y_m\in \mathcal{F}_{-\infty}^m$, then for all $t\geq 2m$, $X_m\circ\theta_t$ and $Y_m$ are independent, i.e., 
	$$\mathbb E[X_m\circ\theta_t\cdot Y_m]=\mathbb E[X_m\circ\theta_t]\mathbb E[Y_m]=\mathbb E[X_m]\mathbb E[Y_m], \text{ for all } t\geq 2m.$$
	Hence for any $m\geq 1$,
	\begin{equation*}
		\begin{split}
			\limsup_{t\to \infty}\big|\mathbb E[X\circ \theta_t\cdot Y]-\mathbb E[X]\mathbb E[Y]\big|\leq 2\big(\|X-X_m\|_{L^2}\|Y\|_{L^2}+\|Y-Y_m\|_{L^2}\|X\|_{L^2}\big).
		\end{split}
	\end{equation*}
	Therefore,
    \begin{equation*}
		\begin{split}
			\limsup_{t\to \infty}\big|\mathbb E[X\circ \theta_t\cdot Y]-\mathbb E[X]\mathbb E[Y]\big|\leq \liminf_{m\to \infty}2\big(\|X-X_m\|_{L^2}\|Y\|_{L^2}+\|Y-Y_m\|_{L^2}\|X\|_{L^2}\big)=0.
		\end{split}
	\end{equation*}
    This proves \eqref{ergodic equivalent def}.
\end{proof}

\begin{remark}
    $(\Omega,\mathcal{F}, (\theta_t)_{t\geq 0}, P)$ is  mixing implies $(\Omega,\mathcal{F}, \theta_{\tau}, P)$ is   mixing for any $\tau>0$. The ergodicity of $(\Omega,\mathcal{F}, (\theta_t)_{t\geq 0}, P)$ and $(\Omega,\mathcal{F}, \theta_{\tau}, P)$ for any $\tau>0$ follows immediately.
\end{remark}
\end{appendices}

	\addtolength{\itemsep}{-1.5 em} 
	\setlength{\itemsep}{-3pt}
	\footnotesize
	\addcontentsline{toc}{section}{References}
	\bibliographystyle{siam}
	\bibliography{ErgodicGBM}
\end{document}